\documentclass[11pt,reqno]{amsart}
\usepackage{amsthm}
\usepackage{amssymb}
\usepackage{latexsym}
\usepackage{multicol}
\usepackage{verbatim,enumerate}
\usepackage[usenames]{color}
\usepackage[all, poly]{xy}
\usepackage[utf8]{inputenc}
\usepackage{hyperref}
\usepackage{amsmath, amscd}
\usepackage{soul}
\usepackage{tikz}
\usetikzlibrary{matrix,arrows,decorations.pathmorphing}
\usetikzlibrary{arrows}
\allowdisplaybreaks

\newlength{\bibitemsep}
\newlength{\bibparskip}
\let\oldthebibliography\thebibliography
\renewcommand\thebibliography[1]{%
  \oldthebibliography{#1}%
  \setlength{\parskip}{\bibitemsep}%
  \setlength{\itemsep}{\bibparskip}%
}

\theoremstyle{plain}
\newtheorem*{prop}{Proposition}
\newtheorem{thm}{Theorem}
\newtheorem*{lem}{Lemma}
\newtheorem*{cor}{Corollary}%[section]

\theoremstyle{definition}
\newtheorem*{example}{Example}
\newtheorem*{defn}{Definition}

\newtheorem*{rem}{Remark}

\theoremstyle{remark}

\newcommand{\lie}[1]{\mathfrak{#1}}

\newcommand\bz{\mathbb Z}

\newcommand{\ev}{\operatorname{ev}}

\newcommand\binko[2]{\genfrac{[}{]}{0pt}{}{#1}{#2}}

\def\gr{\operatorname{gr}}

\advance\textwidth by 1.2in  \advance\oddsidemargin by -.6in \advance\evensidemargin by -.6in

\newenvironment{pf}{\proof}{\endproof}
\newcounter{cnt}
 \makeatletter
\def\mydggeometry{\makeatletter\dg@YGRID=1\dg@XGRID=20\unitlength=0.003pt\makeatother}
\makeatother \theoremstyle{remark}

\numberwithin{equation}{section}
\makeatletter
\def\section{\def\@secnumfont{\mdseries}\@startsection{section}{1}%
  \z@{.7\linespacing\@plus\linespacing}{.5\linespacing}%
  {\normalfont\scshape\centering}}
\def\subsection{\def\@secnumfont{\bfseries}\@startsection{subsection}{2}%
  {\parindent}{.5\linespacing\@plus.7\linespacing}{-.5em}%
  {\normalfont\bfseries}}
\makeatother

\begin{document}

\title[Parabolic quantum affine algebras]{Parabolic quantum affine algebras}
\author{Kudret Bostanci}
\address{University of Bochum, Faculty of Mathematics, Universit{\"a}tsstr. 150, 44801 Bochum, 
Germany}
\email{kudret.bostanci@rub.de}
\thanks{}
\author{Deniz Kus}
\address{Technical University of Munich, TUM School of Computation, Information and Technology, Department of Mathematics, Boltzmannstr. 3, 85748 Garching bei München, 
Germany}
\email{deniz.kus@tum.de}
\thanks{}

\subjclass[2020]{}
\begin{abstract} 
Maximal parabolic subalgebras of untwisted affine Kac–Moody algebras were studied in the context of Borel–de Siebenthal theory in \cite{CKO18}, where they were realized as certain equivariant map algebras with a non-free abelian group action. In this paper, we show that this perspective naturally extends to non-maximal parabolic subalgebras and introduce their quantum analogues - called parabolic quantum affine algebras — in analogy with ordinary quantum affine algebras and their classical counterpart, the loop algebra.
While the definition in the Drinfeld–Jimbo presentation is straightforward, the realization in Drinfeld’s second presentation requires quantum root vectors associated not only to simple roots but also to certain non-simple roots. A distinguished positive root $\gamma_0$ plays a central role in all constructions. Along the way, we construct a PBW-type basis, establish a second triangular decomposition, and determine the action of the braid group on the Cartan part of the algebra via Lusztig’s automorphisms.
Finally, we classify the finite-dimensional irreducible representations under a technical condition on $\gamma_0$, referred to as repetition-free, in terms of Drinfeld polynomials with some additional data. The key difference from the ordinary quantum affine case is that the degrees of the polynomials are only bounded by a certain highest weight, rather than being uniquely determined by it. In the maximal parabolic case, the classification can alternatively be phrased in terms of Drinfeld polynomials satisfying certain divisibility conditions.
\end{abstract}

\maketitle

\section{Introduction}
Quantized enveloping algebras were independently introduced by Drinfeld and Jimbo in the context of integrable systems and solvable lattice models. They provide a systematic approach to constructing solutions of the quantum Yang–Baxter equation and have since found profound connections with various areas of mathematics and physics, including statistical mechanics, dynamical systems, quiver varieties, and cluster algebras, among others. The representation theory of quantum affine algebras - that is, quantized enveloping algebras associated with affine Kac–Moody algebras - has been intensively studied over the past decades. Nevertheless, the understanding of the structure of their finite-dimensional representations remains limited, except in a few special cases.
Using Drinfeld’s second realization, rigorously established in \cite{B94}, the finite-dimensional irreducible representations of quantum loop algebras were classified by Chari and Pressley \cite{CPsl2,CP95} in terms of Drinfeld polynomials. This fundamental result quantizes the classical version of loop algebras developed in \cite{Cint86,CP95}.
Since then, finite-dimensional irreducible objects have been explored using various methods and approaches, including minimal affinizations, graded limits, cluster algebras, path descriptions of $q$-characters, and crystal bases (see, for example, \cite{BK24a,BC19a,BCKV22,C95min,H05ab,HL10a,HL13b,MY12a,N13b,N13a} and the references therein).

In another direction of research, substantial progress has been made in quantizing further important families of algebras and in studying their structure and representation theory. For an involutive automorphism of a finite-dimensional classical Lie algebra, quantized analogues of the universal enveloping algebra of the fixed-point subalgebra were constructed for example in \cite{Nou95a}; these are known as quantum symmetric pairs. A comprehensive theory of quantum symmetric pairs was later developed independently by Letzter \cite{L99a,L02b} for semisimple symmetric Lie algebras. In the infinite-dimensional setting of symmetrizable Kac–Moody algebras, various quantum symmetric pairs have appeared in the literature, such as the $q$-Onsager algebra \cite{BK05a} (which arises as a symmetry algebra in certain quantum integrable models), twisted $q$-Yangians \cite{MRS03a}, and twisted quantum loop algebras. A general theory for symmetrizable Kac–Moody algebras was developed in the seminal work of Kolb \cite{Ko14a}.

This direction has been further developed, for example in \cite{CLW21a,CLW21b,LWZ24a}, under the name of $\imath$-quantum groups. Many fundamental results - such as the theory of canonical bases, Hall algebra realizations, Drinfeld presentations, geometric realizations, and various categorifications of quantized enveloping algebras - have since been extended to the setting of $\imath$-quantum groups.

The motivation of this paper is to construct quantum analogues of another important family of algebras - namely, the parabolic subalgebras of untwisted affine Kac–Moody algebras, studied for example in \cite{CKO18} in the context of Borel–de Siebenthal pairs. These subalgebras can be realized as fixed-point subalgebras of a finite group action on the associated current algebra, where the action is induced by automorphisms of the first kind. In other words, they form examples of equivariant map algebras in the sense of \cite{NSS}.
The representation theory of equivariant map algebras depends crucially on whether the group acts freely on the underlying scheme. In the case of free actions, it has been shown that the representation theory essentially coincides with that of the current algebra (see, for example, \cite{FKKS11}). However, in the case of parabolic subalgebras, the group action is typically not free, and the order of the automorphism can be arbitrarily large; in particular, these automorphisms are not necessarily involutive. As a result, many interesting and non-trivial differences arise in the structure and classification of finite-dimensional irreducible representations.

Building on the classification results of \cite{NSS}, the main motivation of this paper is to extend this framework to the quantum setting - i.e., to define parabolic quantum affine algebras and determine their finite-dimensional irreducible representations. The definition in the Drinfeld-Jimbo picture (see Definition~\ref{maindefDJ}) is straightforward, and one way to justify it is by showing that the classical limit coincides with the corresponding enveloping algebra of the parabolic subalgebra (see Proposition~\ref{classli1}).

For the development of the representation theory, it is more natural to work within Drinfeld’s second realization, as is done for ordinary quantum affine algebras. However, the image under the Beck isomorphism involves quantum root vectors, defined using the braid group action of Lusztig, corresponding to non-simple roots which makes certain calculations more challenging. In particular, it involves a positive root $\gamma_0$ of the underlying simple Lie algebra, which belongs to a simple system of a reductive Lie subalgebra and plays a crucial role in our constructions. We determine how these root vectors arise from the braid group action in Proposition~\ref{rel3210} and identify the algebra generators of parabolic quantum affine algebras in Proposition~\ref{genalgim} in terms of quantum root vectors. These results are then applied to derive a PBW-type basis in Proposition~\ref{pbwpara}, following the approach of \cite{B94A}.

In order to study the theory of (pseudo) highest weight representations, we prove a second triangular decomposition in Corollary~\ref{trisecond} and determine the action of the braid group on the Cartan part (modulo certain terms) in Proposition~\ref{identfuerphil}. This action turns out to be described by the same formulas as in \cite{C02B}, where the action on the space of certain algebra maps is studied to obtain sufficient conditions for the cyclicity of tensor products. However, to be able to better control the modulo terms, we have to restrict ourselves to parabolic subalgebras where $\gamma_0$ is so-called repetition-free, and all Cartan matrix entries are $\geq -2$. This leads to Corollary~\ref{coxident1}, which is the only place where this restriction is necessary. If this corollary holds in general, then the classification picture holds as well.

Finally, the classification of the simple objects (under the aforementioned restrictions) in the category $\mathcal{C}_q$ of finite-dimensional type 1 representations (see Section~\ref{section7} for the appropriate definition of type 1 in this context) is given in terms of Drinfeld polynomials, along with additional data, in Theorem~\ref{mainthmrep}. The main difference to \cite{CP95} is that the degrees of the involved polynomials are only bounded by a certain highest weight, rather than being completely determined by it.

It would be interesting to extend this line of research by asking the same fundamental questions for quantum affine algebras in the parabolic setting. For instance, the study of Category $\mathcal{O}$ was done in \cite{HJ12a} for Borel subalgebras of quantum affine algebras; see also \cite{Ne25a} for an extension to the non-affine setting.

\textit{Organization of the paper:} In Section~\ref{section2}, we establish the main notations used throughout the paper and prove several results concerning finite Weyl groups. Section~\ref{section3} introduces parabolic subalgebras of Kac-Moody algebras and realizes them as twisted current algebras corresponding to automorphisms of the first kind. We also recall the classification of their finite-dimensional irreducible representations in terms of evaluation modules. Section~\ref{section4} is devoted to parabolic quantum affine algebras, focusing on the Drinfeld–Jimbo description, while Section~\ref{section5} explores the analogous Drinfeld realization in terms of quantum root vectors and identifies the algebra generators. In Section~\ref{section6}, we construct a PBW-type basis and provide a second triangular decomposition. Section~\ref{section7} derives useful identities in quantum affine algebras. Finally, in Section~\ref{section8}, we classify the irreducible finite-dimensional representations in the repetition-free case in terms of Drinfeld polynomials with some additional data.

\textit{Acknowledgement:} The authors are grateful for many stimulating and enlightening conversations with Jonathan Beck, Matheus Brito, David Hernandez and Adriano Moura.
\section{Preliminaries}\label{section2}
\subsection{} We denote by $\mathbb{Z}$, $\mathbb{Z}_+$, and $\mathbb{N}$ the set of integers, the set of non-negative integers and the set of positive integers respectively. Given two vector spaces $V,W$ over a field $\mathbb{K}$  the corresponding tensor product $V\otimes_{\mathbb K}W$ will be denoted just as $V\otimes W$. For an indeterminate $q$ and  $r \in \mathbb{Z},\ s,s'\in \mathbb{Z}_+$ with $s\geq s'$ we set
$$ [r]_{q} = \dfrac{q^r-q^{-r}}{q-q^{-1}},\quad [s]_{q}! = [s]_{q}\cdot [s-1]_{q}\cdots [1]_q,\quad \binko{s}{s'}_{q} = \dfrac{[s]_{q}!}{[s']_{q}![s-s']_{q}!}.$$
For a complex Lie algebra $\lie a$, we denote by $\mathbf{U}(\lie a)$ the corresponding universal enveloping algebra and recall that it is a Hopf algebra with  comultiplication $\Delta(x)=x\otimes 1+1\otimes x$ and antipode $S(x)=-x$ for all $x\in\lie a$. Given any commutative associative algebra $R$ over the complex numbers, we equip the tensor product $\lie a\otimes R $ with a Lie algebra structure defined by $$[x\otimes a, y\otimes b]=[x,y]\otimes ab,\ \ x,y\in\lie a, \ a,b\in R.$$
In the special case when $R$ is $\mathbb C[t]$ or $\mathbb C[t^{\pm}]$ we set $$\lie a[t]=\lie a\otimes \mathbb C[t],\ \ \lie a[t^{\pm}]=\lie a\otimes \mathbb C[t^{\pm }].$$

\subsection{} We denote by $A=(a_{ij})_{i,j}$ an indecomposable affine Cartan matrix, and let $\Gamma$ be the corresponding Dynkin diagram with the labelling of vertices as in Table Aff $1$ from \cite[pg.54]{K90}. Let $\mathring{\Gamma}$ be the Dynkin diagram obtained from $\Gamma$ by dropping the zeroth node and let $\mathring{A}$ be the Cartan matrix whose Dynkin diagram is $\mathring{\Gamma}$. Let $\lie g$ and $\mathring{\lie g}$  respectively be the untwisted affine Kac-Moody Lie algebra and the finite-dimensional simple Lie algebra associated to $A$ and $\mathring{A}$ respectively with index set $I=\{0,\dots,n\}$ and $\mathring{I}=I\setminus \{0\}$. We denote by $(\cdot ,\cdot )$ the normalized  invariant form on $\lie g$ (see \cite[Section 6.2]{K90}). We have 
$\lie g\cong \mathring{\lie g}[t^{\pm}]\oplus \mathbb{C}c\oplus \mathbb{C}d$ with Lie bracket 
\begin{align*} [x\otimes t^r+a_1c+b_1d, y\otimes t^s+a_2c+b_2d] = [x, y]\otimes t^{r+s}+&b_1s(y\otimes t^s)&\\&-b_2r(x\otimes t^r)+r\delta_{r, -s}(x ,y ) c.
\end{align*} for all
$x, y\in  \mathring{\lie g},\  r, s \in \mathbb{Z},\  a_1, a_2, b_1, b_2\in \mathbb{C}$, where $c$ is the canonical central element and $d$ is the derivation of $\lie g$. We fix a Cartan subalgebra $\mathring{\lie h}$ and a Borel subalgebra $\mathring{\lie b}$ 
of $\mathring{\lie g}$ with corresponding set of roots $\mathring{R}$ and positive roots $\mathring{R}^+$ respectively. Note that $$\lie h=\mathring{\lie h}\oplus \mathbb{C}c\oplus \mathbb{C}d$$
is a Cartan subalgebra of $\lie g$ and the standard Borel subalgebra of $\lie g$ is given by 
$$\lie b=(\mathring{\lie g}\otimes t \mathbb{C}[t])\oplus (\mathring{\mathfrak{b}}\otimes 1)\oplus \mathbb{C}c\oplus \mathbb{C}d.$$
We denote by $R$ and $R^+$ respectively the set of roots and positive roots of $\lie g$ respectively. We have the following explicit description
$$R=\{\alpha+r\delta: \alpha\in\mathring{R}, r\in \mathbb{Z}\}\sqcup \{r\delta: r\in \mathbb{Z}\backslash \{0\}\}$$
$$R^+=\{\alpha+r\delta: \alpha\in\mathring{R}, r\in \mathbb{N}\}\sqcup \mathring{R}^+\cup \{r\delta: r\in \mathbb{N}\}$$
where $\delta$ denotes the unique non-divisible positive imaginary root of $R$. 
\begin{rem} The form $(\cdot,\cdot)$ induces an isomorphism $\nu: \lie h\rightarrow \lie h^*$ and we keep the same notation for the induced form on $\lie h^*$. 
\end{rem}
In the rest of the paper we fix simple roots $\Pi=\{\alpha_0,\alpha_1,\dots,\alpha_n\}$ of $R$ such that $\mathring{\Pi}=\{\alpha_1,\dots,\alpha_n\}$ is a set of simple roots of $\mathring{R}$. Furthermore, let $\theta\in\mathring{R}^+$ and $\theta_s\in\mathring{R}^+$ respectively be the highest root and highest short root respectively of $\mathring{\lie g}$ and let $\{x^\pm _\alpha,  h_i :\alpha\in \mathring{R}^+, i\in \mathring{I}\}$ be a Chevalley basis of $\mathring{\lie g}$. We set for simplicity $x^\pm_i=x^{\pm}_{\alpha_i}$. For a root $\alpha\in R$ we denote its coroot by $\alpha^{\vee}$ and note that we have $\alpha^{\vee}_i=h_i$ for simple roots. Furthermore it will be convenient to introduce $d_{\alpha}=\frac{2}{(\alpha,\alpha)}$ for a real root $\alpha$.

Let $ P$ (resp. $P^+$) be the $\bz$-span of $\delta $ and the $\bz$--span (resp. $\bz_+$--span) of the affine fundamental weights $\Lambda_0,\dots,\Lambda_n$ defined by
$$\Lambda_0(c)=1,\ \ \Lambda_0(\lie h\oplus\mathbb Cd)=0,\ \ \Lambda_i(\alpha^{\vee}_j)=\delta_{i,j},\ \ i\in \mathring{I},\ \  j\in I.$$
Let $Q$ (resp. $Q^+)$ be the $\mathbb Z$--span (resp. $\mathbb Z_+$--span) of the affine simple roots. The sets $\mathring{P},\mathring{Q}, \mathring{P}^+$ and $\mathring{Q}^+$ are defined analogously and we denote a basis of $\mathring{P}$ by $\{\varpi_1,\dots,\varpi_n\}$. Furthermore, we denote by $P^{\vee}$ and $Q^{\vee}$ the affine coweight and affine root lattice respectively and introduce similarly $\mathring{P}^{\vee}$ and $\mathring{Q}^{\vee}$. Let $ \mathbf{a}_i,\mathbf{a}^{\vee}_i,\mathrm{ht},\mathrm{ht}^{\vee}: Q\rightarrow \mathbb{Z}$ the functions determined by the formulas
$$\alpha=\sum_{i\in I}\mathbf{a}_i(\alpha)\alpha_i,\ \ \beta^{\vee}=\sum_{i\in I}\mathbf{a}^{\vee}_i(\alpha) \alpha_i^{\vee}, \ \ \mathrm{ht}(\alpha)=\sum_{i\in I} \mathbf{a}_i(\alpha),\ \ \mathrm{ht}^{\vee}(\alpha)=\sum_{i\in I} \mathbf{a}^{\vee}_i(\alpha),\ \ \ \alpha\in Q.$$
The support of $\alpha\in Q$ is defined by $\mathrm{supp}(\alpha)=\{i\in I: \mathbf{a}_i(\alpha)\neq 0\}$.
\subsection{} Let $\mathring{W}$ be the Weyl group of $\mathring{R}$ with longest word $w_0$ and recall that the affine Weyl group $W$ is the subgroup of  ${\rm{Aut}}(\lie h^*)$ generated by the set $\{\mathbf{s}_i: i\in I\}$ where $$\mathbf{s}_i(\lambda)=\lambda-\lambda(\alpha^{\vee}_i)\alpha_i, \ \ i\in I,\ \ \lambda\in  P.$$
Clearly $\mathring{W}$ is a subgroup of $W$ and we have an isomorphism $W=\mathring{W}\ltimes t_{M}$
 where  $t_{M}=\{t_\mu: \mu\in M\}$ and $M=\nu(\mathring{Q}^{\vee})$  is the lattice  generated by the elements $w(\theta)$ for all $w\in \mathring{W}$.  The element $t_{\mu}$ acts on $\beta\in \lie h^{*}$ by 
 $$ t_{\mu}(\beta) = \beta -  ( \beta,\mu )\delta,\ \ \beta\in\mathring{\mathfrak{h}}^*\oplus \mathbb{C}\delta,\ \ t_{\mu}(\Lambda_0)=\Lambda_0+\mu-\frac{1}{2}(\mu,\mu)\delta.$$ 
It will also be necessary to introduce the extended affine Weyl group $W^{\mathrm{ext}}= \mathring{W}\ltimes t_L$  
where
$$ L=\nu(\mathring{P}^{\vee}),\ \ M=\nu(\mathring{Q}^{\vee})\subseteq L$$
The extended affine Weyl group has the alternative description
$$W^{\mathrm{ext}}=\mathcal{T}\ltimes W$$
where $\mathcal{T}$ is the subgroup of $W$ stabilizing the dominant Weyl chamber. Recall that the elements of $\mathcal{T}$ correspond to automorphisms of the affine Dynkin diagram $\Gamma$ and (see \cite[Proposition 6]{Bou1})
$$\mathbf{a}_i(\theta)=1\implies t_{\nu(\varpi_i^{\vee})}=t_{\varpi_i}=\tau w_0w_{0,i},\  \tau(0)=i$$
where $w_{0,i}$ is the longest word in the stabilizer of $\varpi_i$ in $\mathring{W}$. As usual $\ell(\cdot)$ denotes the length function on the extended affine Weyl group and recall that elements in $\mathcal{T}$ have length zero. For $w\in W$, set 
$$I(w)=\{i\in I: \mathbf{s}_i \text{ appears in a reduced expression of $w$}\}$$
The following lemma should be well known, but we were unable to find a reference in the literature, so we include a proof for completeness.
\begin{lem}\label{heightmin}
Let $\alpha\in \mathring{R}^+$ and $w\in \mathring{W}$ be of minimal length such that $w(\alpha)$ is simple. Then we have \begin{equation}\label{ht15}\ell(w)=\begin{cases}\mathrm{ht}(\alpha)-1,& \text{ if $\alpha$ is a short root}\\
\mathrm{ht}^{\vee}(\alpha)-1,& \text{ if $\alpha$ is a long root.}
\end{cases}\end{equation}
\begin{proof}
If $\alpha$ is simple, there is nothing to show. So assume in the rest of the proof that $\alpha$ is not simple. We choose $i\in \mathring{I}$ such that $(\alpha,\alpha_i)>0$; otherwise we have that $\alpha=w_0(\theta)$ or $\alpha=w_0(\theta_s)$ contradicting $\alpha\in \mathring{R}^+$.
Hence $\alpha-\alpha_i$ is a root and
\begin{equation}\label{rew}\frac{2(\alpha,\alpha_i)}{(\alpha_i,\alpha_i)}=\frac{(\alpha,\alpha)}{(\alpha_i,\alpha_i)}-\frac{(\alpha-\alpha_i,\alpha-\alpha_i)}{(\alpha_i,\alpha_i)}+1.\end{equation}
If $\alpha$ is a short root, we claim that $\alpha-\alpha_i$ is also short. Assume that $\alpha-\alpha_i$ is long, then we must have that $\alpha_i$ is short (see for example \cite[Lemma 2.1]{KV20}) which gives with \eqref{rew} the contradiction $(\alpha_i,\alpha_i)\leq 0$. So if 
$\alpha$ is a short root we have $\mathbf{s}_i(\alpha)=\alpha-\alpha_i$. Exactly in the same way using \cite[Lemma 2.1]{KV20} we get $\mathbf{s}_i(\alpha)=\alpha-d_i\alpha_i$ if  $\alpha$ is long.
Now we can repeat the above argument until we end with a simple root:
$$\mathbf{s}_{i_1}\cdots\mathbf{s}_{i_k}(\alpha)=\begin{cases}
\alpha-\sum_{p=1}^k \alpha_{i_p},& \text{if $\alpha$ is short}\\ \\
\displaystyle\alpha-\sum_{p=1}^k d_{i_p}\alpha_{i_p},& \text{if $\alpha$ is long.}
\end{cases}$$
This gives the desired upper bound for $\ell(w)$. Now the other estimation is clear since the same calculation as above gives
$$\frac{2(\beta,\alpha_j)}{(\alpha_j,\alpha_j)}=\begin{cases}\pm 1,& \text{ if $\beta$ is short}\\
\pm d_j,& \text{ if $\beta$ is long}\end{cases}$$
for all non-simple positive roots $\beta$ with $(\beta,\alpha_j)\neq 0$. Note that if $w=\mathbf{s}_{q_1}\cdots \mathbf{s}_{q_r}$ is a reduced expression we can assume in each step $(\mathbf{s}_{q_{r-p}}\cdots \mathbf{s}_{q_r}(\alpha),\alpha_{q_{r-p-1}})\neq 0$ and $\mathbf{s}_{q_{r-p}}\cdots \mathbf{s}_{q_r}(\alpha)\notin\mathring{\Pi}$ by the minimality of $w$. 
\end{proof}
\end{lem}
For $\alpha\in\mathring{R}^+$ we set 
$$\mathring{W}(\alpha)=\{w\in \mathring{W}: w(\alpha)\in\mathring{\Pi} \text{ and } \ell(w) \text{ as in \eqref{ht15}}\}$$
$$\mathring{W}_\alpha=\{w\in \mathring{W}: w(\alpha)\in\{\theta,\theta_s\} \text{ and $w$ of minimal length if $\alpha$ is short}\}$$
\begin{rem}\label{choosed76}
Given a short root $\alpha$ and $w^{-1}\in \mathring{W}(\alpha)$ there exists always $u\in \mathring{W}_{\alpha}$ such that $(uw)^{-1}\in\mathring{W}(\theta_s)$. This is a case-by-case analysis and we omit the details.
\end{rem}
We get the following corollary. 
\begin{cor}\label{corhilf4q}
Let $\alpha\in\mathring{R}^+$ and $w\in \mathring{W}(\alpha)$, say $w(\alpha)=\alpha_i$. Then we have
\begin{itemize}
    \item $(\alpha,\alpha_k)>0$ for all $k\in\mathring{I}$ with $\ell(w\mathbf{s}_k)=\ell(w)-1$ \vspace{0,1cm}
    
    \item $I(w)\sqcup \{i\}\subseteq \mathrm{supp}(\alpha)$. 
\end{itemize}
\begin{proof} The proof is by induction on $\ell(w)$ where the case $\ell(w)=0$ is obvious. So assume that 
$$w=w'\mathbf{s}_k,\ \ \ell(w)=\ell(w')+1.$$
Note that $\mathbf{s}_k(\alpha)\in \mathring{R}^+$ and the minimality of $w$ implies $w'\in\mathring{W}(\mathbf{s}_{k}(\alpha))$. If we had $(\alpha,\alpha_k)\leq 0$ we would obtain a contradiction by applying Lemma~\ref{heightmin}:
$$\ell(w)=\widetilde{\mathrm{ht}}(\alpha)-1=\ell(w')+1=\widetilde{\mathrm{ht}}(\alpha-\alpha(\alpha_k^{\vee})\alpha_k)\geq \widetilde{\mathrm{ht}}(\alpha).$$  
where $\widetilde{\mathrm{ht}}= \mathrm{ht}^{\vee}$ or $\widetilde{\mathrm{ht}}=\mathrm{ht}$
depending on whether $\alpha$ is short or long. Hence $(\alpha,\alpha_k)>0$ (in particular $k\in \mathrm{supp}(\alpha)$) and the first part follows. Our induction hypothesis gives also the second part of the claim
$$I(w')\sqcup \{i\}\subseteq \mathrm{supp}(\mathbf{s}_k(\alpha))\subseteq \mathrm{supp}(\alpha).$$
\end{proof}
\end{cor}
\subsection{} We state one more lemma and its corollary, which will be used later.
\begin{lem}\label{domsup1} Let $\lambda, \mu \in \mathring{P}^+$ such that $\lambda- \mu \in \mathring{Q}^+$ and fix $w\in\mathring{W}$. Choose an element $w'\in \mathring{W}$ of minimal length such that $\lambda -w\mu = \lambda-w'\mu.$ Then we have
$$\lambda-w\mu \in  \sum_{k\in \  \text{supp}(\lambda-\mu)\hspace{0,03cm} \sqcup \hspace{0,03cm} I(w')}\mathbb{N} \alpha_k.$$

\begin{proof} If $\ell(w')=0$, there is nothing to show. So assume that $w' = \mathbf{s}_jw''$ and $\ell(w'')+1 = \ell(w').$ Then
$$ \lambda-w\mu= \lambda-w'\mu =\lambda - \mathbf{s}_jw''\mu = \lambda- w''\mu+ \mu((w'')^{-1}(\alpha_j^{\vee}))\alpha_j.$$
Since $\ell(\mathbf{s}_jw'') = \ell(w'')+1$ we know that $(w'')^{-1}(\alpha_j^{\vee})$ is a positive coroot. Furthermore $\mu \in \mathring{P}^+$ and the minimality of $w'$ yields $\mu((w'')^{-1}(\alpha_j^{\vee}))>0$. Hence we can repeat the above procedure with $\lambda- w''\mu$ to finish the proof provided that there is no element $u\in\mathring{W}$ with $\lambda - w''\mu = \lambda -u\mu$ and $\ell(u) < \ell(w'')$. However the existence of such an element would give 
$$\lambda - w''\mu = \lambda -u\mu \implies \mathbf{s}_j\lambda - w'\mu = \mathbf{s}_j\lambda - \mathbf{s}_ju\mu \implies \lambda - w'\mu = \lambda - \mathbf{s}_ju\mu$$
which is a contradiction to the minimality of $w'$ since
$$\ell(w') \leq \ell(\mathbf{s}_ju) \leq \ell(u)+1 < \ell(w'')+1 = \ell(w').$$
\end{proof}
\end{lem}
\begin{cor}\label{observation1}
Let $\beta\in\mathring{R}^+$. There exists $w'\in \mathring{W}_{\beta}$ with $$I(w') \subseteq \mathring{I}\cap 
\mathrm{supp}(-\beta+\delta).$$
\end{cor}
\begin{proof}
We can always take an element $u\in\mathring{W}_{\beta}$. Now the claim follows from Lemma~\ref{domsup1} by choosing $\lambda=\theta, w=u^{-1}$ and $\mu=\theta$ (resp. $\mu=\theta_s$) if $\beta$ is long (resp. short).
\end{proof}
\section{Parabolic subalgebras and their representations}\label{section3}

\subsection{} In this subsection we recall the realization of parabolic subalgebras of $\lie g$ containing the standard Borel subalgebra $\lie b$ (see for example \cite[Lemma 3.4]{CKO18}). Let $J=\{j_1,\dots,j_k\} \subsetneq I$ be a proper subset of the index set of $\lie g$. We define $\mathfrak{p}_{J}$ to be the parabolic subalgebra associated to $J$, i.e. the subalgebra of $\lie g$ generated by $\lie b$ and $\lie g_{- \alpha_i}$ for all $ i \in J$. Recall that 
$$\lie g_{- \alpha_i}=\mathring{\lie g}_{- \alpha_i}=\mathbb{C} \cdot (x_{i}^-\otimes 1),\ \ i\in\mathring{I},\ \ \lie g_{- \alpha_0}=\mathbb{C} \cdot (x_{\theta}^+\otimes t^{-1})$$
and choose generators $x_{0}^{\pm}$ of $\lie g_{\pm \alpha_0}$ such that $[x_{0}^{+},x_0^{-}]$ is the coroot of $\alpha_0$. The following lemma has been proven in \cite[Lemma 3.4]{CKO18}.
\begin{lem}
Suppose that $\mathfrak{p}_J$ is a proper parabolic subalgebra of $\lie g$. Then 
$$\mathfrak{p}_J=\lie b + \sum_{\alpha}\  \lie g_{\alpha}$$
where the $\alpha's$ run over the set $(-R^+)\cap \sum_{i\in J}\mathbb{Z}\alpha_i$.
\hfill\qed
\end{lem}
\subsection{} Let $\tilde{\lie p}_J$ be the commutator subalgebra of $\lie p_J$ modulo the center. We define a $\mathbb{Z}_+$--grading on $\tilde{\lie p}_J$ by declaring that the weight space $\lie g_{\alpha}$ with $\lie g_{\alpha}\subseteq \tilde{\lie p}_J$ has grade 
$$\gr_J(\alpha)=\sum_{i\notin J}\mathbf{a}_i(\alpha).$$
We can decompose our subalgebra into graded components
$$\tilde{\lie p}_J=\bigoplus_{r\in\mathbb{Z}_+} \mathfrak{p}_J[r].$$
\begin{example} Recall that the positive roots of type $B_n$ are given as follows
$$\alpha_{i,j}:=\alpha_i+\cdots+\alpha_j,\ 1\leq i\leq j\leq n,
$$
$$\alpha_{i,\bar{j}}:=\alpha_i+\cdots+\alpha_{j-1}+2(\alpha_j+\cdots+\alpha_n),\ 1\leq i<j\leq n.$$
\begin{enumerate}
\item [(i)] If $J=\mathring{I}$, then $\tilde{\lie p}_{\mathring{I}}$ is the current algebra $\mathring{\lie g}[t]$ associated to $\mathring{\lie g}$ and the above grading is just the usual grading of the current algebra, i.e. $\tilde{\mathfrak{p}}_{\mathring{I}}[r]=\mathring{\lie g}\otimes t^r$.
\item[(ii)] Let $\mathfrak{g}$ be of type $B^{(1)}_3$ and $J=\{0,1,2\}$.  In this case we have
\begin{align*}
&(-R^+)\cap (\mathbb{Z}\alpha_0+\mathbb{Z}\alpha_1+\mathbb{Z}\alpha_2) =\{-\alpha_1, -\alpha_2, -\alpha_{1,2}, \alpha_{2,\bar{3}}-\delta,\alpha_{1,\bar{3}}-\delta,\alpha_{1,\bar{2}}-\delta \}
\end{align*}
and hence we have the following graded basis of $\tilde{\lie p}_J$, $r \in \mathbb{Z}_+$:
$$\hspace{1,25cm} \mathfrak{p}_J[2r]: \ \left\{ x_{\alpha}^{\pm} \otimes t^r, x_{\beta}^{\pm} \otimes t^{r\mp 1},  h_i \otimes t^r: \beta\in\{\alpha_{1,\bar{3}},\alpha_{1,\bar{2}},\alpha_{2,\bar{3}}\},\ \alpha\in\{\alpha_1,\alpha_2,\alpha_{1,2}\},\ i\in \mathring{I} \right\}$$
$$\hspace{0,25cm}
\mathfrak{p}_J[2r + 1]: \  \left\{x_{3}^{+} \otimes t^r,  x_{3}^{-} \otimes t^{r+1}, x_{\alpha_{2,3}}^{+} \otimes t^{r} ,  x_{\alpha_{2,3}}^{-} \otimes t^{r + 1}, x_{\alpha_{1,3}}^{+} \otimes t^{r}, x_{\alpha_{1,3}}^{-} \otimes t^{r + 1}\right\}$$

\item[(iii)] Let $\mathfrak{g}$ be of type $B^{(1)}_3$ and $J=\{0,1\}$. We have 
\begin{align*}(-R^+)\cap (\mathbb{Z}\alpha_0+\mathbb{Z}\alpha_1)=\{-\alpha_1,\alpha_{1,\bar{2}}-\delta \}.
\end{align*}
Hence for $r \in \mathbb{Z}_+$ we have  the following graded basis of $\tilde{\lie p}_J$:
\begin{align*}&\mathfrak{p}_J[4r]: \ \ \ \ \ \  \left\{ x_{1}^{\pm} \otimes t^r,x_{\alpha_{1,\bar{2}}}^{\pm} \otimes t^{r \mp 1},  h_i \otimes t^r: i\in\mathring{I} \right\} &\\\\&
\mathfrak{p}_J[4r + 1]: \  \left\{x_{2}^{+} \otimes t^r,  x_{\alpha_{1,2}}^{+} \otimes t^r,  x_{3}^{+} \otimes t^{r},x_{\alpha_{2,\bar{3}}}^{-} \otimes t^{r + 1} ,  x_{\alpha_{1,\bar{3}}}^{-} \otimes t^{r + 1}\right\}&\\\\& 
 \mathfrak{p}_J[4r+2]: \  \left\{ x_{\alpha_{2,3}}^{-} \otimes t^{r+1},x_{\alpha_{2,3}}^{+} \otimes t^{r}, x_{\alpha_{1,3}}^{-} \otimes t^{r+1}, x_{\alpha_{1,3}}^{+} \otimes t^{r}\right\} &\\\\&
 \mathfrak{p}_J[4r+3]: \  \left\{x_{3}^{-} \otimes t^{r+1}, x_{2}^{-} \otimes t^{r+1},  x_{\alpha_{1,2}}^{-} \otimes t^{r+1}, x_{\alpha_{2,\bar{3}}}^{+} \otimes t^{r} ,  x_{\alpha_{1,\bar{3}}}^{+} \otimes t^{r} \right\}
\end{align*}
\end{enumerate}
\end{example} 
\subsection{} \label{section31}We recall some finite order automorphisms from \cite[Theorem 8.6]{K90} which will be needed later to identify the parabolic subalgebras as twisted current algebras. Consider the tuple $\mathbf{s}=(s_0,s_1,\dots,s_n)$ where $s_i=1$ if $i\notin J$ and $s_i=0$ otherwise. The following map defines an $m$-th order automorphism of $\mathring{\lie g}$:
\begin{equation}\label{Kaciso}\sigma: \mathring{\lie g}\rightarrow \mathring{\lie g},\ \ x_{i}^{+}\mapsto\xi^{s_i}x_{i}^{+},\  i=1,\dots,n,\ \  \ x_{\theta}^{-}\mapsto\xi^{s_0}x_{\theta}^{-}\end{equation}
 where $\xi$ is a $m$-th root of unity and $m=\gr_J(\delta)$. Note that the automorphism $\sigma$ on arbitrary root vectors is given by $\sigma(x_{\alpha}^{\pm})=\xi^{\pm \gr_J(\alpha)}x_{\alpha}^{\pm}$. We obtain a decomposition
 $$\mathring{\lie g}=\mathring{\lie g}_0 \oplus \cdots \oplus \mathring{\lie g}_{m-1},\ \ \mathring{\lie g}_j=\{x\in \mathring{\lie g}: \sigma(x)=\xi^{j} x\},\ \ 0\leq j<m.$$
 The following lemma is clear and can be found in \cite[Proposition 8.6.]{K90}. 
 \begin{lem}\label{irrmod} The Lie algebra $\mathring{\lie g}_0$ is reductive with an $(n-|J|)$-dimensional center and the semisimple part is the Lie algebra whose Dynkin diagram is the subdiagram of $\Gamma$ consisting of the vertices lying in $J$.\hfill\qed
 \end{lem}
 A canonical simple system for $\mathring{\lie g}_0$ is given by $\{\alpha_i: i\in J\cap \mathring{I}\}\sqcup \{-\theta: \text{ if $0\in J$}\}$. However this simple system is not compatible with $\mathring{\Pi}$ in the sense that dominant weights with respect to  $\mathring{\Pi}$ are generically not dominant with respect to the above system. To get around this problem let $w_{\circ}\in \mathring{W}$ be the longest element of the Weyl group generated by $\mathbf{s}_i, i\in J\cap \mathring{I}$ and set $\gamma_0=w_{\circ}(\theta)$. Choose the simple system $$\mathring{\Pi}_0=\{-w_{\circ}(\alpha_i): i\in J\cap \mathring{I}\}\sqcup \{w_{\circ}(\theta): \text{ if $0\in J$}\}=\{\gamma_i: i\in J\cap \mathring{I}\}\sqcup \{\gamma_0:\text{ if $0\in J$} \}$$ 
 where $\gamma_i=\alpha_{\bar{i}}=-w_{\circ}(\alpha_i),\ i\in J\cap \mathring{I},\  \gamma_0=w_{\circ}(\theta)$ and $i\mapsto \bar{i}$ denotes the permutation of $\mathring{I}\cap J$ induced by $-w_{\circ}$; we set in the rest of the paper $\bar{0}=0$. We denote by $\mathring{P}_0^+$ the dominant integral weights with respect to the simple system $\mathring{\Pi}_0$. Since
 $$-w_{\circ}(\alpha_i)\in \mathring{\Pi},\ \ i\in J\cap \mathring{I}$$
and
 $$w_{\circ}(\theta)=\displaystyle\sum_{ i\in \mathring{I}\backslash \mathring{I}\cap J} \mathbf{a}_i(\theta)\alpha_i+\sum_{i\in \mathring{I}\cap J} r_i \alpha_i \in \mathring{R}^+,\ \ \text{if }0\in J$$
we get an obvious map $\iota: \mathring{P}^+\rightarrow \mathring{P}_0^+$ which is not injective in general.
\begin{example} We consider the affine algebra of type $B_3^{(1)}$.
\begin{enumerate}
\item If $J=\{0,1\}$, then $w_{\circ}=\mathbf{s}_1$ and we have $\mathring{\Pi}_0^+=\{\theta,\alpha_1\}=\{\gamma_0,\gamma_1\}$.\vspace{0,1cm}

\item If $J=\{0,1,2\}$, then $w_{\circ}=\mathbf{s}_1\mathbf{s}_2\mathbf{s}_1$ and we have 
$\mathring{\Pi}_0^+=\{\alpha_{2,\bar{3}},\alpha_2,\alpha_1\}=\{\gamma_0,\gamma_1,\gamma_2\}$.
\end{enumerate}
\end{example}
Now we extend $\sigma$ to an automorphism of $\mathring{\mathfrak{g}}[t]$ by setting 
 $$\sigma(x\otimes t^r)\mapsto \xi^{-r} \sigma(x)\otimes t^r.$$
 Therefore, we obtain a decomposition of the fixed point algebra:
 $$\mathring{\mathfrak{g}}[t]^{\sigma}= \left(\mathring{\lie g}_0\otimes \mathbb{C}[t^m]\right)\oplus  \left(\mathring{\lie g}_1\otimes t\mathbb{C}[t^m]\right)\oplus \cdots \oplus \left(\mathring{\lie g}_{m-1}\otimes t^{m-1}\mathbb{C}[t^m]\right).$$
 \begin{rem} In fact, the above automorphism can be extended to $\lie g$. The automorphism group of an infinite-dimensional symmetrizable Kac-Moody algebra has been studied in \cite{KW92} and an automorphism is either of the first or of the second kind (but never both) meaning that the automorphism either stabilizes or exchanges the two conjugacy classes (under the corresponding Kac-Moody group) of Borel subalgebras. Using \cite[Section 4.6]{KW92} we see that the above described automorphisms are all of the first kind since $$\mathrm{dim}(\sigma(\mathfrak{b})\cap \mathfrak{b})=\infty.$$ 
 \end{rem}
Recall that the $\mathring{\lie g}_0$-modules $\mathring{\lie g}_j$ need not to be irreducible, for example, $\mathring{\lie g}_1$ and $\mathring{\lie g}_{m-1}$ decompose into a direct sum of $(n+1-|J|)$ irreducible components. However, the irreducibility, which is guaranteed only for maximal parabolic subalgebras, is crucial in the proof of \cite[Proposition 3.2.]{CKO18}. A slight modification of the proof shows that the following result still holds. For the reader's convenience, we include the proof below.
 \begin{prop}\label{finco}
 Any non-zero ideal in $\mathring{\mathfrak{g}}[t]^{\sigma}$ is of finite-codimension. 
 \begin{proof} Let $\mathfrak{I}$ be an ideal of $\mathring{\mathfrak{g}}[t]^{\sigma}$ and let $\alpha,\beta\in \mathring{R}^+$ be such that $\beta(\alpha^{\vee})\neq 0$. We start by proving the following fact:
\begin{equation}\label{fact} x_{\alpha}^+\otimes t^{\gr_J(\alpha)}g \in\mathfrak{I} \text{ for some $g\in\mathbb{C}[t^m]$ } \implies x_{\beta}^+\otimes  t^{\gr_J(\beta)+\ell m} g \in\mathfrak{I} \text{  for all $\ell\geq 1$. } \end{equation}
 To see this we simply take 
 $(\alpha^{\vee}\otimes t^{\ell m} g)=[x_{\alpha}^+\otimes t^{\gr_J(\alpha)} g,x_{\alpha}^-\otimes t^{\ell m-\gr_J(\alpha)}]\in \mathfrak{I}$
 and thus
 $$(x_{\beta}^+\otimes t^{\gr_J(\beta)+\ell m} g)=\beta(\alpha^{\vee})^{-1}[\alpha^{\vee}\otimes t^{\ell m} g,x_{\beta}^+\otimes t^{\gr_J(\beta)}]\in \mathfrak{I}$$
which shows \eqref{fact}. Since $\mathfrak{I}$ is preserved by the adjoint action of $\mathring{\lie{h}}$, there exists a non-zero element $H\in \mathfrak{I}\cap (\mathring{\mathfrak{h}}\otimes \mathbb{C}[t^m])$. Moreover, since the Cartan matrix $\mathring{A}$ is invertible we can also find an element $p\in\{1,\dots,n\}$ with $[H,x_p^+]\neq 0$. This implies $x_p^+\otimes t^{\gr_J(\alpha_p)} g \in \mathfrak{I}$ for some $g\in\mathbb{C}[t^m]$. Applying \eqref{fact} for all indices connected to $p$ in the Dynkin diagram $\mathring{\Gamma}$ and repeating the argument with the newly obtained indices we get the following
$$\exists g'\in\mathbb{C}[t^m] : (x_i^+\otimes t^{\gr_J(\alpha_i)} g ')\in \mathfrak{I},\ \  \forall i\in \mathring{I}.$$
In particular, by taking the bracket of the above element with $(x_i^-\otimes t^{m-\gr_J(\alpha_i)})$ we obtain
$$\exists f\in\mathbb{C}[t^m] : (\alpha^{\vee}_i\otimes f)\in \mathfrak{I},\  \text{ $\forall$ $i\in \mathring{I}$}.$$
Now let $(x_{\alpha}^{\pm}\otimes t^{j} f')\in \mathring{\lie g}_j\otimes t^j\mathbb{C}[t^m]$. Since there is  an index $i$ with $\alpha(\alpha^{\vee}_i)\neq 0$ we have
$$\pm \alpha(\alpha^{\vee}_i)^{-1} [\alpha^{\vee}_i\otimes f,x_{\alpha}^{\pm}\otimes t^{j} f']=(x_{\alpha}^{\pm}\otimes t^{j} ff')\in\mathfrak{I}.$$
In particular, $\mathring{\lie g}_j\otimes t^j \langle f \rangle\subseteq\mathfrak{I}$ for all $j=0,\dots,m-1$ and therefore $\mathfrak{I}$ is of finite-codimension.
\end{proof}
 \end{prop}
 \begin{rem}
 In fact we never used the restriction on the tuple $\mathbf{s}$ in the above proposition and hence the statement is true for any isomorphism of the form \eqref{Kaciso}. 
 \end{rem}
 Obvioulsy $[\mathring{\lie g}[t]^{\sigma},\mathring{\lie g}[t]^{\sigma}]$ is an ideal in $\mathring{\lie g}[t]^{\sigma}$ and hence of finite-codimension by Proposition~\ref{finco}. This implies that there exists $f_y\in\mathbb{C}[t^m]$ for each $y\in \mathring{\lie g}_0$ such that 
$$y\otimes f_y \in [\mathring{\lie g}[t]^{\sigma},\mathring{\lie g}[t]^{\sigma}].$$
In particular, $$y\in [\mathring{\lie g},\mathring{\lie g}]\implies y\otimes t^m \mathbb{C}[t^m]\in  [\mathring{\lie g}[t]^{\sigma},\mathring{\lie g}[t]^{\sigma}].$$
Moreover each $\mathring{\lie g}_k$ with $k>0$ decomposes into irreducible $\mathring{\lie g}_0$-modules since $\mathring{\lie g}_0$ is reductive and its center acts by semisimple endomorphisms. Since $\mathring{\lie g}_k$ is $\mathring{\lie h}$-stable (recall $\mathring{\lie h}\subseteq \mathring{\lie g}_0$) it is a direct sum of root spaces and we have $\mathring{\lie g}_k=[\mathring{\lie g}_0,\mathring{\lie g}_k]$ for all $k>0$. Hence
$$\mathring{\lie g}_k\otimes t^k \mathbb{C}[t^m]\subseteq [\mathring{\lie g}[t]^{\sigma},\mathring{\lie g}[t]^{\sigma}],\ \ k\neq 0.$$
So the center is given by
\begin{equation}\label{zentr}Z(\mathring{\lie g}_0)\cong \mathring{\lie g}[t]^{\sigma}/[\mathring{\lie g}[t]^{\sigma},\mathring{\lie g}[t]^{\sigma}]\cong \mathbb{C}^{n-|J|}.\end{equation}
 \subsection{}
 The following result extends \cite[Proposition 3.5]{CKO18} from maximal to arbitrary parabolic subalgebras and establishes their identification with twisted current algebras.
 \begin{prop}\label{isotw}
 The following map given on graded elements by
 $$x_{\alpha}^{\pm}\otimes t^r\mapsto (x_{\alpha}^{\pm}\otimes t^{\gr_J(\pm \alpha+r\delta)}),\ \ \alpha\in \mathring{R}^+$$
defines a graded isomorphism $\tilde{\lie p}_J\cong \mathring{\mathfrak{g}}[t]^{\sigma}$.
\begin{proof}
The map is well-defined:
$$\sigma(x_{\alpha}^{\pm}\otimes t^{\gr_J(\pm \alpha+r\delta)})=\xi^{\pm \gr_J(\alpha)-\gr_J(\pm \alpha+r\delta)}(x_{\alpha}^{\pm}\otimes  t^{\gr_J(\pm \alpha+r\delta)})=x_{\alpha}^{\pm}\otimes  t^{\gr_J(\pm \alpha+r\delta)}$$
and is clearly an injective Lie algebra homomorphism. The surjectivity follows from the fact that an element $x_{\alpha}^{\pm}\otimes t^{ms+j}$ in the space $\left(\mathring{\lie g}_j\otimes t^j\mathbb{C}[t^m]\right)$ has preimage $x_{\alpha}^{\pm}\otimes t^k$ where 
$$k=s+\frac{j\mp \gr_J(\alpha)}{m}.$$
Note that $\sigma(x_{\alpha}^{\pm})=\xi^{\pm \gr_J(\alpha)}x_{\alpha}^{\pm}=\xi^{j}x_{\alpha}^{\pm}$ and thus $j\mp \gr_j(\alpha)\equiv 0 \mod m$. It remains to check that the element $x_{\alpha}^{\pm}\otimes t^k$ lies in $\tilde{\lie p}_J$. So if we consider $x_{\alpha}^{+}\otimes t^{ms+j}$, then $x_{\alpha}^{+}\otimes t^k$ lies in $\mathfrak{b}$ unless $k<0$, which can only happen if $j- \gr_J(\alpha)=-m$ and $s=0$. Equivalently $j=0=s$ and $\gr_J(\alpha)=m$. This implies 
$$\mathbf{a}_i(\theta)=\mathbf{a}_i(\alpha), \ \ \text{ for all $i\notin J$}$$
and hence $(\alpha-\delta)\in (-R^{+})\cap \sum_{i\in J} \mathbb{Z}\alpha_i$ (since $\gr_J(\alpha)=m$ forces $0\in J$). The remaining case works similarly and we omit the details. 
\end{proof}
 \end{prop}
 \begin{rem}
Using Proposition~\ref{finco} and Proposition~\ref{isotw}, one can follow the same strategy as in \cite{CKO18} to show that parabolic subalgebras $\tilde{\lie p}_J$ with $m=\gr_J(\delta)>1$ do not arise as equivariant map algebras $(\mathfrak{a}\otimes R)^{\Sigma}$ with $\mathfrak{a}$ semi--simple and $\Sigma$ acting freely on the set of maximal ideals of $R$. As in the proof of \cite[Proposition 3.3]{CKO18}, one uses the fact that $\tilde{\lie p}_J$ has two non-isomorphic simple finite-dimensional quotients. These are constructed from the evaluation maps
 $$\ev_0: \mathring{\mathfrak{g}}[t]^{\sigma}\rightarrow \mathring{\mathfrak{g}}_0,\ \ \ \ev_z: \mathring{\mathfrak{g}}[t]^{\sigma}\rightarrow \mathring{\mathfrak{g}},\ \ z\neq 0.$$

 \end{rem}
\subsection{}Recall from the Gabber-Kac theorem \cite{GK81} that $\lie g$ is generated as a Lie algebra by elements $x_r^{\pm}$, $r\in I$, and $\lie h$ subject to the following relations
\begin{align}\label{1}[x_r^{+},x_s^{-}]&=\delta_{r,s} h_r,\ \ [h,x_s^{\pm}]=\pm\alpha_{s}(h) x_s^{\pm},\ \ [h,h']=0,\ h,h'\in \lie h&\\& \label{2} \left[x_r^{\pm},\dots,\left[x_r^{\pm},\left[x_r^{\pm},x_s^{\pm}\right]\right]\cdots\right]=0,\ \ r\neq s\end{align}
where the number of occurrences of $x_r^{\pm}$ in the Lie word \eqref{2} is $(1-a_{r,s})$.
From this description we obtain that the defining relations for the parabolic subalgebra $\mathfrak{p}_J$ with generators $x_i^{+},x_j^{-},h_i,d$ for $i\in I$ and $j\in J$ are given by the same relations \eqref{1} and \eqref{2} with adapted index set, i.e. we consider only relations where $x_s^{-}$ with $s\in J$ is involved. To see this, assume that we have a relation $Y=0$ in $\mathfrak{p}_J\subseteq \lie g$ which means that $Y$ is a linear combination of  Lie words each contained in the ideal generated by the elements \eqref{1} und \eqref{2}. Let $Y^-\in \lie n^-$ be the projection with respect to the triangular decomposition $\lie g=\lie n^+\oplus \lie h \oplus \lie n^-$. In particular, $Y^-=0$ in $\mathfrak{p}_J$ and is contained in the ideal generated by  (c.f. \cite[Theorem 9.11]{K90})
$$\left[x_r^{-},\dots,\left[x_r^{-},\left[x_r^{-},x_s^{-}\right]\right]\cdots\right]=0,\ \ r\neq s.$$
Since $Y^-$ is contained in a sum of weight spaces $\lie g_{\alpha}$ with $\gr_J
(\alpha)=0$ we must have $r,s\in J$ and the claim follows.
This gives a generators and relations presentation for $\mathbf{U}(\mathfrak{p}_J)$.
\subsection{} In this subsection we will combine the results of \cite{NSS} and Proposition~\ref{isotw}
to obtain a full list of finite-dimensional irreducible representations for $\mathbf{U}(\tilde{\lie p}_J)$. Recall the construction of the simple system $\mathring{\Pi}_0=\{\gamma_i: i\in J\cap \mathring{I}\}\sqcup \{\gamma_0:\text{ if $0\in J$} \}$ from Subsection~\ref{section31}.
\begin{defn}  We define $\mathcal{P}^+$ to be the set of tuples
$$(\boldsymbol{\pi}_{1}(u),\dots,\boldsymbol{\pi}_{n}(u),\mu,w),\ \ \boldsymbol{\pi}_{i}(u)\in\mathbb{C}[u],\ \ \mu\in \mathring{P}_0^+,\ \  w\in \mathbb{C}^{n-|J|}$$
satisfying 
$$\boldsymbol{\pi}_i(0)=1,\ \forall i\in\mathring{I},\ \ \mathrm{deg}(\boldsymbol{\pi}_{\bar{j}}(u))\leq \mu(\gamma^{\vee}_j) \ \forall j\in J$$ 

where we understand $\boldsymbol{\pi}_0(u)=\prod_{i\in \mathring{I}} \boldsymbol{\pi}_i(u)^{\mathbf{a}^{\vee}_i(\gamma_0)}$.
\end{defn}
The next result is essentially a reformulation of the main theorem \cite[Theorem 5.5]{NSS} combined with Proposition~\ref{isotw} and \eqref{zentr}.
\begin{thm}
There is a one-to-one correspondence 
$$\{\text{finite-dimensional irreducible $\mathbf{U}(\tilde{\lie p}_J)$-modules}\}/\sim \ \ \longrightarrow \mathcal{P}^+$$
\begin{proof} For a semi-simple Lie algebra $\lie{a}$ we denote by $V_{\lie a}(\lambda)$ the corresponding finite-dimensional irreducible $\lie a$ module of highest weight $\lambda$. By \cite[Theorem 5.15]{NSS}, Proposition~\ref{isotw} and \eqref{zentr} every finite-dimensional irreducible $\mathbf{U}(\tilde{\lie p}_J)$ module is isomorphic to a tensor product of evaluation representations and a one-dimensional representation of the following form

$$\mathbf{V}_{\underline{k},\underline{\lambda}}^{\mu, w}=\mathrm{ev}_{k_1}^* V_{\mathring{\lie g}}(\lambda_1)\otimes \cdots \otimes \mathrm{ev}_{k_r}^* V_{\mathring{\lie g}}(\lambda_r)\otimes \mathrm{ev}_{0}^* V_{[\mathring{\lie g}_0,\mathring{\lie g}_0]}(\mu)\otimes \mathbb{C}_{w}$$
where 
$$k_1,\dots,k_r\in\mathbb{C}^{\times},\ k^m_i\neq k^m_j \  \forall i\neq j,\  w\in \mathbb{C}^{n-|J|},\ \lambda_1,\dots,\lambda_r\in \mathring{P}^+, \mu\in \mathring{P}_0^+.$$ 
Recall the map $\iota:\mathring{P}^+\rightarrow \mathring{P}_0^+$ and set $\nu=\iota(\lambda_1+\cdots+\lambda_r)+\mu\in\mathring{P}_0^+$.
We define
$$\boldsymbol{\pi}^{\underline{k},\underline{\lambda}}_i(u)=\ \prod_{\substack{s=1}}^r (1-k_s^mu)^{\lambda_s(h_i)},\ i\in \mathring{I}$$
and note that $\mathrm{deg}(\boldsymbol{\pi}^{\underline{n},\underline{\lambda}}_{j}(u))\leq \nu(\gamma^{\vee}_{\bar{j}})$ for all $j\in  J$. Together with \cite[Theorem 5.5]{NSS} we obtain a well-defined map
$$\mathbf{V}_{\underline{k},\underline{\lambda}}^{\mu, w}\mapsto \left(\boldsymbol{\pi}^{\underline{k},\underline{\lambda}}_1(u),\dots,\boldsymbol{\pi}^{\underline{k},\underline{\lambda}}_n(u),\nu,w\right)\in\mathcal{P}^+.$$
Conversely, given $(\boldsymbol{\pi},\nu,w)\in\mathcal{P}_{}^+$ let $\{a_1,\dots,a_{\ell}\}$ be the inverses of the roots of
 $\boldsymbol{\pi}_{1}(u)\cdots \boldsymbol{\pi}_{n}(u)$. We choose for each $a_i$ an $m$-th root $b_i$, i.e. $b_i^m=a_i$, and set $\underline{k}_{\boldsymbol{\pi}}=(b_1,\dots,b_{\ell})$ and $\underline{\lambda}_{\boldsymbol{\pi}}=(\lambda_1,\dots,\lambda_{\ell})$ where $\lambda_s\in \mathring{P}^+$, $1\leq s\leq \ell$, is determined by
$$\lambda_s(h_i)=\text{ multiplicity of $a_s^{-1}$ in $\boldsymbol{\pi}_i(u)$},\ \ i\in \mathring{I}.$$

Furthermore, we define $\mu\in \mathring{P}_0^{+}$ by
$$\mu(\gamma^{\vee}_j)=\nu(\gamma^{\vee}_j)-\mathrm{deg}(\boldsymbol{\pi}_{\bar{j}}(u)),\ \ j\in J.$$
A simple verification shows that the assignment $(\boldsymbol{\pi},\nu,w)\mapsto \mathbf{V}_{\underline{k}_{\boldsymbol{\pi}},\underline{\lambda}_{\boldsymbol{\pi}}}^{\mu,w}$ defines the inverse map.
\end{proof}
\end{thm}
\begin{example} We continue discussing Example~\ref{section31}(2). The centre is trivial since $n=|J|$. So the finite--dimensional irreducible representations are parametrized by elements of the form 
$$(\boldsymbol{\pi}_1(u),\boldsymbol{\pi}_2(u),\boldsymbol{\pi}_3(u),\mu),\ \ \boldsymbol{\pi}_i(0)=1,\ i=1,2,3,\ \mu\in\mathring{P}^+_0$$ such that
$$\mathrm{deg}(\boldsymbol{\pi}_1(u))\leq \mu (\gamma_2^{\vee}),\ \mathrm{deg}(\boldsymbol{\pi}_2(u))\leq \mu (\gamma_1^{\vee}),\ \mathrm{deg}(\boldsymbol{\pi}_2(u))+\mathrm{deg}(\boldsymbol{\pi}_3(u))\leq \mu (\gamma_0^{\vee})$$
\end{example}
The aim of the article is to develop a quantization of the above setting.
%%%%%%%%%%%%%%%%
 \section{Parabolic quantum affine algebras: Drinfeld-Jimbo picutre}\label{section4}
 The aim of this section is to study quantized universal enveloping algebras of parabolic subalgebras. We first introduce them and calculate their classical limit. 
 \subsection{} For an indeterminate $q$ we denote by $\mathbb{C}(q)$ the ring of rational functions. Recall that the Cartan matrix $A$ is symmetrizable, meaning that there exists a diagonal matrix $B=\mathrm{diag}(\epsilon_0,...,\epsilon_n)$, where $\epsilon_i \in \mathbb{N}$ are relatively prime positive integers, such that $BA$ is symmetric. We set $q_i := q^{\epsilon_i}$ and $[a]_{q_i}:=[a]_i$ for simplicity. Consider the symmetric bilinear form on the affine root lattice $Q$ given by 
$$(\alpha_i,\alpha_j)^{\dagger}=\epsilon_{i}a_{i,j}\ (i,j=0,\dots,n)$$
and observe that it coincides with the lacing number times the restriction of the bilinear form introduced in Section~\ref{section2}. We denote the quantum groups of $\lie g$ and $\mathring{\lie g}$ respectively by $\mathbf{U}_q$ and $\mathring{\mathbf{U}}_q$ respectively and refer the reader to \cite{CP94b,L10a} for a precise definition.  The Hopf-algebra structure on $\mathbf{U}_q$ is given as follows
$$\Delta(K_i^{\pm})=K_i^{\pm}\otimes K_i^{\pm},\ \  \Delta(D^{\pm})=D^{\pm}\otimes D^{\pm}$$
$$\Delta(E_i)=E_i\otimes 1+K_i\otimes E_i,\ \ \Delta(F_i)=F_i\otimes K^{-}_i+1\otimes F_i$$
%$$\epsilon(K_i^{\pm})=1,\ \ \epsilon(D^{\pm})=1,\ \ \epsilon(E_i)=0,\ \ \epsilon(F_i)=0$$
$$S(K_i^{\pm})=K_i^{\mp},\ \ S(D^{\pm})=D^{\mp},\ \ S(E_i)=-K^{-}_iE_i,\ \ S(F_i)=-F_iK_i.$$
Moreover, $\mathbf{U}_q$ admits a $Q$-grading $\mathbf{U}_q=\bigoplus_{\eta\in Q}\mathbf{U}_{q,\eta}$ determined by
$$E_i\in \mathbf{U}_{q,\alpha_i},\ \ F_i\in \mathbf{U}_{q,-\alpha_i},\ \ D^{\pm},K_i^{\pm}\in \mathbf{U}_{q,0},\ \forall i\in I$$
and an automorphism $\Phi$ and anti-automorphism $\Omega$ defined on the generators as follows
$$\Phi(E_i)=F_i,\ \ \Phi(F_i)=E_i,\ \ \Phi(K^{\pm}_i)=K^{\pm}_i,\ \Phi(q)=q^{-1},\ \ \Phi(D)=D$$
$$\Omega(E_i)=F_i,\ \ \Omega(F_i)=E_i,\ \ \Omega(K^{\pm}_i)=K^{\mp}_i,\ \Omega(q)=q^{-1},\ \ \Omega(D)=D^{-}$$
%$$\Omega'(E_i)=F_i,\ \ \Omega'(F_i)=E_i,\ \ \Omega'(K^{\pm}_i)=K^{\pm}_i,\ \Omega'(q)=q,\ \ \Omega'(D)=D^{}.$$
For any representation $V$ of $\mathbf{U}_q$  we denote by $^{\Phi}V$ the representation obtained by twisting the action by $\Phi$. 
 \begin{defn}\label{maindefDJ}
 The \textit{parabolic quantum affine algebra} $\mathbf{U}_q^J$ is the $\mathbb{C}(q)$-subalgebra of $\mathbf{U}_q$  generated by the homogeneous elements 
\begin{equation}\label{setofgen}E_i,F_j,K_i^{\pm}, D^{\pm}, C^{\pm 1/2},\ \ \  i\in I, \ j\in J.\end{equation} 
Similarly, we denote by $\mathring{\mathbf{U}}^J_q$ the subalgebra generated by $E_j,F_j,K_j^{\pm},\  j\in J$.
 \end{defn}
We follow \cite[Section 1.3]{B94} and add the square root of the canonical central element as a generator. Note that by definition $\mathbf{U}_q^J$ and $\mathring{\mathbf{U}}_q^J$ are both Hopf subalgebras by restricting the unit, counit, antipode, multiplication and comultiplication. Moreover, some further properties of the quantum affine algebra $\mathbf{U}_q$ are inherited to the parabolic case. We denote by $\mathbf{U}_q^{J,+}$ (respectively $\mathbf{U}_q^{J,-}$)  the subalgebra of $\mathbf{U}_q^J$ generated by the elements $E_i$ for $i \in I$ (respectively $F_j$ for $j\in J$) and by $\mathbf{U}_q^{J,0}$ the subalgebra generated by the elements $K_i^{\pm}, D^{\pm}$ and $C^{\pm 1/2}$ for $i \in I$. Then we have a triangular decomposition: 
$$\mathbf{U}_q^J \cong \mathbf{U}_q^{J,-}  \otimes \mathbf{U}_q^{J,0} \otimes \mathbf{U}_q^{J,+}.$$
The algebras $\mathbf{U}_q^{\pm}$ and $\mathbf{U}_q^{0}$ are defined in the obvious way. Similarly, the parabolic quantum affine algebra is $Q_J\hspace{0,03cm} \sqcup \hspace{0,03cm} Q_{I\backslash{J}}^+$-graded where
%$$\mathbf{U}_q^J=\bigoplus_{\alpha\in Q_J\hspace{0,03cm} \sqcup \hspace{0,03cm} Q_{I\backslash{J}}^+} (\mathbf{U}_q^J)_{\alpha},\ \  (\mathbf{U}_q^J)_{\alpha}=\{u\in \mathbf{U}_q^J: K_iuK_i^{-}=q_i^{\alpha(h_i)}u,\ DuD^{-}=q^{\alpha(d)}u, \forall i\in I\}$$
we understand $Q_A=\sum_{i\in A} \mathbb{Z} \alpha_{i}$ for any subset $A\subseteq I$.
The same properties hold for $\mathring{\mathbf{U}}^J_q$ with obvious modifications. 
\subsection{}\label{section42}
The following justifications and calculations are standard and we give them only for the readers convinience. Instead of defining $\mathbf{U}_q^{J}$ as a subalgebra in $\mathbf{U}_q$ by the set of generators \eqref{setofgen} we can consider the algebra freely generated by the elements \eqref{setofgen} subject to the following defining relations. For $i,i'\in I$ and $j,j'\in J$ we have  
$$C^{\pm 1/2} \text{  is central $C^{\pm}=K_{\delta}^{\pm}$ and }\  [K_{i},K_{i'}] =[K_i,D]=0$$ 
$$K_iK^-_i = K^-_iK_i = C^{1/2} C^{-1/2}=D D^-=D^{-}D=1,\ \  [E_i,F_j]= \delta_{i,j}[K_i;0]$$
$$K_iE_{i'}K_i^- = q_i^{a_{ii'}}E_{i'}, \    K_iF_{j}K_i^- = q_i^{-a_{ij}}F_{j},\  DE_{i'}D^- = q^{\delta_{0,i'}}E_{i'}, \   DF_{j}D^- = q^{-\delta_{0,j}}F_{j}$$
$$\displaystyle \sum_{s=0}^{1-a_{ii'}}(-1)^s E_i^{(1-a_{ii'}-s)}E_{i'}E_i^{(s)} = 0,\ \ i\neq i',\ \sum_{s=0}^{1-a_{jj'}}(-1)^s F_j^{(1-a_{jj'}-s)}F_{j'}F_j^{(s)} = 0, \ \   j\neq j'$$
where $$X^{(k)}=1/[k]_{i}! X^k,\ \ K_{\alpha}=\prod_{i=0}^n K_i^{r_i},\ \text{if }\alpha=\sum_{i=0}^nr_i\alpha_i\in Q,\ \ [K_i;a]=\dfrac{q_i^aK_i-q_i^{-1}K_i^-}{q_i-q_i^{-1}}.$$ Let us denote this algebra by $(\mathbf{U}_q^{J})^{\dagger}$ and similarly we define $(\mathring{\mathbf{U}}_q^{J})^{\dagger}$. Note that, as for ordinary quantum groups, $(\mathbf{U}_q^{J})^{\dagger}$ admits a triangular decomposition and is graded as above.  We clearly have  algebra homomorphisms 
$$\begin{tikzpicture}
    \node (A) at (2,2) {$(\mathring{\mathbf{U}}_q^{J})^{\dagger}$};
     \node (B) at (4,2) {$\mathring{\mathbf{U}}_q^{J}$};
      \node (C) at (6,2) {$\mathbf{U}_q^{J}$};
      \node (F) at (8,2) {$\mathbf{U}_q$};
       \node (E) at (4,1) {$(\mathbf{U}_q^{J})^{\dagger}$};
       
\draw[dotted,->, thick] (A) to[out=35, in=145] node[midway,below]{$\mathring\varphi$} (F);
\draw[dotted,->, thick] (E) to[out=-15, in=-145] node[midway,below]{$\varphi$}(F);
         \draw[thick,->>] (A) -- (B) node[midway,above]{};
          \draw[thick,->>] (E) -- (C) node[midway,left,rotate=-5] {};
           \draw[thick,->] (A) -- (E) node[midway,left,rotate=0] {};
            \draw[thick,right hook-latex] (C) -- (F) node[midway,left,rotate=0] {};
            \draw[thick,right hook-latex] (B) -- (C) node[midway,left,rotate=0] {};
    \end{tikzpicture}
    $$
and all diagrams commute. Now the standard trick shows in fact that $\varphi$ (similarly $\mathring{\varphi}$) is injective and therefore both definitions coincide; we will explain only the idea below. Each $\mathbf{U}_q$ module $V$ can be considered as a module for $(\mathbf{U}_q^{J})^{\dagger}$ and $(\mathring{\mathbf{U}}_q^{J})^{\dagger}$ respectively by pulling back $V$ via $\varphi$ and $\mathring \varphi$ respectively; denote the corresponding modules by $\varphi^{*} V$ and $\mathring{\varphi}^{*} V$. Now the same proof as \cite[Proposition 5.11]{J96ab} shows that there is no non-zero element $u\in (\mathbf{U}_q^{J})^{\dagger}$ acting trivially on each representation of the form $\varphi^{*}(V^{\Phi}_q(\lambda)\otimes V_q(\lambda'))$ for all $\lambda,\lambda'\in P^+$ where $V_q(\lambda)$ denotes the irreducible highest weight module for $\mathbf{U}_q$ of highest weight $\lambda\in P^+$. This will give the injectivity.
 %K_i=q_i^{h_i}
\subsection{} In this subsection we calculate the classical limit $\tilde{\mathbf{U}}^J_{1}$ of $\mathbf{U}_q^J$ and show that a suitable quotient $\mathbf{U}^J_{1}$ is isomorphic to $\mathbf{U}(\mathfrak{p}_J)$ as Hopf algebras, thus justifying the definition made in Definition~\ref{maindefDJ}. Again the calculations in this subsection are standard and mainly follow the ideas from \cite[Section 3]{HK02} and \cite[Section 1]{KdC90}. We will only sketch the ideas. We denote by $\mathbf{A}_1$ be the localization of $\mathbb{C}[q^{\pm}]$ by the ideal $(q-1)$ and let $\mathbf{U}_{q,\mathbf{A}_1}^J$ be the $\mathbf{A}_1$ subalgebra of $\mathbf{U}_q^J$ generated by the elements $E_i,F_j,K_i^{\pm}, C^{\pm 1/2}, D^{\pm 1}$ and
$$H_i := \frac{K_i-K_i^-}{q_i-q_i^{-1}},\ \ c'= \frac{C-C^-}{q-q^{-1}}, \ \ d'= \frac{D-D^-}{q-q^{-1}}$$
where $i\in I,j\in J$. Now taking the classical limit of an object refers to considering its image under the canonical projection 
 $$ \mathbf{U}_{q,\mathbf{A}_1}^J \rightarrow  \mathbf{U}_{q,\mathbf{A}_1}^J\otimes_{\mathbf{A}_1}  \mathbf{A}_1/(q-1)\mathbf{A}_1\cong \mathbf{U}_{q,\mathbf{A}_1}^J/(q-1)\mathbf{U}_{q,\mathbf{A}_1}^J=:\tilde{\mathbf{U}}^J_{1}$$
Similarly as above we can define $\mathbf{U}_{q,\mathbf{A}_1}$ and $\tilde{\mathbf{U}}_1$. Note that the identity map $\mathbf{U}_{q,\mathbf{A}_1}^J\hookrightarrow \mathbf{U}_{q,\mathbf{A}_1}$ induces an embedding $\tilde{\mathbf{U}}^J_1\hookrightarrow \tilde{\mathbf{U}}_1$ and set
$$\mathbf{U}_1^J:=\tilde{\mathbf{U}}_1^J/(K_i-1,D-1,C^{\pm 1/2}-1).$$
Similarly we define $\mathbf{U}_1$ and recall that we have an isomorphism $\mathbf{U}(\lie g)\cong\mathbf{U}_1$ as algebras observed in \cite[Proposition 1.5]{KdC90} (see also \cite[Proposition 2.1]{B94}). This isomorphism is induced by checking certain relations in $\tilde{\mathbf{U}}_1$. Among them are the following ones which already hold in $\tilde{\mathbf{U}}^J_1$. For $i,i'\in I$ and $j,j'\in J$ we have
 \begin{align*}
 &c' \text{ is central and $K_i$ are central elements with $K_i^2=1$ }&\\&
 [E_i,F_j] = \delta_{i,j}H_i, \: [H_i,E_{i'}] = a_{i,i'} K_i E_{'},\: [H_i, F_j] = -a_{i,j}K_iF_j, &\\&
 [d',E_{i'}]=\delta_{i,0} DE_{i'},\ \ [d',F_j]=-\delta_{j,0} DF_j,\ \ D^2=1,\ \ C^2=1 &\\&
 [E_i,\dots,\left[E_i,\left[E_i,E_{i'}\right]\right]\cdots]=0, \ \ i\neq i' &\\&
[F_j,\dots ,\left[F_j,\left[F_j,F_{j'}\right]\right]\cdots]=0, \ \ j\neq j'
 \end{align*}
where the number of occurences of $E_i$ respectively $F_j$ is $(1-a_{ii'})$ and $(1-a_{jj'})$ respectively. It is straightforward to see that the above relations transform to the defining relations of $\mathbf{U}(\lie p_J)$ under the specialization made by passing from $\tilde{\mathbf{U}}^J_1$ to $\mathbf{U}^J_1$. We can derive the following proposition justifying our definition.
 \begin{prop}\label{classli1} We have that $\tilde{\mathbf{U}}^J_1$ is an associative algebra over $\mathbb{C}\cong \mathbf{A}_1/(q-1)\mathbf{A}_1$ generated by $E_i$,$F_j$,$K_i$,$H_i$, $C^{\pm 1/2}$, $D$, $c',d'$ with $i\in I$ and $j\in J$. Moreover, the following map
 $$\mathbf{U}(\mathfrak{p}_J)\rightarrow \mathbf{U}_1^J$$
 $$x_i^{+}\mapsto E_i,\ \ x_j^{-}\mapsto F_j,\ \ h_i\mapsto H_i,\ \ d\mapsto d',\ \  i\in I, j\in J$$
 defines an isomorphism of Hopf algebras. 
\qed
 \end{prop}
\section{Parabolic quantum affine algebras: Drinfeld picture}\label{section5}
In this section we use the results of \cite{B94} to describe the image of a parabolic quantum affine algebra under the Drinfeld-Beck isomorphism. 
\subsection{}\label{section51} For $i\in I$ let $\mathbf{T}_i: \mathbf{U}_q\rightarrow \mathbf{U}_q$ the Lusztig isomorphism defined on the generators of $\mathbf{U}_q$ as follows. For $j\in I$, we have
$$\mathbf{T}_i(E_i)=-F_iK_i,\ \ \mathbf{T}_i(E_j)=\sum_{s=0}^{-a_{ij}}(-1)^{s-a_{ij}} q_i^{-s} E_i^{(-a_{ij}-s)}E_j E_i^{(s)},\ \ \text{ for $i\neq j$}$$
$$\mathbf{T}_i(F_i)=-K^{-}_iE_i,\ \ \mathbf{T}_i(F_j)=\sum_{s=0}^{-a_{ij}}(-1)^{s-a_{ij}} q_i^{s} F_i^{(s)}F_j F_i^{(-a_{ij}-s)},\ \ \text{ for $i\neq j$}$$
$$\mathbf{T}_i(K_j)=K_{j}(K^{-}_i)^{a_{ij}},\ \ \mathbf{T}_i(D)=DK_{i}^{-\delta_{0,i}}$$
Recall that these operators commute with $\Omega$ and satisfy the braid relations \cite{L10a}, hence the definition $\mathbf{T}_{w}=\mathbf{T}_{i_1}\cdots \mathbf{T}_{i_k}$ for a reduced expression $w=\mathbf{s}_{i_1}\cdots \mathbf{s}_{i_k}$, $w\in W$ is well-defined.  This induces an action of the braid group associated to $W$ on $\mathbf{U}_q$. We extend this definition to an action of the braid group of $W^{\mathrm{ext}}$ by setting $$\mathbf{T}_{\sigma}(E_i)=E_{\sigma(i)},\ \ \mathbf{T}_{\sigma}(F_i)=F_{\sigma(i)},\ \ \mathbf{T}_{\sigma}(K_i)=K_{\sigma(i)},\ \ \ \sigma\in \mathcal{T}.$$ The next proposition can be deduced from \cite[Propositon 1.8]{Lu90a} and \cite[Lemma 40.1.2]{L10a} respectively, and will be used throughout the remainder of the paper without further comment.
\begin{prop}\label{spanlem}
Let $w\in W$ and $i\in I$ such that $w(\alpha_i)\in \sum_{j\in J}\mathbb{Z}_+\alpha_j$. Then $\mathbf{T}_w(E_i)$ lies in the subalgebra generated by $E_j, j\in J$. Moreover, if $w(\alpha_i)=\alpha_k$ for some $k\in I$, then $\mathbf{T}_w(E_i)=E_k$.
\hfill\qed
\iffalse 
\begin{proof}
Since $W$ is a Coxeter group we can choose $k\in I$ and $w'\in W$ such that $w=w'y$ ($|w'y|=|w'|+|y|$) with 
$|w's_k|=|w'|+1$, $|w's_i|=|w'|+1$ and
$y=s_is_k\cdots s_k$ or $y=s_ks_i\cdots s_k$
Since $|ws_i|=|w|+1$ the number of factors in $y$ is less than $h(i,k)$ (noch nicht eingeführt). By \cite[Lemma 40.1.1]{L94} $\mathbf{T}_y(E_i)$ lies in the subalgebra generated $E_i$ and $E_k$. Hence $\mathbf{T}_y(F_i) = \Omega(\mathbf{T}_y(E_i))$ lies in the subalgebra generated by $F_i = \Omega(E_i)$ and $F_k= \Omega(E_k).$ So in order to prove the lemma by induction on $|w|$ we only have to observe $w'(\alpha_i),w'(\alpha_k)\in \sum_{j\in J}\mathbb{Z}_+\alpha_j$. Note that 
Since $|w's_k|=|w'|+1$, $|w's_i|=|w'|+1$ and $|ys_i|=|y|+1$ we can write $$w(\alpha_i)=w'y(\alpha_i)=aw'(\alpha_i)+bw'(\alpha_k),\ a,b\geq 0,\ w'(\alpha_i),w'(\alpha_k)\in \sum_{i=0}^n\mathbb{Z}_+\alpha_i.$$
If $a=0$ or $b=0$ we have $w(\alpha_i)=w'(\alpha_{\ell})$ and $\mathbf{T}_w(E_i)=\mathbf{T}_{w'}(E_{\ell}), \ \mathbf{T}_w(F_i)=\mathbf{T}_{w'}(F_{\ell})$ (for some $\ell\in\{i,k\}$ and the claim follows by induction). So assume $a,b>0$. On the other hand $w(\alpha_i)\in \sum_{j\in J}\mathbb{Z}_+\alpha_j$ and since the simple roots are linearly independent we get the desired claim.
\end{proof}
\fi 
    \end{prop}
In the rest of this section we choose a function $\kappa : \mathring{I}\rightarrow \{\pm 1\}$ such that two adjacent nodes  $r,s\in \mathring{I}$ in the Dynkin diagram of $\mathring{\lie g}$ have different signs. i.e. $\kappa(r)=-\kappa(s)$. We extend this map by 
$$\kappa(\alpha)=\kappa(1)^{c_1}\cdots \kappa(n)^{c_n},\ \ \alpha=\sum_{i=1}^nc_i\alpha_i\in \mathring{Q}^+.
$$Let $\mathbf{u}^{\vee}=\varpi_{u_1}^{\vee}+\cdots+\varpi_{u_p}^{\vee}\in \mathring{Q}^{\vee}$ (repetition in the indices is allowed) be an element such that $(\mathbf{u}^{\vee},\alpha_i)>0$ for $i=1,\dots,n$ and fix a reduced expression  \begin{equation}\label{seteq0}t_{\nu(\mathbf{\mathbf{u}^{\vee}})}=\mathbf{s}_{i_1}\cdots \mathbf{s}_{i_{N}}\in W.\end{equation}
We get from \cite[Proposition 8]{Pa95a}) (choosing $A=\emptyset$ in their notation) the following lemma; see also \cite[Theorem 5]{Pa95a}.
\begin{lem}\label{papi} 
Each positive real root for the affine Lie algebra is obtained as
\begin{equation*}\label{seteq1}\{-\alpha+r\delta: \alpha\in \mathring{R}^+,\ r\geq 1\}=\{\beta_k:=\mathbf{s}_{i_1}\mathbf{s}_{i_2}\cdots \mathbf{s}_{i_{k-1}}(\alpha_{i_k}),\ k>0 \}\end{equation*}
\begin{equation*}\label{seteq2}\{\alpha+r\delta: \alpha\in \mathring{R}^+, r\geq 0\}=\{\beta_k:=\mathbf{s}_{i_0}\mathbf{s}_{i_{-1}}\cdots \mathbf{s}_{i_{k+1}}(\alpha_{i_k}),\ k\leq  0\}\end{equation*}
where the indices are understood modulo $N$.
\qed
\end{lem}
This definition allows a total order on the set of positive roots
\begin{equation}\label{convexordö}\beta_0<\beta_{-1}<\beta_{-2}<\cdots<2\delta <\delta<\cdots<\beta_2<\beta_1\end{equation}
and the order is convex in the sense that $\alpha<\alpha+\beta<\beta$ for all $\alpha,\beta,\alpha+\beta\in R^+$ with $\alpha$ real. 
\begin{rem}\label{remwicht} We briefly summarize a few properties, although they are straightforward. 
\begin{enumerate}
\item [(i)] If $\beta_k= -\alpha+p\delta$ and $\beta_{\ell} = -\alpha + r\delta$ for some $0<k<\ell$ then we have $p<r$. This is easy to see since $p\geq r$ gives
$$(r-p)\delta  = \mathbf{s}_{i_1}\cdots \mathbf{s}_{i_k}\left (  \mathbf{s}_{i_{k+1}}\cdots \mathbf{s}_{i_{\ell-1}}(\alpha_{i_{\ell}})+\alpha_{i_k} \right )$$
$$\implies
-\alpha_{i_k}+ (r-p)\delta =  \mathbf{s}_{i_{k+1}}\cdots \mathbf{s}_{i_{\ell-1}}(\alpha_{i_{\ell}}).$$
and the left hand side is a negative root. However the right hand side is a positive root since $\mathbf{s}_{i_{k+1}}\cdots \mathbf{s}_{i_{\ell}}$ is a reduced word and we get a contradiction.
    \item [(ii)] If $\alpha\in R^+$ is such that $$\mathbf{s}_{i_1}\cdots\mathbf{s}_{i_{k-1}}(\alpha)=-\beta+s\delta\in (-\mathring{R}^+ +\mathbb{Z}_+\delta)$$
we must have $s\in\mathbb{N}$ and $-\beta+s\delta\leq \beta_k$. This is because $s=0$ would give that $\alpha = \mathbf{s}_{i_{k-1}}\cdots \mathbf{s}_{i_{\ell+1}}(\alpha_{i_\ell}) $ for some $\ell\in\{1,\dots,k-1\}$ (see, for example \cite[pg. 115 Exercise 1]{Hu90}) and hence $\beta=\mathbf{s}_{i_1}\cdots \mathbf{s}_{i_{\ell-1}}(\alpha_{i_\ell})$
which is impossible by Lemma~\ref{papi}. The second part of the claim can be seen similarly.
\item [(iii)] If $\beta_k<\beta_{\ell}$ for $k,\ell >0$, then $\mathbf{s}_{i_{k-1}}\cdots \mathbf{s}_{i_{1}}(\beta_{\ell})\in -R^+,\ \mathbf{s}_{i_{\ell-1}}\cdots \mathbf{s}_{i_{1}}(\beta_{k})\in R^+.$
\end{enumerate}
\end{rem}
Following \cite[Section 1.6]{BK96} we introduce the quantum root vectors:
\begin{align*}\chi^{+}_{\alpha,r}&= \kappa(\alpha)^r \ \mathbf{T}_{i_0}^{-1}\cdots \mathbf{T}^{-1}_{i_{k+1}}(E_{i_k}), \ \ \text{if } \beta_k=\alpha+r\delta,\  k\leq 0\\
\chi^{-}_{\alpha,r}&=\kappa(\alpha)^r \ \mathbf{T}_{i_1}\cdots \mathbf{T}_{i_k}(F_{i_k}), \ \ \text{if } \beta_k=-\alpha+r\delta ,\  k>0\end{align*}
$$\chi^{-}_{\alpha,-r}=\Omega(\chi^{+}_{\alpha,r}),\ \ \chi^{+}_{\alpha,-r}=\Omega(\chi^{-}_{\alpha,r}).$$
Note that if $\alpha=\alpha_i$ is a simple root then we have 
$$\chi^+_{i,r}=\kappa(i)^r \ \mathbf{T}^{-r}_{\varpi_i}(E_i),\ \chi^-_{i,r}=\kappa(i)^r\ \mathbf{T}^r_{\varpi_i}(F_i),\ r\in\mathbb{Z},\  \text{ where }\mathbf{T}_{\varpi_i}:=\mathbf{T}_{t_{\nu(\varpi_i^{\vee})}},\ \ t_{\nu(\varpi_i^{\vee})}\in W^{\mathrm{ext}}.$$
Moreover, define recursively for $s\in\mathbb{N}$: 
$$sh_{i,s}:=s\bar{\Psi}_{i,s}-(q_i-q_i^{-1})\sum_{k=1}^{s-1}k\cdot \bar{\Psi}_{i,s-k}\cdot h_{i,k},\ \ \bar{\Psi}_{i,r}=\kappa(i)^rC^{r/2}K_i^{-1}\big[E_i,\mathbf{T}_{\varpi_i}^r(F_i)\big]$$
and let $h_{i,-s}=\Omega(h_{i,s}).$
%$$\kappa(i)^sh_{i,-s}:=C^{-s/2}\big[\mathbf{T}_{\varpi_i}^s(E_i),F_i\big]K_i +\frac{(q_i-q_i^{-1})}{s}\sum_{k=1}^{s-1}\kappa(i)^{k} k\cdot h_{i,-k}\cdot C^{-(s-k)/2}\big[\mathbf{T}_{\varpi_i}^{s-k}(E_i),F_i\big]K_i $$
The second realization of Drinfeld \cite{D872},\cite[Theorem 4.7]{B94} reads as follows.
\begin{thm}\label{definingrelationQAA} The quantum affine algebra $\mathbf{U}_q$ is generated over $\mathbb{C}(q)$ by the elements $$\chi_{i,r}^{\pm},\ h_{i,s},\ K_i^{\pm},\ C^{\pm 1/2},\ D, \ \ i\in\mathring{I},\ r\in \mathbb{Z},\ s\in\mathbb{Z}\backslash\{0\}$$ subject to the following defining relations $\big(i,j \in \mathring{I}, r,r' \in \mathbb{Z}, m,m' \in \mathbb{Z}\setminus \{ 0\} \big)$:
\begin{enumerate}
\item $C^{\pm 1/2}$ is central and $K_iK_{i}^{-}=1=K_i^{-}K_i$,\ $[K_i,K_j] = [K_i, h_{j,m}] = 0$\vspace{0,3cm}
\item $K_i \chi_{j,r}^\pm K_i^{-} = q_i^{\pm a_{ij}}\chi_{j,r}^\pm$, $D \chi_{j,r}^\pm D^{-1} = q^{r}\chi_{j,r}^\pm$, $D h_{j,m}D^{-1} = q^{m}h_{j,m}$ \vspace{0,3cm}
\item $\displaystyle [h_{i,m},h_{j,m'}]=\delta_{m,-m'}\frac{1}{m} [ma_{ij}]_{i}\frac{C^m-C^{-m}}{q_j-q_j^{-1}}$\vspace{0,3cm}
\item $\displaystyle [h_{i,m}, \chi_{j,r}^{\pm}]= \pm \frac{1}{m}[m a_{ij}]_{i}C^{\mp (|m|/2)} \chi_{j,r+m}^{\pm}$\vspace{0,3cm}
\item $\chi_{i,r+1}^{\pm} \chi_{j,r'}^{\pm} -\,q_i^{\pm a_{ij}}\chi_{j,r'}^\pm \chi_{i,r+1}^\pm 
  = q_i^{\pm a_{ij}}\chi_{i,r}^{\pm} \chi_{j,r'+1}^{\pm} - \chi_{j,r'+1}^{\pm}\chi_{i,r}^\pm$\vspace{0,3cm}
  \item   $\displaystyle [\chi_{i,r}^{+}, \chi_{j,r'}^{-}] = \delta_{ij} \frac{1}{q_i - q_i^{-1}}\big(C^{(r-r')/2}\phi_{i,r+r'}^+ - C^{(r'-r)/2}\phi_{i,r+r'}^-\big)$\vspace{0,3cm}
  \item  $\displaystyle\sum_{\tau \in \mathfrak{S}_{1-a_{ij}}} \sum_{k=0}^{1-a_{ij}} (-1)^k \begin{bmatrix} 1-a_{ij} \\ k \end{bmatrix}_{\!q_i} 
    \chi_{i, r_{\tau(1)}}^\pm \cdots \chi_{i, r_{\tau(k)}}^\pm \chi_{j,r'}^\pm \chi_{i,r_{\tau(k+1)}}^\pm \cdots \chi_{i,r_{\tau(1-a_{ij})}}^\pm
    = 0 \ \ (i \neq j)$
  for all sequences of integers $r_1,\dots,r_{1-a_{ij}},r'$

\end{enumerate}
where $\phi_{i,r}^\pm$'s are determined by equating coefficients of powers of $u$ in the formula
\[ \phi_i^{\pm}(u)=\sum_{r = 0}^{\infty} \phi_{i, \pm r}^\pm u^{\pm r} = K_i^{\pm} \exp \left( \pm(q_i - q_i^{-1})\sum_{m = 1}^{\infty}
   h_{i, \pm m} u^{\pm m} \right),
\]
and $\phi_{i,\mp r}^{\pm} = 0$ for $r > 0$.
\qed
\end{thm}
We have the relation 
$$x_{i,r}^{\pm}[K_j;a] = [K_j;a\mp a_{ji}]x_{i,r}^{\pm}.$$
\subsection{} The description of the comultiplication in the second realization appears to be rather complicated. Nevertheless, by setting $$X^+=\sum_{i\in \mathring{I}, s\in\mathbb{Z}} \mathbb{C}(q) \chi_{i,s}^+$$ we can derive the following expression from \cite{Da98a}.
\begin{thm}\label{Dami} Let $i\in\mathring{I}$ and $r\in\mathbb{Z}$. We have
    \begin{enumerate}
    \item $\Delta(\chi_{i,r}^+)\in\chi_{i,r}^+\otimes 1+ \mathbf{U}_q\otimes \mathbf{U}_q X^+$ \vspace{0,2cm} 
      \item $\Delta(\chi_{i,r}^-)\in\chi_{i,r}^-\otimes K_i+1\otimes \chi_{i,r}^- +\sum_{s=1}^{r-1} \chi_{i,r-s}^-\otimes \phi_{i,s}^+ +\mathbf{U}_q\otimes \mathbf{U}_q X^+,\ r>0$ \vspace{0,2cm} 
      
    \hspace{-0,5cm}  $\Delta(\chi_{i,r}^-)\in\chi_{i,r}^-\otimes K^{-}_i+1\otimes \chi_{i,r}^- +\sum_{s=1}^{-r} \chi_{i,r+s}^-\otimes \phi_{i,-s}^+ +\mathbf{U}_q\otimes \mathbf{U}_q X^+,\ r\leq 0$\vspace{0,2cm} 
      
      \item $\Delta(\phi_{i,\pm r}^{\pm})\in \sum_{s=0}^r \phi_{i,\pm s}^{\pm}\otimes \phi_{i,\pm(r-s)}^{\pm}+\mathbf{U}_q\otimes \mathbf{U}_q X^+,\ \ r>0$
\end{enumerate}
\qed
\end{thm}
\subsection{}\label{section52} Now, in this presentation, we aim to describe generators of $\mathbf{U}_q^J$ and determine a PBW basis as in \cite[Proposition 3]{B94A}. For this purpose it will be more convenient to work with the modified quantum root vectors 
\begin{align*}
\Dot{\chi}_{\alpha,r}^+ &:= \chi_{\alpha,r}^+ ,\ \Dot{\chi}_{\alpha,-r}^- := \Omega(\Dot{\chi}_{\alpha,r}^+),\ \ \alpha\in \mathring{R}^+, \ r\geq 0 \\
\Dot{\chi}_{\alpha,r}^- &:= -C^rK^{-}_{\alpha}\chi_{\alpha,r}^- ,\ \Dot{\chi}_{\alpha,-r}^+ := \Omega (\Dot{\chi}_{\alpha,r}^-),\ \alpha\in \mathring{R}^+, \ r >0 \\
\Dot{h}_{i,r}&:= C^{r/2}h_{i,r},\ \Dot{h}_{i,-r} := \Omega(\Dot{h}_{i,r}), \ r>0.
\end{align*}
\begin{lem}\label{lemsubset667} We have
$$\{\Dot{\chi}_{\alpha,r}^+,\ \Dot{\chi}_{\alpha,s}^-,\ \Dot{h}_{i,s} : \alpha\in \mathring{R}^+, \ r\in\mathbb{Z}_+,\ s\in\mathbb{N}\}\subseteq \mathbf{U}_q^{J,+}$$
$$\{\Dot{\chi}_{\alpha,-1}^+,\ \Dot{\chi}_{\beta,0}^- :  \alpha,\beta \in \mathring{R}^+, \ \mathrm{supp}(\alpha-\delta),\ \mathrm{supp}(\beta)\subseteq J\}\subseteq \mathbf{U}_q^{J,-}$$ 
\begin{proof}
This is an immediate consequence of Proposition~\ref{spanlem} and the definition of the (modified) quantum root vectors. For $\Dot{h}_{i,s}$ one can use $\Dot{h}_{i,1}=q_i^{-2}\Dot{\chi}_{i,0}^+\Dot{\chi}_{i,1}^--\Dot{\chi}_{i,1}^-\Dot{\chi}_{i,0}^+$ and the recursive formula to argue by induction.
\end{proof}
\end{lem}
For $u\in W^{\mathrm{ext}}$ (resp. $u\in \mathring{W}$) let 
$$R_{u}^+=\{\alpha\in R^+: u(\alpha)\in -R^+\},\ \ (\text{resp. }\mathring{R}_{u}^+=\{\alpha\in \mathring{R}^+: u(\alpha)\in -\mathring{R}^+\}).$$
The same proof of \cite[Proposition 2.8]{CSVW16} gives for all $u_1,u_2\in W^{\mathrm{ext}}\ (\text{resp. } u_1,u_2\in \mathring{W})$:
\begin{equation}\label{dcond1}
R_{u_2}^+\subseteq R_{u_1u_2}^+ ,\ \ (\text{resp. }\mathring{R}_{u_2}^+\subseteq \mathring{R}_{u_1u_2}^+)\implies \ell(u_1u_2)=\ell(u_1)+\ell(u_2).\end{equation}
\begin{prop}\label{hilflem} Let $\beta\in\mathring{R}^+$ and $w^{-1}_{\beta}\in \mathring{W}(\beta)$.
\begin{enumerate}
\item For all $i\in \mathring{I}\cap J$ we have $\mathbf{s}_iw_{\circ}=w_{\circ}\mathbf{s}_{\bar{i}}$ as Weyl group elements of $\mathring{\lie{g}}_0$.\vspace{0,2cm}
%\item We have
%$\ell(wt_{\nu(\varpi_i^{\vee})}) = \ell(w) + \ell(t_{\nu(\varpi_i^{\vee})})$ for all  $i\in\mathring{I}$ and $w\in \mathring{W}$.
\item  Let $u\in \mathring{W}$ such that $u(\beta)\in \mathring{P}^+$. Then we have  $\ell(uw_{\beta}) = \ell(u)+\ell(w_{\beta})$.\vspace{0,2cm}
\item Write 
\begin{equation}\label{thetakurz}-\theta_s+\delta=\alpha_{p_0}+\alpha_{p_1}+\cdots+\alpha_{p_{s+1}},\ 0=p_0<p_1<\cdots<p_{s+1}\leq n.\end{equation} \begin{enumerate}[(i)]\item Let $\beta$ be a short root, $u\in \mathring{W}_{\beta}$ and assume $-\beta+\delta=\beta_k$. Then 
$$\ell(\mathbf{s}_{p_{s}}\cdots\mathbf{s}_{p_0}u\mathbf{s}_{i_1}\cdots \mathbf{s}_{i_{k-1}})=\ell(u\mathbf{s}_{i_1}\cdots \mathbf{s}_{i_{k-1}})-s-1$$
unless $\mathring{\lie g}$ is of type $G_2$. \vspace{0,1cm}
\item For all dominant coweights $\lambda^{\vee}$ and $w_{\theta_s}^{-1}\in \mathring{W}(\theta_s)$ with $(w_{\theta_s}^{-1}(\theta_s),\lambda^{\vee})>0$ we have
$$\ell(\mathbf{s}_{p_{s}}\cdots\mathbf{s}_{p_0}w_{\theta_s}t_{\lambda^{\vee}})=\ell(w_{\theta_s}t_{\lambda^{\vee}})-s-1$$
unless $\mathring{\lie g}$ is of type $G_2$. \vspace{0,2cm}
\item If $\mathring{\mathfrak{g}}$ is of type $G_2$ and $\beta, u$ as in (i) then
$$\hspace{-6,5cm}\ell(\mathbf{s}_{1}\mathbf{s}_{0}u\mathbf{s}_{i_1}\cdots \mathbf{s}_{i_{k-1}})=$$
$$\hspace{3,5cm}\begin{cases}
%\ell(\mathbf{s}_{0}u\mathbf{s}_{i_1}\cdots \mathbf{s}_{i_{k-1}})-1=\ell(u\mathbf{s}_{i_1}\cdots \mathbf{s}_{i_{k-1}}) & \text{if} \ \beta = \alpha_2, u=\mathbf{s}_2\mathbf{s}_1 \\
\ell(\mathbf{s}_{0}u\mathbf{s}_{i_1}\cdots \mathbf{s}_{i_{k-1}})+1=\ell(u\mathbf{s}_{i_1}\cdots \mathbf{s}_{i_{k-1}}) &  \text{if} \ \beta \in\{\alpha_2, \alpha_1+\alpha_2\}\\
 \ell(\mathbf{s}_{0}u\mathbf{s}_{i_1}\cdots \mathbf{s}_{i_{k-1}})-1=\ell(u\mathbf{s}_{i_1}\cdots \mathbf{s}_{i_{k-1}})-2 & \text{if} \ \beta = \theta_s.\\
\end{cases}$$
Furthermore for all dominant coweights $\lambda^{\vee}$ as in (ii) we have
$$\ell(\mathbf{s}_1\mathbf{s}_0w_{\theta_s}t_{\lambda^{\vee}}) = \ell(\mathbf{s}_0w_{\theta_s}t_{\lambda^{\vee}})+1 =\ell(w_{\theta_s}t_{\lambda^{\vee}}). $$
\end{enumerate}

\end{enumerate}
\begin{pf} In order to prove (1) let $k\in \mathring{I}\cap J$. Then
$$\mathbf{s}_iw_{\circ}(\alpha_k)=-\alpha_{\bar{k}}+a_{i\bar{k}}\alpha_i=w_{\circ}(\alpha_k-a_{\bar{i}k}\alpha_{\bar{i}})=w_{\circ}\mathbf{s}_{\bar{i}}(\alpha_k)$$
and 
$$\mathbf{s}_iw_{\circ}(\alpha_0)=-\gamma_0+\delta+\gamma_0(\alpha_i^{\vee})\alpha_i=w_{\circ}(\alpha_0-a_{\bar{i}0}\alpha_{\bar{i}})=w_{\circ}\mathbf{s}_{\bar{i}}(\alpha_0).$$
Hence $\mathbf{s}_iw_{\circ}$ and $w_{\circ}\mathbf{s}_{\bar{i}}$ coincide on $\mathrm{span}\{\alpha_k: k\in J\}.$ We prove part (2) by induction on $\ell(w_{\beta})$ where the initial step is clear; so let $\ell(w_{\beta})>0$. So assume that the length additivity holds for all $\beta'\in \mathring{R}^+$ with $\ell(w_{\beta'})<\ell(w_{\beta})$, i.e. for all $u'\in \mathring{W}$ with $u'(\beta')\in \mathring{P}^+$ we have $\ell(u'w_{\beta'}) = \ell(u')+\ell(w_{\beta'})$.
Let $\alpha \in \mathring{R}_{w_{\beta}}^+$ such that $\alpha\notin \mathring{R}_{uw_{\beta}}^+$ (otherwise we are done with \eqref{dcond1}). We write 
$w_{\beta}(\alpha) = - \sum_{i\in \mathring{I}}r_i\alpha_i$ for some non-negative integers $r_i \geq 0.$
Now applying $u$ yields
\begin{align*}
uw_{\beta}(\alpha)&= -\sum_{i\in I(u)}r_{i}u(\alpha_{i}) -\sum_{i \in \mathring{I}\setminus I(u)}r_iu(\alpha_i).
\end{align*}
If $r_i\neq 0$ for some $i\notin I(u)$ we have $uw_{\beta}(\alpha)\in -\mathring{R}^+$ since $u$ involves only reflections corresponding to nodes in $I(u)$. However, this is impossible since $\alpha\notin \mathring{R}_{uw_{\beta}}^+$. Hence 
$$\alpha=-\sum_{i\in I(u)}r_i w_{\beta}^{-1}(\alpha_i)\in \mathring{R}^+$$
and there exists a node $k \in I(u)$ such that $w_{\beta}^{-1}(\alpha_{k})$ is a negative root, i.e. $\ell(\mathbf{s}_{k}w_{\beta})=\ell(w_{\beta})-1$. We write
$$w_{\beta}=\mathbf{s}_{k}w',\ \ \ell(w_{\beta})=\ell(w')+1$$
and Corollary~\ref{corhilf4q} gives $\beta(\alpha_k^{\vee})=u(\beta)(u(\alpha_k^{\vee}))>0$. This is only possible if $u(\alpha_k)$ is a positive root since $u(\beta)$ is dominant. Hence $\ell(u\mathbf{s}_k)=\ell(u)+1$ and together with our induction hypothesis we obtain 
$$\ell(uw_{\beta})=\ell((u\mathbf{s}_k)w')=\ell(u\mathbf{s}_k)+\ell(w')=\ell(u)+\ell(w')+1=\ell(u)+\ell(w_{\beta}).$$
For part (3)(i) and (3)(ii) we shall show that
\begin{equation}\label{toshow11}t_{-\lambda^{\vee}}w_{\theta_s}^{-1}\mathbf{s}_{p_{0}}\cdots \mathbf{s}_{p_{e}}(\alpha_{p_{e+1}}),\ \mathbf{s}_{i_{k-1}}\cdots \mathbf{s}_{i_{1}}u^{-1}\mathbf{s}_{p_{0}}\cdots \mathbf{s}_{p_{e}}(\alpha_{p_{e+1}})\in -R^+,\ \ -1\leq e\leq s-1.\end{equation}
First note that
$$-\theta_s+\delta=\begin{cases}
    \alpha_0+\alpha_1,&\ \text{if $\mathring{\lie g}$ is of type $C_n$}\\
    \alpha_0+\alpha_2+\cdots+\alpha_n,&\ \text{if $\mathring{\lie g}$ is of type $B_n$}\\
    \alpha_0+\alpha_1+\alpha_2+\alpha_3,&\ \text{if $\mathring{\lie g}$ is of type $F_4$}%\\
  %  \alpha_0+\alpha_1+\alpha_2,&\ \text{if $\mathring{\lie g}$ is of type $G_2$}\\
\end{cases}\ \ \ \ \ \ \theta_s=\begin{cases}
    \varpi_2,&\ \text{if $\mathring{\lie g}$ is of type $C_n$  %or $G_2$
    }\\
    \varpi_1,&\ \text{if $\mathring{\lie g}$ is of type $B_n$  
    }\\
    \varpi_4,&\ \text{if $\mathring{\lie g}$ is of type $F_4$}\\
\end{cases}$$
From the explicit formulas for $-\theta_s+\delta$ and $\theta_s$ we see on the one hand
$$\mathbf{s}_{p_{0}}\cdots \mathbf{s}_{p_{e}}(\alpha_{p_{e+1}})=\alpha_{p_0}+\cdots+\alpha_{p_{e+1}},\ \ \mathbf{s}_{p_{e+1}}\cdots\mathbf{s}_{p_{s}}(\alpha_{p_{s+1}})=\alpha_{p_{e+1}}+\cdots+\alpha_{p_{s+1}},\ \ -1\leq e\leq s-1$$
and on the other hand
\begin{equation}\label{g2prob1}\tau_x:=x^{-1}(\alpha_{p_{e+2}}+\cdots+\alpha_{p_{s+1}})\in\mathring{R}^+,\ \ -1\leq e\leq s-1,\ \ x\in \{u,w_{\theta_s}\}.\end{equation}
To see \eqref{g2prob1} suppose that there exists $q\in\{e+2,\dots,s+1\}$ such that $x^{-1}(\alpha_{p_q})\in -\mathring{R}^+$ (in particular $\ell(s_{p_q}x)<\ell(x)$). Then
$$s_{p_q}u(\beta)=s_{p_q}(\theta_s)=\theta_s,\ \ (s_{p_q}w_{\theta_s})^{-1}(\theta_s)=w^{-1}_{\theta_s}(\theta_s)$$
which contradicts the minimality of $u$ and $w_{\theta_s}$ respectively. Hence \eqref{g2prob1} is obtained and thus
\begin{align}x^{-1}\mathbf{s}_{p_{0}}\cdots \mathbf{s}_{p_{e}}(\alpha_{p_{e+1}}) &\notag = x^{-1}(\alpha_{p_0}+\cdots+\alpha_{p_{e+1}})=-x^{-1}(\theta_s+\alpha_{p_{e+2}}+\cdots+\alpha_{p_{s+1}})+\delta&\\&\label{2409uuh} =-(x^{-1}(\theta_s)+\tau_x)+\delta\in -\mathring{R}^++\mathbb{N}\delta.\end{align}
This shows, choosing $x=u$, that $u^{-1}\mathbf{s}_{p_{0}}\cdots \mathbf{s}_{p_{e}}(\alpha_{p_{e+1}})=\beta_{\ell}$ for some $\ell>0$ and the definition of the order guarantees $\tau_u<\beta_{\ell}$. Now, the convexity together with Remark~\ref{remwicht}(iii) gives the second part of \eqref{toshow11}:
$$\tau_u<\beta_{\ell}\implies \tau_u<\beta_k=\tau_u+\beta_{\ell}<\beta_{\ell}\implies \mathbf{s}_{i_{k-1}}\cdots \mathbf{s}_{i_{1}}(\beta_{\ell})\in -R^+.$$
For the first part of \eqref{toshow11} we choose $x=\theta_s$ and apply $t_{-\lambda^{\vee}}$ to \eqref{2409uuh}:
$$t_{-\lambda^{\vee}}w_{\theta_s}^{-1}\mathbf{s}_{p_{0}}\cdots \mathbf{s}_{p_{e}}(\alpha_{p_{e+1}})=-(w_{\theta_s}^{-1}(\theta_s)+\tau_{w_{\theta_s}})+\delta-(w_{\theta_s}^{-1}(\theta_s)+\tau_{w_{\theta_s}},\lambda^{\vee})\delta\in -R^+.$$
Hence (3)(i) and (3)(ii) are proven. The second part in (iii) is obtained with $w_{\theta_s}^{-1}(\theta_s) = \mathbf{s}_1\mathbf{s}_2(\theta_s) = \alpha_2$ and 
\begin{align*}t_{-\lambda^{\vee}}w_{\theta_s}^{-1}\mathbf{s}_0(\alpha_1) &= \alpha_1 + (1+ (\alpha_1,\lambda^{\vee}))\delta \in R^+, \\ t_{-\lambda^{\vee}}w_{\theta_s}^{-1}(\alpha_0) &= -\alpha_1 -3\alpha_2 + (1- (\alpha_1+3\alpha_2,\lambda^{\vee}))\delta \in -R^+.\end{align*} 
The first part in (iii) requires a case-by-case analysis of the three short roots of $G_2$ and we omit the details.  
\end{pf}
\end{prop}
\subsection{} For $\beta\in \mathring{R}^+$ and $u_{\beta}\in\mathring{W}_{\beta}$ we write 
$$-u_{\beta}(\beta)+\delta=\alpha_{p_0}+\alpha_{p_1}+\cdots+\alpha_{p_{s+1}},\ 0=p_0<p_1<\cdots<p_{s+1}\leq n$$
and note that the notation is compatible with the one in Proposition~\ref{hilflem}.
\begin{prop}\label{rel3210} 
Let $\beta\in \mathring{R}^+$. The following statements are true:
\begin{enumerate}
    \item For any $u_{\beta}\in\mathring{W}_{\beta}$ we have \begin{align*} \chi_{\beta,-1}^+&=\kappa(\beta)\cdot \mathbf{T}_{u_{\beta}}^{-1}\mathbf{T}_{p_0}\cdots \mathbf{T}_{p_{s-1}}\mathbf{T}_{p_s}^{\epsilon}\mathbf{T}_{p_{s+1}}(E_{p_{s+1}})\\
\chi_{\beta,1}^-&=\kappa(\beta)\cdot \mathbf{T}_{u_{\beta}}^{-1}\mathbf{T}_{p_0}\cdots \mathbf{T}_{p_{s-1}}\mathbf{T}_{p_s}^{\epsilon}\mathbf{T}_{p_{s+1}}(F_{p_{s+1}}) \end{align*}
where $\epsilon=-1$ if $\mathring{\mathfrak{g}}$ is of type $G_2$ and $\beta\in\{\alpha_1+\alpha_2, \alpha_2\}$ and $\epsilon=1$ otherwise.
\vspace{0,2cm}
\item Let $w_{\beta}^{-1}\in\mathring{W}(\beta)$, say $w_{\beta}(\alpha_i)=\beta$, $i\in\mathring{I}$. Then 
\begin{align*}\chi_{\beta,-1}^+&= \kappa(\beta)\cdot \mathbf{T}_{w_{\beta}}\mathbf{T}_{\varpi_i}(E_i)\\
\chi_{\beta,1}^-& = \kappa(\beta)\cdot  \mathbf{T}_{w_{\beta}}\mathbf{T}_{\varpi_i}(F_i)\end{align*}
unless $\mathring{\lie{g}}$ is of type $G_2$ and $\beta=\theta_s$.
\end{enumerate}

In particular, the above elements are independent of the choice of reduced expression \eqref{seteq0}.
\begin{enumerate}
    \item [(3)] If the reduced expression in \eqref{seteq0} is induced from a concatenation of reduced expressions of the translations in the order $t_{\nu(\mathbf{u}^{\vee})}=t_{\nu(\varpi_{u_1}^{\vee})}\cdots t_{\nu(\varpi_{u_p}^{\vee})}$
we have
 $$\chi_{\beta,z_{u_1}+\cdots+
 z_{u_{\ell -1}}+1}^- = \kappa(\beta)^{z_{u_1}+\cdots+
 z_{u_{\ell -1}}} \mathbf{T}_{\varpi_{u_1}} \cdots \mathbf{T}_{\varpi_{u_{\ell-1}}}(\chi_{\beta,1}^-),\ \ \ell\in\mathbb{N}$$
 where $z_i:=(\beta, \varpi_i^{\vee})$ and the indices are understood modulo $p$, i.e. $z_{u_{ap+j}}=z_{u_j}$.
\end{enumerate}

%{\color{blue} If $\mathring{\mathfrak{g}}$ is of type $G_2$, then
%$$\kappa(\beta)\cdot x_{\beta,-1}^+ = \begin{cases}
%\mathbf{T}_{u_{\beta}}^{-1}\mathbf{T}_0\mathbf{T}_1^{-1}\mathbf{T}_2(E_2), & \text{if} \ \beta\in\{\alpha_1+\alpha_2, \alpha_2\},\\
%\mathbf{T}_{u_{\beta}}^{-1}\mathbf{T}_0\mathbf{T}_1\mathbf{T}_2(E_2), & \text{if} \ \beta = \theta_s, \\
%\mathbf{T}_{u_\beta}^{-1}\mathbf{T}_0(E_0), & \text{if} \ \beta \ \text{long}
%\end{cases}=\mathbf{T}_{w_{\beta}}\mathbf{T}_{\varpi_i}(E_i).$$}

\begin{pf}
Notice that the first line of equations in (1) and (2) follow from the second line of equations by applying $\Omega$. To see (1) let $-\beta+\delta=\beta_k$ (recall the reduced expression from \eqref{seteq0}) and note that if $\alpha\in R^+$ such that $\gamma:=\mathbf{s}_{i_1}\cdots \mathbf{s}_{i_{k-1}}(\alpha)\in -R^+$, then $\alpha = \mathbf{s}_{i_{k-1}}\cdots \mathbf{s}_{i_{\ell+1}}(\alpha_{i_\ell})$ (see \cite[pg. 115 Exercise 1]{Hu90}) for a suitable $\ell\in\{1,\dots,k-1\}$. Thus 
$$-\gamma=\mathbf{s}_{i_1}\cdots \mathbf{s}_{i_{\ell-1}}(\alpha_{i_\ell})=\sum_{i\in I}s_i\alpha_i$$
with $s_0>0$ (c.f. Lemma~\ref{papi}) which implies that $u_{\beta}(\gamma)\in -R^+$ since $0\notin I(u_{\beta})$. Hence we have once more $R_{\mathbf{s}_{i_1}\cdots \mathbf{s}_{i_{k-1}}}^+\subseteq R_{u_{\beta}\mathbf{s}_{i_1}\cdots \mathbf{s}_{i_{k-1}}}^+$ and the length additivity is guaranteed by \eqref{dcond1}. We obtain
$$\mathbf{T}_{u_{\beta}\mathbf{s}_{i_1}\cdots \mathbf{s}_{i_{k-1}}}=\mathbf{T}_{u_{\beta}}\mathbf{T}_{i_1}\cdots \mathbf{T}_{i_{k-1}}.$$
Moreover, using $\mathbf{s}_{p_s}\cdots \mathbf{s}_{p_1}\mathbf{s}_{p_0}u_{\beta}\mathbf{s}_{i_1}\cdots \mathbf{s}_{i_{k-1}}(\alpha_{i_k})=\alpha_{p_{s+1}}$ we get with Proposition~\ref{spanlem} and Proposition~\ref{hilflem}(3)(i),(iii):
$$E_{p_{s+1}}=\mathbf{T}_{\mathbf{s}_{p_s}\cdots \mathbf{s}_{p_1}\mathbf{s}_{p_0}u_{\beta}\mathbf{s}_{i_1}\cdots \mathbf{s}_{i_{k-1}}}(E_{i_k})=\mathbf{T}^{-\epsilon}_{p_s}\cdots \mathbf{T}^{-1}_{p_1}\mathbf{T}^{-1}_{p_0}\mathbf{T}_{u_{\beta}}\mathbf{T}_{i_1}\cdots \mathbf{T}_{i_{k-1}}(E_{i_k}).$$
Hence part (1) is proven. 

For part (2) we fix an arbitrary $\tilde{u}_{\beta}\in \mathring{W}_{\beta}$ if $\beta$ is a long root and $\tilde{u}_{\beta}\in \mathring{W}_{\beta}$ as in Remark~\ref{choosed76} if $\beta$ is short, i.e. $(\tilde{u}_{\beta}w_{\beta})^{-1}\in\mathring{W}(\theta_s)$. Note that 
\begin{align} \label{corw0}
(\mathbf{s}_{p_{s+1}}\cdots\mathbf{s}_{p_{0}}\tilde{u}_{\beta}w_{\beta}t_{\nu(\varpi^{\vee}_i)})^{-1}(\alpha_{p_{s+1}}) = (w_{\beta}t_{\nu(\varpi^{\vee}_i)})^{-1}(\beta-\delta) = \alpha_i \in \mathring{R}^+.
\end{align}
Therefore 
\begin{align*}
\mathbf{T}_{\tilde{u}_{\beta}}^{-1}\mathbf{T}_{p_0}\cdots \mathbf{T}_{p_{s-1}}\mathbf{T}_{p_s}^{\epsilon}\mathbf{T}_{p_{s+1}}\mathbf{T}_{\mathbf{s}_{p_{s+1}}\cdots\mathbf{s}_{p_{0}}\tilde{u}_{\beta}w_{\beta}t_{\nu(\varpi^{\vee}_i)}}&= \mathbf{T}_{\tilde{u}_{\beta}}^{-1}\mathbf{T}_{p_0}\cdots\mathbf{T}_{p_{s-1}}\mathbf{T}_{p_s}^{\epsilon}\mathbf{T}_{\mathbf{s}_{p_{s}}\cdots\mathbf{s}_{p_{0}}\tilde{u}_{\beta}w_{\beta}t_{\nu(\varpi^{\vee}_i)}}&\\&=
\mathbf{T}_{\tilde{u}_{\beta}}^{-1}\mathbf{T}_{\tilde{u}_{\beta}w_{\beta}t_{\nu(\varpi^{\vee}_i)}}=
\mathbf{T}_{w_{\beta}}\mathbf{T}_{\varpi_i}
\end{align*}
where the first equation follows from \eqref{corw0}, the second from Proposition~\ref{hilflem}(3)(ii),(iii) and the last one from Proposition~\ref{hilflem}(2). Now the statement follows from (1) and \eqref{corw0} by applying the above equations to $E_i$. 

For the last part of the proposition set $r=z_{u_1}+\cdots+z_{u_{\ell-1}}$ and assume $\mathbf{s}_{i_1}\cdots\mathbf{s}_{i_{k-1}}(\alpha_{i_k}) = -\beta + (r+1)\delta$. We first claim that $\mathbf{T}_{\varpi_{\ell}}(\chi_{\beta,1}^-)=\chi_{\beta,1}^-$ whenever $(\beta,\varpi_{\ell}^{\vee})=0$. If $\beta=\theta_s$ and $\mathring{\lie{g}}$ is of type $G_2$, we never have $(\beta,\varpi_{\ell}^{\vee})=0$. So we can use part (2) to get with Corollary~\ref{corhilf4q}:
$$\mathbf{T}_{\varpi_{\ell}}(\chi_{\beta,1}^-)=\kappa(\beta)\cdot \mathbf{T}_{\varpi_{\ell}}\mathbf{T}_{w_{\beta}}\mathbf{T}_{\varpi_i}(F_i)=\kappa(\beta)\cdot \mathbf{T}_{w_{\beta}}\mathbf{T}_{\varpi_i}\mathbf{T}_{\varpi_{\ell}}(F_i)=\chi_{\beta,1}^-.$$ 
Therefore we can assume without loss of generality (in order to describe $\chi_{\beta,r+1}^-$) that (c.f. \cite[Lemma 1]{B94A})
$$\mathbf{s}_{i_1}\cdots\mathbf{s}_{i_{k-1}}(\alpha_{i_k}) = t_{\nu(\varpi^{\vee}_{u_1})}\cdots t_{\nu(\varpi^{\vee}_{u_{\ell-1}})} \tau \mathbf{s}_{e_1}\cdots \mathbf{s}_{e_{m-1}}(\alpha_{e_m})$$
where $\tau\mathbf{s}_{e_1}\cdots \mathbf{s}_{e_{m-1}}$ is an initial subword of $t_{\nu(\varpi^{\vee}_{u_{\ell}})}$. Since
$$\tau\mathbf{s}_{e_1}\cdots \mathbf{s}_{e_{m-1}}(\alpha_{e_m})=-\beta+((r+1)-z_{u_1}-\cdots-z_{u_{\ell -1}})\delta = -\beta + \delta,$$
we get $\chi_{\beta,1}^-=\kappa(\beta)\cdot \mathbf{T}_{\tau}\mathbf{T}_{e_1}\cdots \mathbf{T}_{e_{m-1}}\mathbf{T}_{e_{m}}(F_{e_m})$ by the independence of $\chi_{\beta,1}^-$ of the reduced expression. Hence the claim follows. 
\end{pf}
\end{prop}
\begin{rem}\label{remrepresentable} \begin{enumerate}
    \item If $\mathring{\lie g}$ is of type $G_2$ and $\beta=\theta_s$ one can show that 
$$\chi_{\theta_s,1}^-=\kappa(\theta_s)\cdot \mathbf{T}_0\mathbf{T}^2_1\mathbf{T}^{-1}_0\mathbf{T}_2\mathbf{T}_1\mathbf{T}_{\varpi_2}(F_2).$$
\item Since $\chi_{\beta,\mp 1}^{\pm}$ does not depend on the choice of the reduced expression by Proposition~\ref{rel3210}, we get that for any $k\in\mathring{I}$ with $(\beta,\varpi_k^{\vee})>0$ the element $\chi_{\beta,1}^-$ is representable as 
$$\chi_{\beta,1}^-=\kappa(\beta)\cdot \mathbf{T}_{\varpi_k}\mathbf{T}^{-1}_{k_q}\cdots \mathbf{T}^{-1}_{k_{r+1}}(F_{k_{r}})$$
where $t_{\nu(\varpi_k^{\vee})}=\tau \mathbf{s}_{k_1}\cdots \mathbf{s}_{k_{r-1}}\mathbf{s}_{k_r}\cdots \mathbf{s}_{k_q}$ and $\tau \mathbf{s}_{k_1}\cdots \mathbf{s}_{k_{r-1}}(\alpha_{k_r})=-\beta+\delta$. We only have to choose a reduced expression of $t_{\nu(\mathbf{u}^{\vee})}$ which starts with a reduced expression of the translation $t_{\nu(\varpi_k^{\vee})}$.
\end{enumerate}
\end{rem}
We have seen in Proposition~\ref{rel3210} that certain $\chi_{\beta,r+1}^-$ are obtained from $\chi_{\beta,1}^-$ by applying suitable operators $\mathbf{T}_{\varpi_i}$'s. We conclude this subsection by providing an example that demonstrates this is not true in general. However, it is not difficult to establish a criterion under which this property still holds, leading to the identification of Weyl group compatible coweights (for further details, see \cite{KBthesis}).
\begin{example}
We discuss the $B^{(1)}_3$ case here where $J=\lbrace 0,1,3 \rbrace$, $w_{\circ} = \mathbf{s}_1\mathbf{s}_3, \ \gamma_0 = w_{\circ}(\theta) = \theta$ and $\mathbf{u}^{\vee} = \varpi^{\vee} _1+\varpi^{\vee} _2+\varpi^{\vee} _3$. We consider the reduced expression
$$t_{\nu(\mathbf{u}^{\vee})}=\mathbf{s}_0\mathbf{s}_2\mathbf{s}_3\mathbf{s}_2\mathbf{s}_0\cdot \mathbf{s}_1\mathbf{s}_2\mathbf{s}_3\mathbf{s}_0\mathbf{s}_2\mathbf{s}_3\mathbf{s}_0\mathbf{s}_2\cdot \mathbf{s}_1\mathbf{s}_2\mathbf{s}_3\mathbf{s}_0\mathbf{s}_2\mathbf{s}_3\mathbf{s}_1\mathbf{s}_2\mathbf{s}_3.$$
A direct calculation shows
$$\beta_1 = -\gamma_0 + \delta, \ \ \beta_6 = -\gamma_0 + 2\delta, \ \ \beta_{10} = -\gamma_0 + 3\delta, \ \ \beta_{14} = -\gamma_0 +4\delta.$$
Proposition~\ref{rel3210} gives
$$\chi_{\gamma_0,2}^- = \kappa(\gamma_0)\cdot \mathbf{T}_{\varpi_1}(\chi_{\gamma_0,1}^-), \ \ \chi_{\gamma_0,4}^- = \kappa(\gamma_0)\cdot \mathbf{T}_{\varpi_1}\mathbf{T}_{\varpi_2}(\chi_{\gamma_0,1}^-).$$
    However we can not obtain $\chi_{\gamma_0,3}^-$  from $\chi_{\gamma_0,1}^-$ by applying suitable $\mathbf{T}_{\varpi_i}$'s. To see this we will apply (because of weight reasons)
    $$\mathbf{T}_{\varpi_1}^2 \text{ or } \mathbf{T}_{\varpi_2} \ \text{or} \ \mathbf{T}_{\varpi_3}$$
    to $\chi_{\gamma_0,1}^-$ and verify that it is indeed not equal to $\chi_{\gamma_0,3}^-$ in all three cases. We have
    \begin{align*}
    \kappa(\gamma_0)\cdot  \mathbf{T}_{\varpi_1}^2(\chi_{\gamma_0,1}^-) &= \mathbf{T}_{\tau_1}\mathbf{T}_1\mathbf{T}_2\mathbf{T}_3\mathbf{T}_2\mathbf{T}_1\mathbf{T}_{\tau_1}\mathbf{T}_1\mathbf{T}_2\mathbf{T}_3\mathbf{T}_2\mathbf{T}_1\mathbf{T}_0(F_0) \\
    &=\mathbf{T}_0\mathbf{T}_2\mathbf{T}_3\mathbf{T}_2\mathbf{T}_0\mathbf{T}_1\mathbf{T}_2\mathbf{T}_3\mathbf{T}_2\mathbf{T}_1\mathbf{T}_0(F_0) 
    \end{align*}
    where $\tau_1$ is the automorphism of the affine Dynkin diagram given by $\tau_1(0) = 1$ and
    \begin{align*}
    \kappa(\gamma_0)\cdot \chi_{\gamma_0,3}^- &=  \mathbf{T}_0 \mathbf{T}_2 \mathbf{T}_3 \mathbf{T}_2 \mathbf{T}_0 \mathbf{T}_1 \mathbf{T}_2 \mathbf{T}_3 \mathbf{T}_0\mathbf{T}_2(F_2).
    \end{align*}
    Now suppose $\mathbf{T}_{\varpi_1}^2(\chi_{\gamma_0,1}^-)=\chi_{\gamma_0,3}^-$. Then 
    $$\mathbf{T}_0\mathbf{T}_2\mathbf{T}_3\mathbf{T}_2\mathbf{T}_0\mathbf{T}_1\mathbf{T}_2\mathbf{T}_3\mathbf{T}_2\mathbf{T}_1\mathbf{T}_0(F_0) = \mathbf{T}_0 \mathbf{T}_2 \mathbf{T}_3 \mathbf{T}_2 \mathbf{T}_0 \mathbf{T}_1 \mathbf{T}_2 \mathbf{T}_3 \mathbf{T}_0\mathbf{T}_2(F_2) $$
        $$\implies \mathbf{T}_2\mathbf{T}_1\mathbf{T}_0(F_0) = \mathbf{T}_0\mathbf{T}_2(F_2) \implies  \mathbf{T}_2\mathbf{T}_0(F_0) = \mathbf{T}_0\mathbf{T}_2(F_2) \implies \mathbf{T}_2(E_0) = \mathbf{T}_0(E_2)
   $$
    which is false. With similar calculations the remaining two cases also yield false statements.
    
\end{example} 
\subsection{} 
In this subsection, we will outline a few consequences.
\begin{prop}\label{zzhnböö9} The following statements are true:
\begin{enumerate}
\item [(a)] Let $\beta \in \mathring{R}^+$ and $k\in\mathring{I}$ such that $(\varpi_k^{\vee},\beta) =1$. Then $\chi_{\beta,1}^-$ lies in the $\mathbb{C}(q)$-subalgebra generated by $\chi_{k,1}^-$ and $\chi_{i,0}^-$ for $i\in \text{supp}(\beta)\setminus \lbrace k \rbrace.$
\end{enumerate}
Assume $J\neq \mathring{I}$, so that $\gamma_0\in \mathring{R}^+$.
\begin{enumerate}
\item [(b)] Let $w= \mathbf{s}_{k_1}\cdots \mathbf{s}_{k_p}$ be a reduced expression with $k_1,\dots,k_p\in\mathring{I}\cap J$. Then
$\mathbf{T}_{w}(\chi_{\gamma_0,1}^-)$ lies in the $\mathbb{C}(q)$-subalgebra generated by the elements $\chi_{\gamma_0,1}^-,\chi_{k_1,0}^-,\dots,\chi_{k_p,0}^-$.

\item [(c)] If $\beta \in \mathring{R}^+$ satisfies $\mathrm{supp}(-\beta +\delta)\subseteq J$, then the element $\chi_{\beta,1}^-$ lies in the $\mathbb{C}(q)$-algebra generated by the elements $\chi_{\gamma_0,1}^-,\chi^-_{j,0}$,  $j\in \mathring{I}\cap J$.
\end{enumerate}
\begin{proof}
We first prove (a). From Remark~\ref{remrepresentable} we know that (given such $k\in\mathring{I}$)  
$$\chi_{\beta,1}^-=\kappa(\beta)\cdot \mathbf{T}_{\varpi_k}\mathbf{T}^{-1}_{k_q}\cdots \mathbf{T}^{-1}_{k_{r+1}}(F_{k_{r}})$$
where $t_{\nu(\varpi_k^{\vee})}=\tau \mathbf{s}_{k_1}\cdots \mathbf{s}_{k_{r-1}}\mathbf{s}_{k_r}\cdots \mathbf{s}_{k_q}$ and $\tau \mathbf{s}_{k_1}\cdots \mathbf{s}_{k_{r-1}}(\alpha_{k_r})=-\beta+\delta$. However, $(\varpi_k^{\vee},\beta) = 1$ gives additionally 
 $\mathbf{s}_{k_q}\cdots \mathbf{s}_{k_{r+1}}(\alpha_{k_r}) = \beta$ and thus 
 $$\mathbf{T}^{-1}_{k_q}\cdots \mathbf{T}^{-1}_{k_{r+1}}(F_{k_{r}})=\Phi\circ \mathbf{T}_{\mathbf{s}_{k_q}\cdots \mathbf{s}_{k_{r+1}}}(E_{k_r})$$
 is contained in the $\mathbb{C}(q)$-algebra generated by $F_i: i \in \text{supp}({\beta})$. Now applying $\mathbf{T}_{\varpi_k}$ shows (a). 
 
By Proposition~\ref{hilflem}(1) and Proposition~\ref{rel3210} we have
$$\mathbf{T}_j(\chi_{\gamma_0,1}^-)=\kappa(\gamma_0)\cdot \mathbf{T}^{-1}_{\omega_{\circ}}\mathbf{T}_{\bar{j}}\mathbf{T}_0(F_0),\ \forall j\in \mathring{I}\cap J.$$
Since $\mathbf{T}_{\bar{j}}\mathbf{T}_0(F_0)$ lies in the algebra span of the elements $\mathbf{T}_0(F_0)$ and $K_{\bar{j}}^{-1}E_{\bar{j}}$ we see that the action of $\mathbf{T}^{-1}_{\omega_{\circ}}$ on $\mathbf{T}_{\bar{j}}\mathbf{T}_0(F_0)$ lies in the subalgebra generated by the elements $\chi_{\gamma_0,1}^-$ and $\chi_{j,0}^-$ since $\mathbf{T}^{-1}_{\omega_{\circ}}(E_{\bar{j}})=-K_j^{-1}F_j$. In particular $\mathbf{T}_{w}(\chi_{\gamma_0,1}^-)$ is contained in the subalgebra generated by $\mathbf{T}_{w\mathbf{s}_{k_p}}(\chi_{\gamma_0,1}^-)$ and  $\mathbf{T}_{w\mathbf{s}_{k_p}}(\chi_{k_p,0}^-)$. Now we can argue by induction that the first element lies in the subalgebra generated by the elements  $\chi_{\gamma_0,1}^-,\chi_{k_1,0}^-,\dots,\chi_{k_{p-1},0}^-$ and for the second element we use Proposition~\ref{spanlem}. This shows part (b).
%wie im Beweis vom Lemma unten

Now we prove part (c) through the equivalent statement that $CK^-_{w_{\circ}(\beta)}\mathbf{T}_{w_{\circ}}(\chi_{\beta,1}^-)$ lies in the $\mathbb{C}(q)$-algebra generated by $E_0, E_j$, $j\in \mathring{I}\cap J$: 
\begin{align*}
    CK^-_{w_{\circ}(\beta)}\mathbf{T}_{w_{\circ}}(\chi_{\beta,1}^-) \in \langle E_0, E_j \ | \ j\in \mathring{I}\cap J \rangle &\Leftrightarrow C^{} K^-_{\beta}\chi_{\beta,1}^- \in  \langle \mathbf{T}_{w_{\circ}}^{-1}(E_0), \mathbf{T}_{w_{\circ}}^{-1}(E_j) \ | \ j\in \mathring{I}\cap J\rangle \\
    &\Leftrightarrow C^{} K^-_{\beta}\chi_{\beta,1}^- \in  \langle CK_{\gamma_0}^-\chi_{\gamma_0,1}^-, K_j^-\chi_{j,0}^- \ | \ j\in \mathring{I}\cap J\rangle \\
    &\Leftrightarrow \chi_{\beta,1}^- \in  \langle \chi_{\gamma_0,1}^-,\chi_{j,0}^- \ | \ j\in \mathring{I}\cap J\rangle.
\end{align*} 
By Corollary~\ref{observation1} we can assume that we have an element $u_{\beta}\in\mathring{W}_{\beta}$ with $I(u_{\beta})\subseteq \mathring{I}\cap J$ and therefore we can write $w_\circ=u u_{\beta}$ for some $u\in\langle \mathbf{s}_i: i\in \mathring{I}\cap J \rangle$ and $\ell(u_{\beta})+\ell(u)=\ell(w_\circ)$. 
With Proposition~\ref{rel3210} we have 
$$\mathbf{T}_{w_{\circ}}(\chi_{\beta,1}^-) =\kappa(\beta)\cdot \mathbf{T}_{u}\mathbf{T}_{p_0} \cdots \mathbf{T}_{p_{s}}^{}\mathbf{T}_{p_{s+1}}(F_{p_{s+1}})$$
because $\mathrm{supp}(-\beta +\delta)\subseteq J\subsetneq I$ excludes the case $\beta\in\{\alpha_1+\alpha_2,\alpha_2\}$ in type $G_2$.
Since 
\begin{align*}u\mathbf{s}_{p_0}\cdots \mathbf{s}_{p_e}(\alpha_{p_{e+1}}) &= u(-\theta_s + \delta - \alpha_{p_{e+2}}-\cdots- \alpha_{p_{s+1}}) &\\&= -w_{\circ}(\beta)-u(\alpha_{p_{e+2}}+ \cdots + \alpha_{p_{s+1}}) +\delta \in R^+\end{align*}
    for all $e\in \{-1,\dots,s\}$ we can conclude that
    $$-CK^-_{w_{\circ}(\beta)}\mathbf{T}_u\mathbf{T}_{p_0} \cdots \mathbf{T}_{p_{s}}\mathbf{T}_{p_{s+1}}(F_{p_{s+1}}) = \mathbf{T}_{u\mathbf{s}_{p_0}\cdots \mathbf{s}_{p_s}}(E_{p_{s+1}}).$$ Now again with Proposition~\ref{spanlem} we obtain that the above element lies in the algebra generated by $E_0, E_j, j\in \mathring{I}\cap J$, which finishes the proof.
\end{proof}
\end{prop}
\begin{rem}\label{trisp2} Assume that $\mathfrak{p}_J$ is maximal, i.e. $I\setminus J = \lbrace i_0 \rbrace$, $i_0\neq 0$ and $\mathbf{a}_{i_0}(\delta)\leq 2$. Then $\chi_{\beta,1}^-$ is contained in the $\mathbb{C}(q)$-algebra generated by
$$\chi_{\gamma_0,1}^-,\chi_{k,1}^-,\chi_{j,0}^-,\ j\in\mathring{I}\cap J,\ k\in\mathring{I}.$$
To see this we consider a few cases.
%i\in \mathrm{supp}(\beta)\backslash \{i_0\}$$
If $\mathrm{supp}(-\beta+\delta)\subseteq J$, we can use Proposition~\ref{zzhnböö9}(c). Otherwise, we must have $(\varpi_{i_0}^{\vee},\beta)\leq 1$ since $\mathbf{a}_{i_0}(\delta)\leq 2$ and if $(\varpi_{i_0}^{\vee},\beta)= 1$ we can use directly Proposition~\ref{zzhnböö9}(a). We are left with the case $(\varpi_{i_0}^{\vee},\beta)= 0$ which implies in particular $i_0\notin \mathrm{supp}(\beta)$ and $\mathrm{supp}(\beta)\subsetneq \mathring{I}$. However, we can always find $k\in\mathring{I}$ with $(\varpi_k^{\vee},\beta) =1$ provided that $\mathrm{supp}(\beta)\subsetneq \mathring{I}$. Hence we can use once more Proposition~\ref{zzhnböö9}(a).
\end{rem}
\subsection{} The algebra generators are given as follows. 
\begin{prop}\label{genalgim}
The following elements generate $\mathbf{U}_q^J$ as an algebra 
$$\chi_{j,r}^{\pm},\ \chi_{i,r}^{+},\ \chi_{i,r+1}^{-},\ \   i\in \mathring{I}\backslash (\mathring{I}\cap J),\ j\in \mathring{I}\cap J,\ r\in\mathbb{Z}_+$$
$$K_i^{\pm}, h_{i,s}, C^{\pm 1/2}, D^{\pm},\ \   s \in \mathbb{N},\   i\in\mathring{I}$$
$$\chi_{\gamma_0,r+1}^{-}, \ r\in\mathbb{Z}_+,\  \text{ if } \ J\neq \mathring{I},\ \ \chi_{\gamma_0,r- 1}^{+}, \ r\in\mathbb{Z}_+,\  \text{ if } \ 0\in J.$$
\begin{pf} Let $\mathcal{A}_q^J$ be the algebra generated by the elements above. From Lemma~\ref{lemsubset667} we can derive that $\mathcal{A}_q^J$ is contained in $\mathbf{U}_q^J$. 
For the converse direction it is enough to show that the elements $E_i, i\in I$ and $F_j$, $j\in J$ are contained in $\mathcal{A}_q^J$ which is obvious if $i\in\mathring{I}$ and $j\in\mathring{I}\cap J$. This fact is also well-known if $\mathring{I}=J$ (see for example \cite[Section 12.2]{CP95}). So let $\mathring{I}\neq J$ and use Proposition~\ref{rel3210}(1) to write 
$$E_0=-\kappa(\gamma_0)\cdot K_0\mathbf{T}_{w_{\circ}}(\chi_{\gamma_0,1}^{-}),\ \ F_0=-\kappa(\gamma_0)\cdot\mathbf{T}_{w_{\circ}}(\chi_{\gamma_0,-1}^+) K_0^{-1}\ \ (\text{if $0\in J$}).$$
Now these elements are contained in $\mathcal{A}_q^J$ by Proposition~\ref{zzhnböö9}(b). 
\end{pf}
\end{prop}
\begin{rem}
In fact, the last set of generators is only required when $r=0$, and thus the parabolic quantum affine algebra is independent of the choice of reduced expression \eqref{seteq0}. While this is not surprising given the definition, we include this observation for completeness.
\end{rem}
\subsection{} Similarly, if we denote by $\mathring{\mathcal{A}}_q^{J}$ the subalgebra of $\mathbf{U}_q^J$
generated by the elements $$\chi_{j,0}^-,\ \ \chi_{j,0}^+,\ \ K_{j}^{\pm},\ \ j\in \mathring{I}\cap J$$
$$ \chi_{\gamma_0,1}^-,\ \ \chi_{\gamma_0,-1}^+,\ \  K^{\pm}_{\gamma_0-\delta},\ \ \text{ (if $0\in J$)}$$
we can naturally identify $\mathring{\mathcal{A}}_q^{J}$ with $\mathring{\mathbf{U}}_q^J$, the quantum group associated to the semi-simple part of $\lie{\mathring{g}}_0$ (c.f. Lemma~\ref{irrmod}), via the isomorphism %$\mathbf{T}_{w_{\circ}}^{-1}\circ \Omega'\circ S^{-1}$ given by (up to sign):
$$F_{\bar{j}}\mapsto \chi_{j,0}^-,\ \ E_{\bar{j}}\mapsto \chi_{j,0}^+,\ \ K^{\pm}_{\bar{j}}\mapsto K_{j}^{\pm},\ \ j\in \mathring{I}\cap J$$
$$F_0\mapsto \chi_{\gamma_0,1}^-,\ \ E_0\mapsto \chi_{\gamma_0,-1}^+,\ \ K^{\pm}_0\mapsto K^{\pm}_{\gamma_0-\delta},\ \ \text{ (if $0\in J$)}.$$
Thus, we have obtained an analogue of the Drinfeld picture: the parabolic quantum affine algebra generated by the Drinfeld-type generators described in Proposition~\ref{genalgim}, together with a Hopf subalgebra
$$\mathring{\mathbf{U}}_q^J\cong \mathring{\mathcal{A}}_q^{J} \hookrightarrow \mathcal{A}_q^{J}\cong \mathbf{U}_q^J$$ 
whose category of finite-dimensional representations is semi-simple.
\section{PBW basis and second triangular decomposition}\label{section6}
\subsection{}The same methods presented in \cite{B94,B94A} also yield a PBW type basis of $\mathbf{U}_q^J$. We mention this result here for completeness.

\begin{prop}\label{pbwpara} Fix an arbitrary order on $\{1,\dots,n\}\times \mathbb{N}$ and recall the convex order from \eqref{convexordö}. 
\begin{enumerate}

\item We have a basis of $\mathbf{U}_q^{J,+}$ given by the following ordered monomials 

\begin{align} \label{uqjpbw0}
   M_{f,g,\varphi}=\prod_{\substack{\beta\in \mathring{R}^+, r\geq 0 \\ \text{convexly ordered}}} (\Dot{\chi}_{\beta,r}^+)^{f(\beta,r)}\prod_{(i,r)\in \lbrace 1,...,n \rbrace \times \mathbb{N}}(\Dot{h}_{i,r})^{g(i,r)}\prod_{\substack{\beta\in \mathring{R}^+, r> 0 \\ \text{convexly ordered}}}(\Dot{\chi}_{\beta,r}^-)^{\varphi(\beta,r)}
\end{align}
for various finitely supported functions $f: \mathring{R}^+ \times \mathbb{Z}_+ \to \mathbb{Z}_+, \ g: \lbrace 1,...,n \rbrace \times \mathbb{N} \to \mathbb{Z}_+, \ \varphi: \mathring{R}^+\times \mathbb{N} \to \mathbb{Z}_+$ and the second product is taken in the fixed order. \vspace{0,2cm}
\item We have a basis of $\mathbf{U}_q^{J,-}$ given by the following ordered monomials
\begin{align} \label{uqjpbw}
  N_{f,\varphi}=\prod_{\substack{\beta\in \mathring{R}^+,\ 
    \mathrm{supp}(\beta-\delta)\subseteq J\\ \text{reversed convexly ordered}}}(\Dot{\chi}_{\beta,-1}^+)^{\varphi(\beta)}\prod_{\substack{\beta\in \mathring{R}^+,\  \mathrm{supp}(\beta)\subseteq J\\ \text{reversed convexly ordered}}}(\Dot{\chi}_{\beta,0}^-)^{f(\beta)}
\end{align}
for various functions $f,\varphi: \mathring{R}^+ \to \mathbb{Z}_+$.
    \end{enumerate}
\end{prop}
\begin{proof} The first part follows immediately from \cite[§1.Corollary 4]{B94A}. Clearly the elements in \eqref{uqjpbw} are linearly independent by \cite[Proposition 3]{B94A} and are contained in $\mathbf{U}_q^{J,-}$ by Lemma~\ref{lemsubset667}. Let $B$ be the set of monomials obtained by applying $\Omega$ to \eqref{uqjpbw0} and let $B_1 \subseteq B$ the subset of elements of monomials of the form \eqref{uqjpbw}. Given $x\in \mathbf{U}_q^{J,-}$ we can write  $x=y+z$ by using the PBW basis of $\mathbf{U}_q$ constructed in \cite{B94A} where $y$ (resp. $z$) is a linear combination of elements in $B_1$ (resp. $B\setminus B_1$). However, $x-y = z$ and thus the weight of $z$ is contained in $-\sum_{j\in J}\mathbb{Z}_+\alpha_j$ which is only possible if $x=y$.
\end{proof}
\begin{rem}\label{reprdspa} Proposition~\ref{pbwpara} induces a PBW basis of $\mathbf{U}_q^J$
\begin{equation}\label{bhg4902}\{N_{f',\varphi'}K_{\alpha}C^{\pm s/2}D^rM_{f,g,\varphi},\ \ f',\varphi',f,g,\varphi,\alpha,s,r\}.\end{equation}
Using this basis, one can define as in \cite{KdC90} (see also \cite[Section 1.8]{BK96}) a filtration on $\mathbf{U}_q^J$ such that its associated graded space is described as in \cite[Proposition 1.8]{BK96}. In particular, the quantum root vectors commute (resp. $q$-commute) in the associated graded space. This implies that if we reorder the elements in \eqref{bhg4902} in any given fixed order, we will remain with a basis of $\mathbf{U}_q^J$. This observation will be used later.
\end{rem}
\subsection{} The quantum affine algebra has two well-known triangular decompositions. We want to derive an analogue for parabolic quantum affine algebras in the rest of this section. We first note the following lemma. 
\begin{prop}\label{algebrengleich3} The following four subspaces of $\mathbf{U}_q^J$ coincide
\begin{enumerate}
    \item the subalgebra generated by 
\begin{equation}\label{ghset41}\{\chi_{\alpha,-1}^+, \chi_{\beta,r}^+, \ \  \alpha,\beta \in \mathring{R}^+,\ \mathrm{supp}(\alpha-\delta )\subseteq J,\ r\geq 0\} \end{equation}
\item the subalgebra generated by 
\begin{equation}\label{ghset411}\{\chi_{i,r}^+, \ i \in \mathring{I},\ r\geq 0\}\sqcup \{\chi_{\gamma_0,-1}^+: \text{ if $0\in J$}\} \end{equation} 
\item the subspace spanned by the elements 
\begin{equation}\label{ghset411111}\prod_{\substack{\beta\in \mathring{R}^+, r\geq 0 \\ \text{convexly ordered}}} (\chi_{\beta,r}^+)^{\varphi'(\beta,r)}\prod_{\substack{\beta\in \mathring{R}^+,\ 
    \mathrm{supp}(\beta-\delta)\subseteq J \\ \text{reversed convexly ordered}}}(\chi_{\beta,-1}^+)^{\varphi(\beta)} \end{equation}
    for various (finitely supported) functions $\varphi':\mathring{R}^+\times \mathbb{Z}_+\rightarrow \mathbb{Z}_+$, $\varphi:\mathring{R}^+\rightarrow \mathbb{Z}_+$ \vspace{0,2cm}
\item $\mathbf{U}_q^J \cap \langle \chi_{i,r}^+ : i\in\mathring{I},\ r\in \mathbb{Z} \rangle.$
\end{enumerate}
We denote this subalgebra by $\mathbf{U}_q^J(+)$.
\begin{proof} From \cite[Lemma 5.4]{Da15a} and Lemma~\ref{lemsubset667} we obtain that  \begin{align}\label{eqsch5r}
        \langle \chi_{\alpha,-1}^+, \chi_{\beta,r}^+ : \alpha,\beta \in \mathring{R}^+, \mathrm{supp}(\alpha-\delta)\subseteq J,\ r\geq 0 \rangle\subseteq \mathbf{U}_q^J \cap \langle \chi_{i,r}^+ : i\in\mathring{I},\ r\in \mathbb{Z} \rangle 
       \end{align}
Clearly \eqref{ghset411} is a subset of \eqref{ghset41} and the subspace spanned by \eqref{ghset411111} is contained in the subalgebra generated by the elements \eqref{ghset41}. 
So in order to finish the proof we are left to show
\begin{itemize}
    \item The set \eqref{ghset41} is contained in the subalgebra generated by the elements \eqref{ghset411} \vspace{0,2cm}
    \item The subalgebra in (4) is contained in the subspace generated by \eqref{ghset411111}.
\end{itemize} From \eqref{eqsch5r} and \cite[Proposition 9.3 (i)]{Da15a} (see also the discussion in \cite[Section 2.3]{HJ12a}) we have 
\begin{equation}\label{stabil23}\chi_{\beta,r}^+ \in \langle \chi_{i,r}^+ : i\in \mathring{I},\ r\in\mathbb{Z}\rangle\cap \langle E_i : i\in I\rangle=\langle \chi_{i,r}^+ : i\in \mathring{I},\ r\geq 0\rangle.\end{equation} So it remains to consider the elements $\chi_{\alpha,-1}^+$ with $\mathrm{supp}(\alpha-\delta) \subseteq J$ (in particular $0\in J$). However, this case follows from Proposition~\ref{zzhnböö9}(c) by applying $\Omega$ to the generators and the first part is obtained. The proof of the second part closely follows the idea of \cite[Lemma 5]{B94A}. So suppose that we find an element 
 $$y\in \mathbf{U}_q^J \cap \langle \chi_{i,r}^+ : i\in\mathring{I},\ r\in \mathbb{Z} \rangle$$
 which is not contained in the linear span of the elements \eqref{ghset411111}. We have
    $$\mathbf{T}_{t_{\nu(\mathbf{u}^{\vee})}}^ky\in \mathbf{U}_q^- \cdot \mathbf{U}_q^0,\ \ \mathbf{T}_{t_{\nu(\mathbf{u}^{\vee})}}^{-k}y\in \mathbf{U}_q^{+},\  \text{ for some $k>>0$}.$$
By Remark~\ref{reprdspa} we can write  
    $$y = \sum c_{f_1,f_2,z,g_1,p,g_2}\ g_1\cdot f_1 \cdot z\cdot f_2\cdot p\cdot  g_2,$$
   where $z\in\mathbf{U}_q^0$ and 
$$f_1=\prod_{\substack{\beta\in \mathring{R}^+,\ 
    \mathrm{supp}(\beta-\delta)\subseteq J \\ \text{reversed convexly ordered}}}(\chi_{\beta,-1}^+)^{\varphi(\beta)} 
    ,\ \ f_2=\prod_{\substack{\beta\in \mathring{R}^+,\ \mathrm{supp}(\beta)\subseteq J \\ \text{reversed convexly ordered}}}(\Dot{\chi}_{\beta,0}^-)^{f(\beta)} 
$$  

$$g_1=\prod_{\substack{\beta\in \mathring{R}^+, r\geq 0 \\ \text{convexly ordered}}} (\chi_{\beta,r}^+)^{\varphi'(\beta,r)},\ \ p=\prod_{(i,r)\in \lbrace 1,...,n \rbrace \times \mathbb{N}}(\Dot{h}_{i,r})^{g(i,r)},\ \ 
   g_2=\prod_{\substack{\beta\in \mathring{R}^+, r> 0 \\ \text{convexly ordered}}}(\Dot{\chi}_{\beta,r}^-)^{g'(\beta,r)}$$
for some suitable choice of functions $\varphi,\varphi',f,g,g'$. Although we have dotted quantum root vectors in the PBW basis, we can move all $K_i's$ and $D's$ into $z$. By assumption we have $c_{f_1,f_2,z,g_1,p,g_2} \neq 0$ for some  $f_2, p,z$ or $g_2$ not equal to $1$. Choosing $k >>0$ big enough, we can guarantee $\mathbf{T}_{t_{\nu(\mathbf{u}^{\vee})}}^k(y) \in \mathbf{U}_q^-\cdot \mathbf{U}_q^0$ and
$$
    \mathbf{T}_{t_{\nu(\mathbf{u}^{\vee})}}^k(f_1)\in \mathbf{U}_q^- \cdot \mathbf{U}_q^0,\ \  \mathbf{T}_{t_{\nu(\mathbf{u}^{\vee})}}^k(g_1) \in \mathbf{U}_q^- \cdot \mathbf{U}_q^0,\ \  \mathbf{T}_{t_{\nu(\mathbf{u}^{\vee})}}^k(z) \in \mathbf{U}_q^0$$
   $$ \mathbf{T}_{t_{\nu(\mathbf{u}^{\vee})}}^k(f_2) \in \mathbf{U}_q^0\cdot \mathbf{U}_q^+, \ \ \mathbf{T}_{t_{\nu(\mathbf{u}^{\vee})}}^k(p\cdot g_2) \in \mathbf{U}_q^+
 $$
while also $\mathbf{T}_{t_{\nu(\mathbf{u}^{\vee})}}^{-k}(y)\in\mathbf{U}_q^{+}$ and
$$
\mathbf{T}_{t_{\nu(\mathbf{u}^{\vee})}}^{-k}(f_1) \in \mathbf{U}_q^{+},\ \ \mathbf{T}_{t_{\nu(\mathbf{u}^{\vee})}}^{-k}(g_1) \in \mathbf{U}_q^{+}, \ \
\mathbf{T}_{t_{\nu(\mathbf{u}^{\vee})}}^{-k}(z) \in \mathbf{U}_q^0.
$$
As in the proof of \cite[Lemma 5]{B94A} the expression
$$
\sum c_{f_1,f_2,z,g_1,p,g_2} \mathbf{T}_{t_{\nu(\mathbf{u}^{\vee})}}^k(g_1)\cdot \mathbf{T}_{t_{\nu(\mathbf{u}^{\vee})}}^k(f_1)\cdot \mathbf{T}_{t_{\nu(\mathbf{u}^{\vee})}}^k(z)\cdot \mathbf{T}_{t_{\nu(\mathbf{u}^{\vee})}}^k(f_2) \cdot \mathbf{T}_{t_{\nu(\mathbf{u}^{\vee})}}^k(p\cdot g_2)
$$
is a sum in PBW monomials, thus we get $f_2 = p = g_2 = 1$, i.e. 
$$y = \sum c_{f_1,z,g_1} g_1 \cdot f_1 \cdot z.$$ Also if $c_{f_1,z,g_1} \neq 0$ for some $z\neq 1$ we get a contradiction since $\mathbf{T}_{t_{\nu(\mathbf{u}^{\vee})}}^{-k}(y)\in \mathbf{U}_q^{J,+}$.

\end{proof}    
\end{prop}
\subsection{}
We also have \begin{align*}\mathbf{U}_q^J(0):&=\mathbf{U}_q^J\cap \langle h_{i,r}, K_i^{\pm}, C^{\pm 1/2}, D^{\pm}: i\in\mathring{I},\ r\in \mathbb{Z}\setminus \{0\} \rangle &\\&
=\langle h_{i,s}, K_i^{\pm}, C^{\pm 1/2}, D^{\pm}: i\in\mathring{I},\ s>0 \rangle. \end{align*}
Only the intersection $\mathbf{U}_q^J(-):=\mathbf{U}_q^J \cap \langle \chi_{i,r}^- : i\in\mathring{I},\ r\in \mathbb{Z} \rangle$ does not have a nice description in terms of Drinfeld-type generators. In certain important situations this is still possible. We begin by recalling the Levendorskii–Soibelman formula from \cite[Proposition 7]{B94A}, in the form required for the purposes of this article.
\begin{lem}\label{LSformel}
Let $\beta\in\mathring{R}$ and $r\geq 0$. Then $[C^{-1/2}h_{i,1},\chi_{\beta,r+1}^-]$ is a $\mathbb{C}(q)$-linear combination of ordered monomials $(\chi_{\tau_1,s_1+1}^{-})^{a_1}\cdots (\chi_{\tau_k,s_k+1}^{-})^{a_k}$ satisfying $$\delta <-\tau_1+(s_1+1)\delta<\cdots<-\tau_k+(s_k+1)\delta<-\beta+(r+1)\delta.$$
\qed
\end{lem}
\begin{prop}\label{algebrengleich23} The following three subspaces of $\mathbf{U}_q^J$ coincide
\begin{enumerate}
    \item the subalgebra generated by 
\begin{equation}\label{ghset4150}\{\chi_{\alpha,0}^-, \chi_{\beta,r+1}^-, \ \  \alpha,\beta \in \mathring{R}^+,\ \mathrm{supp}(\alpha)\subseteq J,\ r\geq 0\} \end{equation}
\item the subpace generated by 
\begin{equation}\label{huuugt4}\prod_{\substack{\beta\in \mathring{R}^+, r> 0 \\ \text{convexly ordered}}}(\chi_{\beta,r}^-)^{\varphi(\beta,r)} \prod_{\substack{\beta\in \mathring{R}^+,\  \mathrm{supp}(\beta)\subseteq J\\ \text{reversed convexly ordered}}}(\chi_{\beta,0}^-)^{f(\beta)}\end{equation}
 for various (finitely supported) functions $\varphi:\mathring{R}^+\times \mathbb{N}\rightarrow \mathbb{Z}_+$, $f:\mathring{R}^+\rightarrow \mathbb{Z}_+$ \vspace{0,2cm}
\item $\mathbf{U}_q^J \cap \langle \chi_{i,r}^- : i\in\mathring{I},\ r\in \mathbb{Z} \rangle.$
\end{enumerate}
Moreover, if $\mathfrak{p}_J$ is maximal, i.e. $I\setminus J = \lbrace i_0 \rbrace$, $i_0\neq 0$ and $\mathbf{a}_{i_0}(\delta)\leq 2$, then the subalgebras above also coincide with
\begin{enumerate}
\item[(4)] the subalgebra generated by 
\begin{equation}\label{ghset41151}\{\chi_{\gamma_0,r+1}^-,\ \chi_{i,r+1}^-,\  \chi_{j,r}^- : \ i \in \mathring{I}\setminus (\mathring{I}\cap J),\ j\in\mathring{I}\cap J,\ r\geq 0\}. \end{equation} 

\end{enumerate}
\end{prop}
\begin{proof}
    Again by \cite[Lemma 5.4]{Da15a} and Lemma~\ref{lemsubset667} we can derive that  \begin{align}\label{eqsch5r1}
        \langle \chi_{\alpha,0}^-, \chi_{\beta,r+1}^-, \ \  \alpha,\beta \in \mathring{R}^+,\ \mathrm{supp}(\alpha)\subseteq J,\ r\geq 0\rangle &\subseteq \mathbf{U}_q^J \cap \langle \chi_{i,r}^- : i\in\mathring{I},\ r\in \mathbb{Z} \rangle, 
        %\\
        %&\subseteq\{y\in \mathbf{U}^J_q : \mathbf{T}_{t_{\nu(\mathbf{u}^{\vee})}}^ky\in \mathbf{U}_q^-, k >>0\}
    \end{align}
\eqref{ghset41151} is a subset of \eqref{ghset4150} and the subspace generated by \eqref{huuugt4} is contained in the subalgera generated by \eqref{ghset4150}. 
Again we are left to show
\begin{itemize}
   \item The subalgebra in (3) is contained in the subspace generated by \eqref{huuugt4}. \vspace{0,2cm}
    \item Under the hypothesis of the proposition, the set \eqref{ghset4150} is contained in the subalgebra generated by the elements \eqref{ghset41151} which we call $\mathcal{B}$. 
   \end{itemize} 
The proof of the first part is similar to the proof of Proposition~\ref{algebrengleich3} and we omit the details. For the second part, note that the elements $\chi_{\alpha,0}^-$ for all $\alpha\in\mathring{R}^+$ with $\mathrm{supp}(\alpha)\subseteq J$ are clearly contained in $\mathcal{B}$. We define an order $\leq $ on $\{\chi_{\beta,r+1}^-: \beta\in\mathring{R}^+,\ r\geq 0\}$ as follows
$$\chi_{\gamma,s+1}^-<\chi_{\beta,r+1}^-:\iff s<r \text{ or $r=s$ and } \mathrm{ht}(\gamma)<\mathrm{ht}(\beta).$$
We extend this order to an arbitrary total order (on the above set) and proceed by induction. Fix a positive root $\beta\in\mathring{R}^+$ and $r\geq 0$. From Remark~\ref{trisp2} we know that $\chi_{\beta,r+1}^-\in \mathcal{B}$ if $r=0$ and the induction begins. For $r>0$ we get from  Lemma~\ref{LSformel} and our induction hypothesis that there exists $w\in \mathbb{C}(q)$ with 
$$[C^{-1/2}h_{i,1},\chi_{\beta,r}^-]-w \cdot \chi_{\beta,r+1}^-\in \sum_{} \mathbb{C}(q) \chi^-_{\tau_1,s_1+1}\cdots \chi^-_{\tau_k,s_k+1}\in \mathcal{B},\ \ \chi_{\beta,r}^-\in \mathcal{B}.$$
Now if we choose $i\in\mathring{I}$ with $\beta(h_i)\neq 0$ (which always exists) we can also guarantee that $w\neq 0$; this information can be obtained by considering the specialization $q\rightarrow 1$. Hence it remains to show $[C^{-1/2}h_{i,1},\chi_{\beta,r}^-]\in \mathcal{B}.$
We write $\chi_{\beta,r}^-$ as a linear combination of monomials in the generators of $\mathcal{B}$ and fix a non-zero summand $\chi^-_{\kappa_1,p_1}\cdots \chi^-_{\kappa_k,p_k}$, i.e. 
each $\chi^-_{\kappa_{\ell},p_{\ell}}$, $1\leq \ell\leq k$ is contained in \eqref{ghset41151}. We will show $[C^{-1/2}h_{i,1},\chi^-_{\kappa,p_m}]\in \mathcal{B}$ for all $1\leq m\leq k$ where the only non-trivial case is $-\kappa_m+p_m\delta=-\gamma_0+(p+1)\delta$ (see the defining relations stated in Theorem~\ref{definingrelationQAA}).
In particular $p+1\leq r$ and $\mathrm{ht}(\beta)\geq \mathrm{ht}(\gamma_0)$. Again we write with Lemma~\ref{LSformel}
$$[C^{-1/2} h_{i,1},\chi^-_{\gamma_0,p+1}]\in \mathbb{C}(q) \chi^-_{\gamma_0,p+2}+ \sum \mathbb{C}(q) \chi^-_{\tau'_1,s'_1+1}\cdots \chi^-_{\tau'_{\ell},s_{\ell}'+1},\ \ s_{\iota}'+1\leq p+2\leq r+1.$$ 
Once more by our induction hypothesis each monomial on the right hand side is contained in $\mathcal{B}$, since either $s_{\iota}'+1<p+2\leq r+1$ for all $1\leq \iota\leq \ell$ or there exists $\iota_0\in\{1,\dots,\ell\}$ with $s_{\iota_0}'+1=p+2=r+1$ in which case we must have $\mathrm{ht}(\kappa_{\iota}')<\mathrm{ht}(\gamma_0)\leq\mathrm{ht}(\beta)$ for all $\iota$.
\end{proof}
\begin{rem} For $\alpha\in \mathring{R}^+$ with $\text{supp}(-\alpha+\delta) \not \subseteq J$ and $|\text{supp}(\alpha)\cap (\mathring{I}\setminus \mathring{I}\cap J)|\geq 2 $, the element $\chi_{\alpha,1}^-$ is never contained in the subalgebra generated by \eqref{ghset41151}. This follows by weight reasons and hence the equality of the algebras generated by \eqref{ghset4150} and \eqref{ghset41151} respectively fails in general. For example, we can choose $\alpha=\alpha_2+\alpha_3+\alpha_4$ for $J=\{0,1,3\}$ in type $B_4^{(1)}.$
\end{rem} 
As a consequence of Proposition~\ref{algebrengleich3}, Proposition~\ref{algebrengleich23} and Remark~\ref{reprdspa} we get the following second triangular decomposition. 
\begin{cor}\label{trisecond}
We have an isomorphism of vector spaces 
$$\mathbf{U}_q^J\cong \mathbf{U}_q^J(-)\otimes \mathbf{U}_q^J(0)\otimes \mathbf{U}_q^J(+)$$
\qed
\end{cor}
\section{Useful Identities in quantum affine algebras}\label{section7}
In this section, we study the action of the Lusztig operators on $\mathbf{U}_q^J(0)$ modulo certain terms contained in $\mathbf{U}_q^J(+)$. A braid group action on the space of algebra maps from the algebra generated by $h_{i,r}$, where $i \in \mathring{I}$ and $r > 0$, to $\mathbb{C}(q)$ was studied in \cite{C02B} to obtain sufficient conditions for the cyclicity of tensor products of irreducible representations of quantum affine algebras. In this section we will see that same formulas apply when studying the action on $\mathbf{U}_q^J(0)$ itself. Our computations are direct, and for our purposes, we must also control the precise modulo terms in a certain sense. For an alternative approach, see \cite[Section 6]{FWW24}.
\subsection{}
We denote by $\mathcal{I}^{s}_{i,r}$ the left ideal in $\mathbf{U}_q$ generated by the elements $\chi_{i,k}^+$, $s\leq k\leq r$. \textit{All calculations in this section are modulo the ideal generated by $(C-1)$.} We will derive certain identities for (parabolic) quantum affine algebras which are needed for the study of finite-dimensional type 1 representations in the next section. It will be convinient to introduce
$$\bar{\Psi}'_{j,a}=\bar{\Psi}_{j,a},\ a>0,\ \bar{\Psi}'_{j,0}=(q_j-q_j^{-1})^{-1}.$$ 
The following proposition is needed later.
\begin{prop}\label{propidentw}
Let $i,j\in\mathring{I}$, $r,s,\ell\in\mathbb{Z}_+$ and $0\leq k\leq r$. 
\begin{enumerate}
    \item  Modulo the ideal $\mathcal{I}_{i,r}^{k+1}$ we have
    $$[\chi_{i,k}^+,\bar{\Psi}'_{j,r-k}]\chi_{i,s}^{-}\equiv -[a_{ij}]_{i}\sum_{p=1}^{r-k}q_i^{-a_{ij}(p-1)}\bar{\Psi}'_{j,r-k-p}\phi^+_{i,k+p+s}.$$
    \item Modulo the ideal $\mathcal{I}_{i,r}^{k+1}$ and $\mathcal{I}_{i,0}^{0}$ we have
    $$\chi_{i,0}^+[\chi_{i,k}^+,\bar{\Psi}'_{j,r-k}](\chi_{i,0}^-)^2\equiv -[a_{ij}]_{i}\sum_{p=1}^{r-k} q_i^{-a_{ij}(p-1)}A^{i,j}_{r-k-p,k+p},$$
    where 
\begin{align}\label{aijdef1}A^{i,j}_{r-k-p,k+p}=\bar{\Psi}'_{j,r-k-p}&\left([\chi_{i,0}^+,\phi^+_{i,k+p}]\chi_{i,0}^-+2\phi_{i,k+p}^+[K_i;0]\right)&\\&\notag+[\chi_{i,0}^+,\bar{\Psi}'_{j,r-k-p}]\left(\chi_{i,0}^-\phi_{i,k+p}^++\phi_{i,k+p}^+\chi_{i,0}^-\right).\end{align}
\item Modulo the ideal $\mathcal{I}_{i,r}^{k+1}$ we have
$$\big[[\chi_{i,k}^+,\bar{\Psi}'_{j,r-k}],\phi_{i,\ell}^+\big]\chi_{i,s}^{-}
\equiv-[a_{ij}]_{i}(q_i-q_i^{-1})\left(\sum_{p=1}^{r-k}q_i^{-a_{i,j}(p-1)}\bar{\Psi}'_{j,r-k-p}[\chi_{i,k+p}^+,\phi_{i,\ell}^+]\chi_{i,s}^{-}\right)$$

%$$[\chi_{i,k}^+,\bar{\Psi}_{j,r-k}]\phi_{i,p}^+\chi_{i,s}^{-}=-[a_{j,i}]_j(q_j-q_j^{-1})\left(\sum_{\ell=1}^{r-k-1}q_j^{-(r-k-\ell-1)a_{j,i}}\bar{\Psi}_{j,\ell}\left([\chi_{i,r-\ell}^+,\phi_{i,p}^+]\chi_{i,s}^{-}+\frac{\phi_{i,p}^+\phi_{i,r+s-\ell}^+}{(q_i-q_i^{-1})}\right)\right)$$
%$$-[a_{j,i}]_jq_j^{-(r-k-1)a_{j,i}}\left([\chi_{i,r}^+,\phi_{i,p}^+]\chi_{i,s}^{-}+\frac{\phi_{i,p}^+\phi_{i,r+s}^+}{(q_i-q_i^{-1})}\right)$$
\end{enumerate}
\begin{proof}
\textit{Part (1):} The proof is by downward induction on $k$. The case $k=r$ is obvious and for $k=r-1$ we get
%and all calculations in the rest of the proof are modulo $\mathcal{I}_{i,\leq r}^{\geq k+1}$.
$$[\chi_{i,r-1}^+,\bar{\Psi}'_{j,1}]\chi_{i,s}^{-}=[\chi_{i,r-1}^+,h_{j,1}]\chi_{i,s}^{-}=-[a_{ji}]_{j}\chi_{i,r}^+\chi_{i,s}^{-}\equiv -[a_{ij}]_{i}\bar{\Psi}_{j,0}'\phi^+_{i,r+s}.$$
So the induction begins.
Now with \cite[Lemma 4.1]{B94} we conclude
$$[\chi_{i,k}^+,\bar{\Psi}'_{j,r-k}]\chi_{i,s}^{-}=-[a_{ji}]_{j}\left(\sum_{\ell=1}^{r-k}q_i^{a_{ij}(\ell-1)}(q_j-q_j^{-1})\chi_{i,k+\ell}^+\bar{\Psi}'_{j,r-k-\ell}\right)\chi_{i,s}^{-}$$
$$\equiv -[a_{ij}]_{i}\left(\sum_{\ell=1}^{r-k}q_i^{a_{ij}(\ell-1)}\bar{\Psi}'_{j,r-k-\ell}\phi^+_{i,k+\ell+s}\right)-[a_{ij}]_{i}\sum_{\ell=1}^{r-k}q_i^{a_{ij}(\ell-1)}(q_i-q_i^{-1})[\chi_{i,k+\ell}^+,\bar{\Psi}'_{j,r-k-\ell}]\chi_{i,s}^{-}.$$
\iffalse
So for example, for $k=r-2$ we get
$$[\chi_{i,r-2}^+,\bar{\Psi}_{j,2}]\chi_{i,s}^{-}=-[a_{ij}]_{i}(\bar{\Psi}_{j,1}\phi^+_{r-1+s}+d\frac{\phi^+_{i,r+s}}{(q_j-q_j^{-1})})-[a_{ij}]_{i}(q_i-q_i^{-1})[\chi_{i,r-1}^+,\bar{\Psi}_{j,1}]\chi_{i,s}^{-}$$
$$=-[a_{ij}]_{i}(\bar{\Psi}_{j,1}\phi^+_{r-1+s}+q_i^{-a_{ij}}\frac{\phi^+_{i,r+s}}{(q_j-q_{j}^{-1})})$$
So for example, for $k=r-3$ we get
$$[\chi_{i,r-3}^+,\bar{\Psi}_{j,3}]\chi_{i,s}^{-}=-[a_{ij}]_{i}(\bar{\Psi}_{j,2}\phi^+_{i,r-2+s}+d^{}\bar{\Psi}_{j,1}\phi^+_{i,r-1+s}+d^{2}\frac{\phi^+_{i,r+s}}{(q_j-q_j^{-1})})
$$
$$-[a_{ij}]_{i}((q_i-q_i^{-1})[\chi_{i,r-2}^+,\bar{\Psi}_{j,2}]\chi_{i,s}^{-}+d(q_i-q_i^{-1})[\chi_{i,r-1}^+,\bar{\Psi}_{j,1}]\chi_{i,s}^{-})$$
$$=-[a_{ij}]_{i}(\bar{\Psi}_{j,2}\phi^+_{i,r-2+s}+q_i^{-a_{ij}}\bar{\Psi}_{j,1}\phi^+_{i,r-1+s}+q_i^{-2a_{ij}}\frac{\phi^+_{i,r+s}}{(q_j-q_{j}^{-1})})$$
\fi
Now we can apply our induction hypothesis to the second sum and get 
$$[\chi_{i,k}^+,\bar{\Psi}'_{j,r-k}]\chi_{i,s}^{-}\equiv-[a_{ij}]_{i}\left(\sum_{\ell=1}^{r-k}q_i^{a_{ij}(\ell-1)}\bar{\Psi}'_{j,r-k-\ell}\phi^+_{i,k+\ell+s}\right)
$$
$$+[a_{ij}]^2_{i}\sum_{\ell=1}^{r-k}q_i^{a_{ij}(\ell-1)}(q_i-q_i^{-1})\left(\sum_{\ell'=1}^{r-k-\ell}q_i^{-a_{ij}(\ell'-1)}\bar{\Psi}'_{j,r-k-\ell-\ell'}\phi^+_{i,k+\ell+\ell'+s}\right).$$
This finishes the proof by observing that the coefficient of $\bar{\Psi}'_{j,r-k-p}\phi^+_{i,k+p+s}$ ($1\leq p\leq r-k$) in the above expression is given by
\begin{equation}\label{identt1}-[a_{ij}]_{i}\left(q_i^{a_{ij}(p-1)}-[a_{ij}]_{i}(q_i-q_i^{-1})\sum_{\ell=1}^{p-1}q_i^{a_{ij}(2\ell-p)}\right)=-[a_{ij}]_{i}q_i^{-a_{ij}(p-1)}.\end{equation}

\textit{Part (2):} We obtain again with \cite[Lemma 4.1]{B94}
$$\chi_{i,0}^+[\chi_{i,k}^+,\bar{\Psi}_{j,r-k}'](\chi_{i,0}^-)^2=-[a_{ji}]_{j}\chi_{i,0}^+\left(\sum_{\ell=1}^{r-k}q_i^{a_{ij}(\ell-1)}(q_j-q_j^{-1})\chi_{i,k+\ell}^+\bar{\Psi}'_{j,r-k-\ell}\right)(\chi_{i,0}^-)^2$$
$$=-[a_{ij}]_{i}(q_i-q_i^{-1})\sum_{\ell=1}^{r-k}q_i^{a_{ij}(\ell-1)}\left(\chi_{i,0}^+[\chi_{i,k+\ell}^+,\bar{\Psi}'_{j,r-k-\ell}](\chi_{i,0}^-)^2+\chi_{i,0}^+\bar{\Psi}'_{j,r-k-\ell}\chi_{i,k+\ell}^+(\chi_{i,0}^-)^2\right).$$
%$$=-[a_{ij}]_{i}(q_i-q_i^{-1})\sum_{\ell=1}^{r-k}q_i^{a_{ij}(\ell-1)}\left(\chi_{i,0}^+[\chi_{i,k+\ell}^+,\bar{\Psi}'_{j,r-k-\ell}](\chi_{i,0}^-)^2\right)$$
%$$-[a_{ij}]_{i}(q_i-q_i^{-1})\sum_{\ell=1}^{r-k}q_i^{a_{ij}(\ell-1)}\left(\chi_{i,0}^+\bar{\Psi}'_{j,r-k-\ell}\chi_{i,k+\ell}^+(\chi_{i,0}^-)^2\right)$$
Since $(q_i-q_i^{-1})\chi_{i,k+\ell}^+(\chi_{i,0}^-)^2\equiv \phi_{i,k+\ell}^+\chi_{i,0}^-+\chi_{i,0}^-\phi_{i,k+\ell}^+$ and $\chi_{i,0}^+\phi^+_{i,k+\ell}\equiv 0$ we get
$$\chi_{i,0}^+[\chi_{i,k}^+,\bar{\Psi}_{j,r-k}'](\chi_{i,0}^-)^2\equiv -[a_{ij}]_{i}(q_i-q_i^{-1})\sum_{\ell=1}^{r-k}q_i^{a_{ij}(\ell-1)}\left(\chi_{i,0}^+[\chi_{i,k+\ell}^+,\bar{\Psi}'_{j,r-k-\ell}](\chi_{i,0}^-)^2\right)$$
$$-[a_{ij}]_{i}\sum_{\ell=1}^{r-k}q_i^{a_{ij}(\ell-1)}A^{i,j}_{r-k-\ell,k+\ell}.$$
Again by applying our induction hypothesis to the first sum and using once more \eqref{identt1} gives the claim.

\textit{Part (3):} A similar induction as in the proof of (1) shows that $[\chi_{i,k}^+,\bar{\Psi}'_{j,r-k}]\phi_{i,\ell}^+\chi_{i,s}^{-}$ is congruent to
$$-[a_{ij}]_{i}(q_i-q_i^{-1})\left(\sum_{p=1}^{r-k}q_i^{-a_{i,j}(p-1)}\bar{\Psi}'_{j,r-k-p}\left([\chi_{i,k+p}^+,\phi_{i,\ell}^+]\chi_{i,s}^{-}+\frac{\phi_{i,\ell}^+\phi_{i,k+p+s}^+}{(q_i-q_i^{-1})}\right)\right).$$
Now applying (1) gives the claim.
\end{proof}
\end{prop}
We obtain the following corollary.
\begin{cor}\label{identcor523} Let $i,j\in\mathring{I}$, $r,s\in\mathbb{Z}_+$ and $0\leq k\leq r$.
\begin{enumerate}
    \item Modulo the ideal $\mathcal{I}_{i,r}^{k}$ we have 
    $$[\chi_{i,k}^+,\phi^+_{j,r-k}]\chi_{i,s}^{-}\equiv -[a_{ij}]_{i}\left(\sum_{p=1}^{r-k}q_i^{-p a_{ij}}\phi^+_{j,r-k-p}\phi^+_{i,k+p+s}+\frac{\phi^+_{j,r-k}(\phi^+_{i,k+s}-\phi^-_{i,k+s})}{(q_i^{a_{ij}}+1)}\right)$$
    \item Let $\ell\in\{1,\dots,r\}$. Modulo the ideal $\mathcal{I}_{i, \ell}^{0}$ we have 
    $$A^{i,j}_{r-\ell,\ell}\equiv-q_i^{-1}[2]_{i}\bigg(\sum_{\ell_1=1}^{\ell}q_i^{-2\ell_1+1}\bar{\Psi}'_{j,r-\ell}\phi_{i,\ell-\ell_1}^+\phi_{i,\ell_1}^++[a_{ij}]_{i}\sum_{\ell_2=1}^{r-\ell}q_i^{-a_{ij}(\ell_2-1)}\bar{\Psi}'_{j,r-\ell-\ell_2}\phi_{i,\ell_2}^+\phi_{i,\ell}^+ $$
    $$-\bar{\Psi}'_{j,r-\ell}\phi_{i,\ell}^+[K_i;0]-[a_{ij}]_{i}(q_i^2-1)\sum_{\ell_2=1}^{r-\ell}\sum_{\ell_1=1}^{\ell}q_i^{-a_{ij}(\ell_2-1)-2\ell_1}\bar{\Psi}'_{j,r-\ell-\ell_2}\phi_{i,\ell-\ell_1}^+\phi_{i,\ell_2+\ell_1}^+ \bigg).$$
    \end{enumerate}
    \begin{proof}
        The first part can be derived from Proposition~\ref{propidentw}(1) and $\phi_{i,p}^+=(q_i-q_i^{-1})K_i\bar{\Psi}'_{i,p}$ for each $i\in\mathring{I}$ and $p\geq 0$. The second part results from substituting the formulas from (1), Proposition~\ref{propidentw}(1), and Proposition~\ref{propidentw}(3) into  $A^{i,j}_{r-\ell,\ell}$ (recall the definition from \eqref{aijdef1}). We omit the details.
    \end{proof}
\end{cor}
\subsection{} The first main result of this section is the following.
\begin{prop}\label{identfuerphil}
We have the following identities for $i,j\in\mathring{I}$
modulo $\mathcal{I}_{j,0}^{0}$ and $\mathcal{I}_{i,\infty}^{0}$ $$\mathbf{T}_i(\phi^+_j(u))= 
      \begin{cases}\phi^+_j(u),& \text{ if $a_{ij}=0$},\\
      \phi^+_j(u)\phi^+_i(q_iu),& \text{ if $a_{ij}=-1$},\\
      \phi^+_i(q_i^2u)^{-1},& \text{ if $a_{ij}=2$},\\
\phi^+_j(u)\phi^+_i(u)\phi^+_i(q_i^2u),& \text{ if $a_{ij}=-2$},\\
         \end{cases}$$
\end{prop}
The proof will be given in the next three subsections where the first case is clear. 
\subsection{} Assume that $a_{ij}=-1$ and let $r\geq 1$. All the subsequent calculations are modulo the ideals $\mathcal{I}_{i,r}^{0}$ and $\mathcal{I}_{j,0}^{0}$. We have  $$ (q_j-q_{j}^{-1})^{-1}\mathbf{T}_i(\phi_{j,r}^+)=[\mathbf{T}_i(\chi_{j,0}^+),\mathbf{T}_i(\chi_{j,r}^-)]=[\chi_{i,0}^+\chi_{j,0}^+,\chi_{j,r}^{-}\chi_{i,0}^{-}]-q_i[\chi_{i,0}^+\chi_{j,0}^+,\chi_{i,0}^{-}\chi_{j,r}^{-}]$$
     $$-q_i^{-1}[\chi_{j,0}^+\chi_{i,0}^+,\chi_{j,r}^{-}\chi_{i,0}^{-}]+[\chi_{j,0}^+\chi_{i,0}^+,\chi_{i,0}^{-}\chi_{j,r}^{-}]$$
    where the second equation follows from $\mathbf{T}_i\mathbf{T}_{\varpi_j}=\mathbf{T}_{\varpi_j}\mathbf{T}_i$. Furthermore, \begin{align*}& \bullet \ [\chi_{i,0}^+\chi_{j,0}^+,\chi_{j,r}^{-}\chi_{i,0}^{-}]\equiv \chi_{i,0}^+\chi_{j,0}^+\chi_{j,r}^{-}\chi_{i,0}^{-}\equiv \chi_{i,0}^{+}\frac{\phi_{j,r}^+}{(q_j-q_j^{-1})}\chi_{i,0}^{-}&\\& \hspace{3,29cm}\equiv \frac{\phi_{j,r}^+}{(q_j-q_j^{-1})}[K_i;0]+\frac{1}{(q_j-q_j^{-1})}[\chi_{i,0}^{+},\phi_{j,r}^+]\chi_{i,0}^{-}&\\&\bullet \  [\chi_{i,0}^+\chi_{j,0}^+,\chi_{i,0}^{-}\chi_{j,r}^{-}]=\chi_{i,0}^+\chi_{i,0}^{-}\chi_{j,0}^+\chi_{j,r}^{-}\equiv \chi_{i,0}^+\chi_{i,0}^{-}\frac{\phi_{j,r}^+}{(q_j-q_j^{-1})}\equiv \frac{\phi_{j,r}^+}{(q_j-q_j^{-1})}[K_i;0]&\\&\bullet \  [\chi_{j,0}^+\chi_{i,0}^+,\chi_{j,r}^{-}\chi_{i,0}^{-}]\equiv \chi_{j,0}^+\chi_{i,0}^+\chi_{j,r}^{-}\chi_{i,0}^{-}\equiv \chi_{j,0}^+\chi_{j,r}^{-}[K_i;0]\equiv \frac{\phi_{j,r}^+}{(q_j-q_j^{-1})}[K_i;0]&\\& \bullet \ [\chi_{j,0}^+\chi_{i,0}^+,\chi_{i,0}^{-}\chi_{j,r}^{-}]\equiv \chi_{j,0}^+\chi_{i,0}^+\chi_{i,0}^{-}\chi_{j,r}^{-}\equiv \chi_{j,0}^+[K_i;0]\chi_{j,r}^{-}\equiv \frac{\phi_{j,r}^+}{(q_j-q_j^{-1})}[K_i;1]\end{align*}
 
    %$$=q_iK_j[E_i,\bar{\Psi}_{j,r}]\chi_{i,0}^{-}+q_i\frac{\phi_{j,r}^+}{(q_j-q_j^{-1})}\frac{K_i-K_i^{-1}}{(q_i-q_i^{-1})}$$
    % since 
     %$$[E_i,K_j\bar{\Psi}_{j,r}]\chi_{i,0}^{-}=(q_iK_jE_i\bar{\Psi}_{j,r}-K_j\bar{\Psi}_{j,r}E_i)\chi_{i,0}^{-}\equiv q_iK_j[E_i,\bar{\Psi}_{j,r}]\chi_{i,0}^{-}+(q_i-1)K_j\bar{\Psi}_{j,r}E_i\chi_{i,0}^{-}$$
    % $$\equiv q_iK_j[E_i,\bar{\Psi}_{j,r}]\chi_{i,0}^{-}+(q_i-1)\frac{\phi_{j,r}^+}{(q_j-q_j^{-1})}\frac{K_i-K_i^{-1}}{(q_i-q_i^{-1})}$$
  
 All together we obtain with Corollary~\ref{identcor523}(1)
    $$\mathbf{T}_i(\phi_{j,r}^+)\equiv ([K_i;0]-[K_i;-1])\phi_{j,r}^++[\chi_{i,0}^{+},\phi_{j,r}^+]\chi_{i,0}^{-}\equiv \sum_{p=0}^{r}q_i^{p}\phi^+_{j,r-p}\phi^+_{i,p}$$
 which finishes the proof in this case.   
 \subsection{}In this subsection, we prove the case $i=j$. To begin, we collect a few additional results. From $\mathbf{T}_i^{-1}\mathbf{T}_{\varpi_i}\mathbf{T}_i^{-1}=\mathbf{T}_{\varpi_i}^{-1}\prod_{i\neq j}\mathbf{T}_{\varpi_j}^{-a_{ij}}$ (see \cite[Section 2]{B94}) we deduce
$\mathbf{T}_i(\chi_{i,1}^+)=-K_i^{-}\chi_{i,1}^-$ and hence 
\begin{equation}\label{need1}\mathbf{T}_i(h_{i,1})=K_i[\mathbf{T}_i(\chi_{i,1}^+),\mathbf{T}_i(F_i)]=-q_i^{2}h_{i,1}+(q_i^2-q_i^{-2})\mathbf{T}_i(\chi_{i,1}^+)\chi_{i,0}^+.\end{equation}
Moreover, for $p\geq 1$ 
$$[\chi_{i,0}^+,\mathbf{T}_i(\phi_{i,p}^+)]=\mathbf{T}_i([-K_i^{-1}\chi_{i,0}^-,\phi_{i,p}^+])=K_i\mathbf{T}_i([\phi_{i,p}^+,\chi_{i,0}^-]).$$
Substituting the following formula,
$$[\phi_{i,p}^+,\chi_{i,0}^-]=-[2]_{i}\left(\sum_{k=1}^pq_i^{2(1-k)}(q_i-q_i^{-1})\phi_{i,p-k}^+\chi_{i,k}^-\right)-(1-q_i^{-2})\chi_{i,0}^{-}\phi_{i,p}^+$$
which can be derived from \cite[Lemma 4.1]{B94}, we obtain 
\begin{equation}\label{need2}\chi_{i,0}^+\mathbf{T}_i(\phi_{i,p}^+)=-[2]_{i}\left(\sum_{k=1}^pq_i^{2(1-k)+2}(q_i-q_i^{-1})K_i\mathbf{T}_i(\phi_{i,p-k}^+)\mathbf{T}_i(\chi_{i,k}^-)\right)+q_i^2\mathbf{T}_i(\phi_{i,p}^+)\chi_{i,0}^+.\end{equation}
The second part of following lemma proves Proposition~\ref{identfuerphil} for the case $i=j$.

 \begin{lem} Let $i\in\mathring{I}$ and $r\geq 0$. 
 \begin{enumerate}
% \item Modulo $\mathcal{I}_{i,\leq r}^{\geq 0}$ we have 
%\mathbf{T}_i(\chi^-_{i,r})\equiv 0$. 
     \item Modulo $\mathcal{I}_{i,r-1}^{0}$ we have 
     $$\sum_{p=0}^{r}q_i^{2(r-p)}\mathbf{T}_i(\phi_{i,p}^+)\phi_{i,r-p}^+\equiv \delta_{r,0}.$$
     \item Modulo the left ideal generated by $\chi_{i,s_1}^+\chi_{i,s_2}^+$, $0\leq s_1+s_2\leq r-1$ we have $$\mathbf{T}_i(\chi_{i,r}^-)\equiv-\sum_{p=0}^rq_i^{2(r-p)}\mathbf{T}_i(\phi_{i,p}^+)\chi_{i,r-p}^+.$$

 \end{enumerate}
     \begin{proof} %\textit{Part (1):} For $r=0$ this is immediate by definition. So let $r>0$ and note  $$\mathbf{T}_i(\chi^-_{i,r})=-[2]_{i}\mathbf{T}_i([h_{i,1},\chi^-_{i,r-1}])\equiv [2]_{i}\mathbf{T}_i(\chi^-_{i,r-1})\mathbf{T}_i(h_{i,1})\equiv -q_i^2[2]_{i}\mathbf{T}_i(\chi^-_{i,r-1})h_{i,1}$$
   % where the second step follows by induction and the last step follows from \eqref{need1}. Now the proof is complete, since the right hand side is obviously congruent to $0$ by  relation (4) of Theorem~\ref{definingrelationQAA}.
 The proof is by induction on $r$ where the case $r=0$ is obvious. Let $r\geq 1$ and assume that both identities hold for all natural numbers less than $r$. Using the induction hypothesis, we obtain modulo $\mathcal{I}_{i,r-1}^{0}$ the following
\begin{align*}\mathbf{T}_i(\phi_{i,r}^+)&=(q_i-q_i^{-1})[\mathbf{T}_i(\chi_{i,1}^+),\mathbf{T}_i(\chi_{i,r-1}^-)]\equiv -(q_i-q_i^{-1})\mathbf{T}_i(\chi_{i,r-1}^-)\mathbf{T}_i(\chi_{i,1}^+)&\\&
=q_i^2(q_i-q_i^{-1})K_i^{-1}\mathbf{T}_i(\chi_{i,r-1}^-)\chi_{i,1}^-
\equiv -(q_i-q_i^{-1})\sum_{p=0}^{r-1}q_i^{2(r-p)}K_i^{-1}\mathbf{T}_i(\phi_{i,p}^+)\chi_{i,r-1-p}^+\chi_{i,1}^- &\\&
\equiv -\sum_{p=0}^{r-1}q_i^{2(r-p)}K_i^{-1}\mathbf{T}_i(\phi_{i,p}^+)\phi_{i,r-p}^+
\end{align*}
which gives part (1). Moreover, using \eqref{need1} and induction we get modulo the left ideal generated by $\chi_{i,s_1}^+\chi_{i,s_2}^+$, $0\leq s_1+s_2\leq r-1$:
\begin{align*}\mathbf{T}_i(\chi_{i,r}^-)&=-\frac{1}{[2]_{i}}[\mathbf{T}_i(h_{i,1}),\mathbf{T}_i(\chi_{i,r-1}^-)]=\frac{q_i^2}{[2]_{i}}[h_{i,1},\mathbf{T}_i(\chi_{i,r-1}^-)]-(q_i-q_i^{-1})[\mathbf{T}_i(\chi_{i,1}^+)\chi_{i,0}^+,\mathbf{T}_i(\chi_{i,r-1}^-)]&\\& \equiv -\sum_{p=0}^{r-1}q_i^{2(r-p)}\frac{1}{[2]_{i}}[h_{i,1},\mathbf{T}_i(\phi_{i,p}^+)\chi_{i,r-1-p}^+]-(q_i-q_i^{-1})[\mathbf{T}_i(\chi_{i,1}^+)\chi_{i,0}^+,\mathbf{T}_i(\chi_{i,r-1}^-)]&\\& \equiv -\sum_{p=0}^{r-1}q_i^{2(r-p)}\frac{1}{[2]_{i}}\left([h_{i,1},\mathbf{T}_i(\phi_{i,p}^+)]\chi_{i,r-1-p}^++\mathbf{T}_i(\phi_{i,p}^+)[h_{i,1},\chi_{i,r-1-p}^+]\right)&\\& \hspace{2cm}+(q_i-q_i^{-1})\left([\mathbf{T}_i(\chi_{i,r-1}^-),\mathbf{T}_i(\chi_{i,1}^+)]\chi_{i,0}^++\mathbf{T}_i(\chi_{i,1}^+)[\mathbf{T}_i(\chi_{i,r-1}^-),\chi_{i,0}^+]\right)&\\& = -\sum_{p=0}^{r-1}q_i^{2(r-p)}\frac{1}{[2]_{i}}[h_{i,1},\mathbf{T}_i(\phi_{i,p}^+)]\chi_{i,r-1-p}^+-\sum_{p=0}^{r}q_i^{2(r-p)}\mathbf{T}_i(\phi_{i,p}^+)\chi_{i,r-p}^+&\\& \hspace{2cm}+(q_i-q_i^{-1})\mathbf{T}_i(\chi_{i,1}^+)[\mathbf{T}_i(\chi_{i,r-1}^-),\chi_{i,0}^+]&\\& 
\equiv -\sum_{p=0}^{r}q_i^{2(r-p)}\mathbf{T}_i(\phi_{i,p}^+)\chi_{i,r-p}^{+}+(q_i-q_i^{-1})\mathbf{T}_i(\chi_{i,1}^+)\chi_{i,0}^+\mathbf{T}_i(\chi_{i,r-1}^-)\end{align*}
 where the last congruence equation follows from the induction hypothesis that $\mathbf{T}_i(\phi_{i,p}^+)$ can be written as a polynomial expression in the set $\{K_i^{\pm},\phi^+_{i,k}: k\geq 1\}$ plus terms in the ideal $\mathcal{I}_{i,p-1}^{0}$. Now using once more induction we get 
 $$\chi_{i,0}^+\mathbf{T}_i(\chi_{i,r-1}^-)\equiv -\sum_{p=0}^{r-1}q_i^{2(r-1-p)}\chi_{i,0}^+\mathbf{T}_i(\phi_{i,p}^+)\chi_{i,r-1-p}^+\equiv -\sum_{p=1}^{r-1}q_i^{2(r-1-p)}[\chi_{i,0}^+,\mathbf{T}_i(\phi_{i,p}^+)]\chi_{i,r-1-p}^+$$
 and hence $\chi_{i,0}^+\mathbf{T}_i(\chi_{i,r-1}^-)\equiv 0$ by \eqref{need2}. This finishes the proof.
     \end{proof}
 \end{lem}
\subsection{} Assume that $a_{ij}=-2$ and $r\geq 1$. The following calculations are modulo $\mathcal{I}_{i,r}^{ 0}$ and $\mathcal{I}_{j, 0}^{0}$. We have 
\begin{align*} \bar{\Psi}_{j,0}' \mathbf{T}_i(\phi_{j,r}^+)&=[\mathbf{T}_i(\chi_{j,0}^+),\mathbf{T}_i(\chi_{j,r}^-)]
%$$\left[(\chi_{i,0}^+)^{(2)}\chi_{j,0}^+-q_i^{-1}\chi_{i,0}^+\chi_{j,0}^+\chi_{i,0}^+ +q_i^{-2}\chi_{j,0}^+(\chi_{i,0}^+)^{(2)}, \chi_{j,r}^-(\chi_{i,0}^-)^{(2)}-q_i\chi_{i,0}^-\chi_{j,r}^-\chi_{i,0}^-+q_i^2(\chi_{i,0}^-)^{(2)}\chi_{j,r}^-\right]$$
=\left[(\chi_{i,0}^+)^{(2)}\chi_{j,0}^+, \chi_{j,r}^-(\chi_{i,0}^-)^{(2)}\right]&\\&-q_i \left[(\chi_{i,0}^+)^{(2)}\chi_{j,0}^+, x_{i,0}^-\chi_{j,r}^-\chi_{i,0}^-\right]+q_i^2 \left[(\chi_{i,0}^+)^{(2)}\chi_{j,0}^+,(\chi_{i,0}^-)^{(2)}\chi_{j,r}^-\right]&\\& -q_i^{-1} \left[ \chi_{i,0}^+\chi_{j,0}^+\chi_{i,0}^+,\chi_{j,r}^-(\chi_{i,0}^-)^{(2)}\right] + \left[ \chi_{i,0}^+\chi_{j,0}^+\chi_{i,0}^+, \chi_{i,0}^-\chi_{j,r}^-\chi_{i,0}^-\right] &\\&-q_i \left[\chi_{i,0}^+\chi_{j,0}^+\chi_{i,0}^+,(\chi_{i,0}^-)^{(2)}\chi_{j,r}^-\right] +q_i^{-2} \left[ \chi_{j,0}^+(\chi_{i,0}^+)^{(2)},\chi_{j,r}^-(\chi_{i,0}^-)^{(2)}\right]&\\& -q_i^{-1} \left[ \chi_{j,0}^+(\chi_{i,0}^+)^{(2)}, \chi_{i,0}^-\chi_{j,r}^-\chi_{i,0}^-\right] + \left[\chi_{j,0}^+(\chi_{i,0}^+)^{(2)},(\chi_{i,0}^-)^{(2)}\chi_{j,r}^-\right] \end{align*}
Moreover, 
\begin{align*}& \bullet \  (q_j-q_j^{-1})\left[(\chi_{i,0}^+)^{(2)}\chi_{j,0}^+, \chi_{j,r}^-(\chi_{i,0}^-)^{(2)}\right] \equiv \dfrac{1}{[2]_{i}^2} (q_j-q_j^{-1})(\chi_{i,0}^+)^{2}\chi_{j,0}^+\chi_{j,r}^-(\chi_{i,0}^-)^{2}&\\&
\hspace{0,5cm} \equiv \dfrac{1}{[2]_{i}^2} (\chi_{i,0}^+)^{2}\phi_{j,r}^+(\chi_{i,0}^-)^{2}= \dfrac{1}{[2]_{i}^2} \left (\chi_{i,0}^+\phi_{j,r}^+\chi_{i,0}^+(\chi_{i,0}^-)^{2} +  \chi_{i,0}^+\left[\chi_{i,0}^+,\phi_{j,r}^+\right](\chi_{i,0}^-)^2\right )&\\&\hspace{0,5cm} = \dfrac{1}{[2]_{i}^2} \left (\chi_{i,0}^+\phi_{j,r}^+\chi_{i,0}^-\chi_{i,0}^+\chi_{i,0}^- + \chi_{i,0}^+\phi_{j,r}^+[K_i;0]\chi_{i,0}^-  +  \chi_{i,0}^+\left[\chi_{i,0}^+,\phi_{j,r}^+\right](\chi_{i,0}^-)^2\right )&\\&\hspace{0,5cm} \equiv \dfrac{1}{[2]_{i}^2} \left (\chi_{i,0}^+\phi_{j,r}^+\chi_{i,0}^-[K_i;0] + \chi_{i,0}^+\phi_{j,r}^+[K_i;0]\chi_{i,0}^-  +  \chi_{i,0}^+\left[\chi_{i,0}^+,\phi_{j,r}^+\right](\chi_{i,0}^-)^2\right )&\\&\hspace{0,5cm}= \dfrac{1}{[2]_{i}^2} \left ([K_i;0] \chi_{i,0}^+\phi_{j,r}^+\chi_{i,0}^- + [K_i;-2] \chi_{i,0}^+\phi_{j,r}^+\chi_{i,0}^-  +  \chi_{i,0}^+\left[\chi_{i,0}^+,\phi_{j,r}^+\right](\chi_{i,0}^-)^2\right )&\\&\hspace{0,5cm}\equiv \dfrac{1}{[2]_{i}} [K_i;-1] \chi_{i,0}^+\phi_{j,r}^+\chi_{i,0}^-  +  \dfrac{1}{[2]_{i}^2} \chi_{i,0}^+\left[\chi_{i,0}^+,\phi_{j,r}^+\right](\chi_{i,0}^-)^2&\\\\&
%Furthermore, modulo $E_j,x_{i,p}^+, \ 0\leq p \leq r.$
\bullet \ (q_j-q_j^{-1})\left[(\chi_{i,0}^+)^{(2)}\chi_{j,0}^+, \chi_{i,0}^-\chi_{j,r}^-\chi_{i,0}^-\right]\equiv  (q_j-q_j^{-1}) \dfrac{1}{[2]_{i}}(\chi_{i,0}^+)^{2}\chi_{j,0}^+\chi_{i,0}^-\chi_{j,r}^-\chi_{i,0}^-&\\&\hspace{0,5cm}
\equiv \dfrac{1}{[2]_{i}} (\chi_{i,0}^+)^{2}\chi_{i,0}^-\phi^+_{j,r}\chi_{i,0}^-= \dfrac{1}{[2]_{i}} \left ( \chi_{i,0}^+\chi_{i,0}^-\chi_{i,0}^+\phi^+_{j,r}\chi_{i,0}^- + \chi_{i,0}^+[K_i;0]\phi^+_{j,r}\chi_{i,0}^- \right )&\\&\hspace{0,5cm}
= \dfrac{1}{[2]_{i}} \left ( \chi_{i,0}^+\chi_{i,0}^-\chi_{i,0}^+\phi^+_{j,r}\chi_{i,0}^- + [K_i;-2]\chi_{i,0}^+\phi^+_{j,r}\chi_{i,0}^- \right )&\\&\hspace{0,5cm}
= \dfrac{1}{[2]_{i}} \left ( \chi_{i,0}^-(\chi_{i,0}^+)^2\phi^+_{j,r}\chi_{i,0}^- + [2]_{i}[K_i;-1]\chi_{i,0}^+\phi^+_{j,r}\chi_{i,0}^- \right )&\\&\hspace{0,5cm} \equiv [K_i;-1]\chi_{i,0}^+\phi^+_{j,r}\chi_{i,0}^- \ \ (\text{see Corollary~\ref{identcor523}(1) and note that $[\chi_{i,0}^+, \phi_{j,r}^+] \equiv 0$})&\\\\&
%jjj
 \bullet \  (q_j-q_j^{-1})\left[(\chi_{i,0}^+)^{(2)}\chi_{j,0}^+,(\chi_{i,0}^-)^{(2)}\chi_{j,r}^-\right] \equiv \dfrac{1}{[2]_{i}^2}(q_j-q_j^{-1})(\chi_{i,0}^+)^{2}\chi_{j,0}^+(\chi_{i,0}^-)^{2}\chi_{j,r}^-&\\&\hspace{0,5cm}\equiv  \dfrac{1}{[2]_{i}^2}(\chi_{i,0}^+)^{2}(\chi_{i,0}^-)^{2}\phi_{j,r}^+ 
 \equiv\dfrac{1}{[2]_{i}}[K_i;0][K_i;-1]\phi_{j,r}^+&\\\\&
 %jjj
 \bullet \ \left[ \chi_{i,0}^+\chi_{j,0}^+\chi_{i,0}^+,\chi_{j,r}^-(\chi_{i,0}^-)^{(2)}\right] \equiv \dfrac{1}{[2]_{i}}\chi_{i,0}^+\chi_{j,0}^+\chi_{i,0}^+\chi_{j,r}^-(\chi_{i,0}^-)^{2}=\dfrac{1}{[2]_{i}}\chi_{i,0}^+\chi_{j,0}^+\chi_{j,r}^-\chi_{i,0}^+(\chi_{i,0}^-)^{2} &\\&\hspace{0,5cm}
  = \dfrac{1}{[2]_{i}}\left ( \chi_{i,0}^+\chi_{j,0}^+\chi_{j,r}^-\chi_{i,0}^-\chi_{i,0}^+\chi_{i,0}^- + \chi_{i,0}^+\chi_{j,0}^+\chi_{j,r}^-[K_i;0]\chi_{i,0}^- \right )&\\&\hspace{0,5cm}
  \equiv \dfrac{1}{[2]_{i}}\left ([K_i;0] \chi_{i,0}^+\chi_{j,0}^+\chi_{j,r}^-\chi_{i,0}^- + [K_i;-2]\chi_{i,0}^+\chi_{j,0}^+\chi_{j,r}^-\chi_{i,0}^- \right )&\\& \hspace{0,5cm}=[K_i,-1]\chi_{i,0}^+\chi_{j,0}^+\chi_{j,r}^-\chi_{i,0}^-\equiv\dfrac{1}{(q_j-q_j^{-1})}[K_i,-1]\chi_{i,0}^+\phi_{j,r}^+\chi_{i,0}^- &\\\\&
  %zzz
 \bullet \ \left[\chi_{i,0}^+\chi_{j,0}^+\chi_{i,0}^+, \chi_{i,0}^-\chi_{j,r}^-\chi_{i,0}^-\right] \equiv \chi_{i,0}^+\chi_{j,0}^+\chi_{i,0}^+\chi_{i,0}^-\chi_{j,r}^-\chi_{i,0}^- &\\&\hspace{0,5cm}= \chi_{i,0}^+\chi_{j,0}^+\left(\chi_{i,0}^-\chi_{i,0}^+\chi_{j,r}^-\chi_{i,0}^- + [K_i;0]\chi_{j,r}^-\chi_{i,0}^-\right)  \equiv [K_i;0]\left(\chi_{i,0}^+\chi_{j,0}^+\chi_{i,0}^-\chi_{j,r}^-+ \chi_{i,0}^+\chi_{j,0}^+\chi_{j,r}^-\chi_{i,0}^-\right)&\\& \hspace{0,5cm} \equiv \dfrac{1}{(q_j-q_j^{-1})}[K_i;0]\left ( [K_i;0]\phi_{j,r}^+ + \chi_{i,0}^+\phi_{j,r}^+\chi_{i,0}^- \right )&\\\\&
 %ttt
 \bullet \ \left[\chi_{i,0}^+\chi_{j,0}^+\chi_{i,0}^+,(\chi_{i,0}^-)^{(2)}\chi_{j,r}^-\right]\equiv \dfrac{1}{[2]_{i}} \chi_{i,0}^+\chi_{j,0}^+\chi_{i,0}^+(\chi_{i,0}^-)^{2}\chi_{j,r}^- &\\&\hspace{0,5cm}\equiv \dfrac{1}{[2]_{i}} \chi_{i,0}^+\chi_{j,0}^+\left  ( \chi_{i,0}^-[K_i;0]\chi_{j,r}^- + [K_i;0]\chi_{i,0}^-\chi_{j,r}^-\right )&\\&\hspace{0,5cm}=\dfrac{1}{[2]_{i}} \left  ( [K_i;2]\chi_{i,0}^+\chi_{j,0}^+\chi_{i,0}^-\chi_{j,r}^- + [K_i;0]\chi_{i,0}^+\chi_{j,0}^+\chi_{i,0}^-\chi_{j,r}^-\right )&\\&\hspace{0,5cm}=[K_i;1]\chi_{i,0}^+\chi_{j,0}^+\chi_{i,0}^-\chi_{j,r}^-\equiv \dfrac{1}{(q_j-q_j^{-1})}[K_i;1][K_i;0]\phi_{j,r}^+&\\\\&
%rrr
 \bullet \ \left[ \chi_{j,0}^+(\chi_{i,0}^+)^{(2)},\chi_{j,r}^-(\chi_{i,0}^-)^{(2)}\right] \equiv \dfrac{1}{[2]_{i}^2} \chi_{j,0}^+(\chi_{i,0}^+)^{2}\chi_{j,r}^-(\chi_{i,0}^-)^{2}= \dfrac{1}{[2]_{i}^2} \chi_{j,0}^+\chi_{j,r}^-(\chi_{i,0}^+)^{2}(\chi_{i,0}^-)^{2}&\\&\hspace{0,5cm}\equiv \dfrac{1}{[2]_i(q_j-q_j^{-1})}[K_i;-1][K_i;0]  \phi_{j,r}^+&\\\\&
 \bullet \ \left[ \chi_{j,0}^+(\chi_{i,0}^+)^{(2)}, x
 \chi_{i,0}^-\chi_{j,r}^-\chi_{i,0}^-\right]\equiv \dfrac{1}{[2]_{i}}\chi_{j,0}^+(\chi_{i,0}^+)^{2} \chi_{i,0}^-\chi_{j,r}^-\chi_{i,0}^- &\\&\hspace{0,5cm}
 = \dfrac{1}{[2]_{i}} \left ( \chi_{j,0}^+\chi_{i,0}^+\chi_{i,0}^-\chi_{j,r}^-\chi_{i,0}^+\chi_{i,0}^- + \chi_{j,0}^+\chi_{i,0}^+[K_i,0]\chi_{j,r}^-\chi_{i,0}^- \right )
 &\\&\hspace{0,5cm}
 \equiv \dfrac{[K_i,0]}{[2]_{i}} \left ( \chi_{j,0}^+\chi_{i,0}^+\chi_{i,0}^-\chi_{j,r}^- + \chi_{j,0}^+\chi_{i,0}^+\chi_{j,r}^-\chi_{i,0}^- \right )&\\&\hspace{0,5cm}
 \equiv \dfrac{[K_i,0]}{[2]_{i}} \left ( \chi_{j,0}^+[K_i;0]\chi_{j,r}^- + \frac{[K_i;0]}{(q_j-q_j^{-1})}\phi_{j,r}^+\right )&\\&\hspace{0,5cm}\equiv \dfrac{[K_i,0]}{[2]_{i}(q_j-q_j^{-1})} \left ( [K_i;2]\phi_{j,r}^+ + [K_i;0]\phi_{j,r}^+\right )
 \equiv \dfrac{1}{(q_j-q_j^{-1})} [K_i;0][K_i;1]\phi_{j,r}^+ &\\\\&
 \bullet \ \left[\chi_{j,0}^+(\chi_{i,0}^+)^{(2)},(\chi_{i,0}^-)^{(2)}\chi_{j,r}^-\right]\equiv \dfrac{1}{[2]_{i}^2} \chi_{j,0}^+(\chi_{i,0}^+)^{2}(\chi_{i,0}^-)^{2}\chi_{j,r}^-\equiv \dfrac{1}{[2]_i} \chi_{j,0}^+ [K_i;0][K_i;-1]\chi_{j,r}^-&\\&\hspace{0,5cm}\equiv 
 \dfrac{1}{[2]_{i}(q_j-q_j^{-1})}[K_i;2][K_i;1] \phi_{j,r}^+
\end{align*}
Now, summing all terms with the prescribed coefficients yields the following expression
\begin{align*}\mathbf{T}_i(\phi_{j,r}^+)&\equiv \dfrac{1}{[2]_{i}} \left((1-[2]^2_{i})[K_i;-1]+[2]_{i} [K_i;0]\right)\chi_{i,0}^+\phi_{j,r}^+\chi_{i,0}^-+  \dfrac{1}{[2]_{i}^2} \chi_{i,0}^+\left[\chi_{i,0}^+,\phi_{j,r}^+\right](\chi_{i,0}^-)^2&\\&
+\dfrac{1}{[2]_{i}}\left(([2]^2_{i}-2)[K_i;0][K_i;-1]-[2]_{i}[K_i;0][K_i;2]
+[K_i;2][K_i;1] \right)\phi_{j,r}^+&\\&=\dfrac{1}{[2]_{i}^2} \chi_{i,0}^+\left[\chi_{i,0}^+,\phi_{j,r}^+\right](\chi_{i,0}^-)^2-\dfrac{1}{[2]_{i}}[K_i;-3] \chi_{i,0}^+\phi_{j,r}^+\chi_{i,0}^-
+\dfrac{1}{[2]_{i}}[K_i;-2][K_i;-1]\phi_{j,r}^+
%&\\&
%\equiv \dfrac{q_i^{4}(q_j-q_j^{-1})}{[2]_{i}^2} \phi_{j,0}^+\chi_{i,0}^+\left[\chi_{i,0}^+,\bar{\Psi}_{j,r}'\right](\chi_{i,0}^{-})^2+\dfrac{1}{[2]_{i}}\left((q_i^2-1)[K_i;-1]-[K_i;-3]\right) \chi_{i,0}^+\phi_{j,r}^+\chi_{i,0}^-&\\&
%+\dfrac{1}{[2]_{i}}[K_i;-2][K_i;-1]\phi_{j,r}^+
%\equiv -\frac{1}{[2]_i}[K_i;-3] \left[\chi_{i,0}^+,\phi_{j,r}^+\right]\chi_{i,0}^{-}+\frac{1}{[2]^2_i}\chi_{i,0}^+\left[\chi_{i,0}^+,\phi_{j,r}^+\right](\chi_{i,0}^{-})^2+\phi_{j,r}^+&\\&
\end{align*}
From here, we proceed as follows. Using $\chi_{i,0}^+\phi_{j,r}^+\chi_{i,0}^-\equiv [\chi_{i,0}^+,\phi_{j,r}^+]\chi_{i,0}^-+[K_i;0]\phi_{j,r}^+$ and the congruence (recall $\phi_{j,r}^+=(q_j-q_j^{-1})K_j\bar{\Psi}_{j,r}'$)
$$\chi_{i,0}^+[\chi_{i,0}^+,\phi_{j,r}^+](\chi_{i,0}^{-})^2\equiv q_i^{4}(q_j-q_j^{-1})K_j\chi_{i,0}^+[\chi_{i,0}^+,\bar{\Psi}_{j,r}'](\chi_{i,0}^{-})^2+(q_i^{2}+1)[K_i;-1]\chi_{i,0}^+\phi_{j,r}^+\chi_{i,0}^-$$
we can rewrite $\mathbf{T}_i(\phi_{j,r}^+)$ in terms of the following three expressions:
$$\bullet\ \chi_{i,0}^+[\chi_{i,0}^+,\bar{\Psi}_{j,r}'](\chi_{i,0}^{-})^2\ \ \ \ \bullet\  [\chi_{i,0}^+,\phi_{j,r}^+]\chi_{i,0}^-\ \  \ \ \bullet \ \phi_{j,r}^+$$
Now substituting the formulas from Proposition~\ref{propidentw} and Corollary~\ref{identcor523} we get a huge expression which will finally simplify to 
$$\mathbf{T}_i(\phi_{j,r}^+)\equiv \sum_{s=0}^{r}\sum_{\ell=0}^sq_i^{2(r-s)}\phi_{j,\ell}^+\phi_{i,s-\ell}^+\phi_{i,r-s}^+.$$
We omit the details here and refer for the final steps to \cite{KBthesis}.
\subsection{} Define as usual \begin{equation}\label{lam2waw}\Lambda_i(u)=\sum_{r=0}^{\infty}\Lambda_{i,r}u^r=\mathrm{exp}\left(-\sum_{s=1}^{\infty}\frac{h_{i,s}}{[s]_{i}}u^s\right),\ \ P_i(u)=\Lambda_i(q_iu),\ \ i\in\mathring{I}.\end{equation}
Then we have 
 $$\phi^+_{i}(u)=K_i \frac{\Lambda_i(uq_i^{-1})}{\Lambda_i(uq_i)},\ \ \Lambda_{i,r}=\frac{-q_i^{r}
    }{q_i^r-q_i^{-r}}\sum_{s=1}^{r}q_i^{-s}\phi^+_{i,s}\Lambda_{i,r-s}K_i^{-1},\ \ r>0.$$

\begin{cor}\label{identfuerphil0}
We have the following identities for $i,j\in\mathring{I}$
modulo $\mathcal{I}_{i,\infty}^{0}$ and $\mathcal{I}_{j,\infty}^{0}$
$$\mathbf{T}_i(\Lambda_j(u))= 
      \begin{cases}\Lambda_j(u),& \text{ if $a_{ji}=0$},\\
      \Lambda_j(u)\Lambda_i(q_iu),& \text{ if $a_{ji}=-1$},\\
      \Lambda_i(q_i^2u)^{-1},& \text{ if $a_{ji}=2$},\\
\Lambda_j(u)\Lambda_i(q_ju)\Lambda_i(q_iq_ju) ,& \text{ if $a_{ji}=-2$},\\
         \end{cases}$$
         \begin{proof}
The proof in all cases is similar and we demonstrate the $a_{ji}=-1$ case only. Note that by definition we have  
$$(q_j-q_j^{-1})\sum_{k=1}^{\infty}h_{j,k}u^k=\mathrm{log}(K_j^{-1}\phi_j^+(u))$$
Applying $\mathbf{T}_i$ and using Proposition~\ref{identfuerphil} we get 
\begin{align*}\mathbf{T}_i(h_{j,k})&\equiv \begin{cases}
    h_{j,k}+\frac{(q_i-q_i^{-1})}{(q_j-q_j^{-1})}q_i^{k}h_{i,k},& \text{ if $a_{ij}=-1$}\\
    h_{j,k}+\frac{(q_i-q_i^{-1})}{(q_j-q_j^{-1})}(1+q_i^{2k})h_{i,k},& \text{ if $a_{ij}=-2$}\\
    \end{cases}&\\& =h_{j,k}+q_i^{k}\frac{[k]_j}{[k]_i}h_{i,k}\end{align*}
where the last step follows from $q_i^{2}=q_j$ if $a_{ij}=-2$ and $q_i=q_j$ otherwise. 
Now applying this to \eqref{lam2waw} gives $\mathbf{T}_i(\Lambda_j(u))=\Lambda_j(u)\Lambda_i(q_iu)$ as desired.
\end{proof}
\end{cor}
\section{Parametrization of finite-dimensional irreducible representations}\label{section8}
In this section, we study the representation theory of the subalgebra of the parabolic quantum affine algebra generated by the elements listed in Proposition~\ref{genalgim} without $D^{\pm}$, over the field $\mathbb{K}:=\mathbb{C}(q)$.
%Klappt für beliebigen Körper $\mathbb{K}$ mit: characteristic zero with $q$ a transcendental element in $\mathbb{K}^{\times}$.
For simplicity we will denote this subalgebra again by $\mathbf{U}_q^{J}$ and keep the same notation for the subalgebras introduced in the previous sections; for example $\mathbf{U}_q^J(0)$ is now the subalgebra generated by $h_{i,s},K_i^{\pm}, C^{\pm 1/2} , i\in\mathring{I},s\in\mathbb{N}$ (without $D^{\pm}$). Recall that $$\mathcal{B}_q^j:=\langle\chi_{j,0}^{\pm},K_j^{\pm}\rangle,\ j\in \mathring{I}\cap J, \ \text{ and }\  \mathcal{B}_q^0:=\langle\chi_{\gamma_0,\mp 1}^{\pm},C^{\mp}K_{\gamma_0}^{\pm}\rangle\ \text{ if } 0\in J$$ generate subalgebras isomorphic to $\mathbf{U}_{q_j}(\mathfrak{sl}_2)$ and $\mathbf{U}_{q_0}(\mathfrak{sl}_2)$ respectively. 
\subsection{} We start with a few definitions in analogy to the representation theory for quantum affine algebras. Let $\mathcal{K}$ the subalgebra generated by $\{K_i^{\pm}\}_{i\in\mathring{I}}$ and set $\mathcal{K}^{*}=(\mathbb{K}^{\times })^{\mathring{I}}$ which is a group under pointwise multiplication. We define the standard partial ordering on $\mathcal{K}^{*}$:
$$\boldsymbol{w}\leq \boldsymbol{w}' \iff \boldsymbol{w}(\boldsymbol{w}')^{-1} \text{ is a product of $\{\bar{\alpha}^{-1}_i\}_{i\in \mathring{I}}$}$$
where $\bar{\alpha}_i\in \mathcal{K}^{*}$ is given by $\bar{\alpha}_i(j)=q_i^{a_{i,j}}$.
%Note that the algebra generated by $h_{i,s},K_i^{\pm}, i\in\mathring{I},s\in\mathbb{N}$ coincides with the algebra generated by $\phi_{i,s}^+, K_i^{\pm}, i\in\mathring{I},s\in\mathbb{N}.$
\begin{defn}
Let $V$ be a representation for $\mathbf{U}_q^{J}$.
\begin{enumerate}
\item A series of elements $\Psi=(\psi_{i,r})_{i\in\mathring{I},r\geq 0}$ in the field $\mathbb{K}$ with $\psi_{i,0}\neq 0$ for all $i\in \mathring{I}$ is called an $\ell$\textit{-weight}. We will often write $\Psi=(\Psi_i(u))_{i\in \mathring{I}}$ where $\Psi_i(u)=\sum_{r\geq 0}\psi_{i,r} u^r$, and we denote the set of $\ell$-weights by $\mathfrak{L}$. The subspace
$$V_{\Psi}=\left\{v\in V: \exists N\in\mathbb{N} \ \text{ such that }(\phi_{i,r}^+-\psi_{i,r})^Nv=0,\  \forall i\in\mathring{I},r\in\mathbb{Z}_+\right\}$$ is called the $\ell$\textit{-weight space} of $\ell$-weight $\Psi$ and its elements are called $\ell$\textit{-weight vectors} of $\ell$-weight $\Psi$. \vspace{0,2cm}
\item $V$ is called an $\ell$\textit{-weight module} if every vector of $V$ is a linear combination of $\ell$-weight vectors. \vspace{0,2cm}
\item  A vector $v\in V$ is said to be a \textit{highest} $\ell$\textit{-weight vector} of $\ell$-weight $\Psi$ if  
$$\chi_{i,r}^+v=0, \ \ \phi^+_{i,r}v=\psi_{i,r}v,\ i\in \mathring{I},\ r\in \mathbb{Z}_+,\ \ \chi_{\gamma_0,r-1}^+v=0, \ r \in \mathbb{Z}_+ \  (\text{if $0\in J$}).$$
If additionally $V=\mathbf{U}_q^Jv$, then $V$ is said to be a \textit{module of highest} $\ell$-\textit{weight} $\Psi$. In this case, the $\ell$-weight $\Psi$ is uniquely determined by $V$ and $V$ is $\mathcal{K}$-diagonalizable
$$V=\bigoplus_{\boldsymbol{w}\in \mathcal{K}^{*}} V_{\boldsymbol{w}},\ \ V_{\boldsymbol{w}}=\{v\in V: K_iv=\boldsymbol{w}(i)v\ \forall i\in\mathring{I}\}.$$
Moreover, $V_{\boldsymbol{w}}=0$ unless $\boldsymbol{w}\leq \boldsymbol{w}(\Psi)$ where $\boldsymbol{w}(\Psi)(i)=\Psi_{i,0}$ and $\mathrm{dim} V_{\boldsymbol{w}(\Psi)}=1$.

 \item $V$ is called a \textit{type 1 module} if $C^{1/2}$ acts as the identity element and $V$ is of type 1 viewed as a module for the quantum group of the semi-simple part of the reductive Lie algebra $\mathring{\lie g}_0$, that is 
$$V=\bigoplus_{\mu\in \mathring{P}_0} V_{\mu},\ \ V_{\mu}=\left\{v\in V: K_{\gamma_j}v=q_j^{\mu(\gamma^{\vee}_{j})}v,\  \forall j\in J\right\}.$$

\item We define $M(\Psi)$ to be the quotient of $\mathbf{U}_q^{J}$ by the left ideal generated by 
$$\{\chi_{i,r}^+,\ \ \chi_{\gamma_0,r-1}^+ \  (\text{if $0\in J$}),\ \ \phi_{i,r}^+-\psi_{i,r},\ \ C^{\pm 1/2}-1,\  \ i\in \mathring{I}, \ r\in \mathbb{Z}_+\}.$$
\end{enumerate}
\end{defn}
We denote by $\mathcal{C}_q$ the category of type 1 finite-dimensional $\ell$-weight modules.
%Note that a highest $\ell$-weight module is not assumed to be a $\ell$-weight module. 
\begin{rem}\label{jj25f}
Let $V$ be a finite-dimensional $\mathbf{U}_q^J(0)$-module such that $C^{\pm 1/2}$ acts as identity and assume that $V$ is a direct sum of $\ell$-weight spaces. For any $\mathbf{U}_q^J(0)$-invariant subspace $W
$, we have
$$W=\bigoplus_{\mathbf{\Psi}\in\mathfrak{L}} \ (V_{\mathbf{\Psi}}\cap W).$$
Moreover, any $V_{\mathbf{\Psi}}\neq 0$ contains a simultaneous eigenvector for the action of $\mathbf{U}_q^J(0)$ (the eigenvalue of $\phi_{i,r}^+$ on this eigenvector is $\psi_{i,r}$). This can be seen by a simple application of Zorn's lemma.
\end{rem}
We summarize a few properties which are adopted from the representation theory of quantum affine algebras; the proof follows the ideas of \cite{CP95}. 
\begin{prop} Let $\Psi\in\mathfrak{L}$ such that
\begin{equation}\label{typ1gar}\prod_{i\in \mathring{I}}\psi_{i,0}^{\mathbf{a}_i(\gamma_j)}\in q_j^{\mathbb{Z}_+},\ \ \forall j\in J.\end{equation} We have 
\begin{enumerate}
    \item $M(\Psi)$ is a type 1 module and a module of highest $\ell$-weight $\Psi$.\vspace{0,2cm}
    
    \item $M(\Psi)$ has a unique irreducible quotient $V(\Psi)$. \vspace{0,2cm}
    
    \item Let $V$ be a simple object in $\mathcal{C}_q$. Then $V\cong V(\Psi)$ for some $\Psi\in\mathfrak{L}$ satisfying \eqref{typ1gar}. 
\end{enumerate}
    \begin{proof}
        The first part follows directly from the definition, equation~\eqref{typ1gar}, and Corollary~\ref{trisecond}.
        The second part follows by a standard argument using Corollary~\ref{trisecond} that $M(\Psi)$ has a unique maximal proper submodule. For the third part, let $v_0\in V$ be a simultaneous eigenvector for the $\mathcal{K}$-action which exists by Remark~\ref{jj25f} and the fact that $V\neq 0$ is a $\ell$-weight module. Set  $$V^0=\{v\in V: \chi_{i,r}^+v=0,\ \chi_{\gamma_0,r-1}^+v=0 \ (\text{if $0\in J$}),\ \forall i\in\mathring{I},\ r\in\mathbb{Z}_+\}$$ and note that each sequence of non-zero elements
        $$v_0,\ y_1v_0,\ y_2y_1v_0,\dots \ \ \ y_i\in\{\chi_{i,r}^+,\chi_{\gamma_0,r-1}^+\ (\text{if $0\in J$}): i\in\mathring{I},\ r\in\mathbb{Z}_+\}$$
        would give a sequence of linearly independent elements in $V$, since the $\mathcal{K}$-eigenvalues are all different. Thus $V^0\neq 0$, since $V$ is finite-dimensional. 
        From the defining relations stated in Theorem~\ref{definingrelationQAA} and \eqref{stabil23}, it is clear that $V^0$ is $\mathbf{U}_q^J(0)$-invariant provided that
\begin{equation}\label{stabil2}\chi_{\gamma_0,-1}^+h_{i,s}v=0,\ \forall \ i\in\mathring{I},\ s\in\mathbb{N},\ v\in V^0.\end{equation}
    From Proposition~\ref{algebrengleich3} we have $\chi_{\gamma_0,-1}^+\in \mathbf{U}_q^J\cap \langle \chi_{i,r}^+: i\in\mathring{I}, r\in\mathbb{Z}\rangle$ and combining this with Theorem~\ref{definingrelationQAA} we get 
        $$[C^{s/2}h_{i,s},\chi_{\gamma_0,-1}^+]\in \mathbf{U}_q^J\cap \langle \chi_{i,r}^+: i\in\mathring{I}, r\in\mathbb{Z}\rangle=\langle \chi_{i,r}^+, \chi_{\gamma_0,-1}^+ (\text{if $0\in J$)}: i\in\mathring{I},\ r\geq 0\rangle.$$
Thus \eqref{stabil2} holds and $V^0$ is a $\mathbf{U}_q^J(0)$-invariant subspace of $V$. Since $V^0\neq 0$, there exists by Remark~\ref{jj25f} an $\ell$-weight $\Psi$ with $(V^0)_{\Psi}=V_{\Psi}\cap V^0\neq 0$. Again by  Remark~\ref{jj25f} we can choose a simultaneous eigenvector in $(V^0)_{\Psi}$ for the $\mathbf{U}_q^J(0)$-action. Since $V$ is simple, this vector generates $V$ and thus $V\cong V(\Psi)$. The fact that $\Psi$ satisfies \eqref{typ1gar} follows from $\mathfrak{sl}_2$-theory using the subalgebras $\mathcal{B}_q^j$ for $j\in J$.
\end{proof}
\end{prop}
\subsection{} A positive root $\alpha\in \mathring{R}^+$ is said to be  \textit{repetition-free} if there exists an element $w=\mathbf{s}_{i_1}\cdots\mathbf{s}_{i_k}\in \mathring{W}$ and $i\in\mathring{I}$ such that the indices $i_1,\dots,i_k,i$ are all distinct and $w(\alpha_i)=\alpha$. For a repetition-free root $\alpha$, it is not hard to see that the element $w$ appearing in the definition can be chosen such that $w^{-1}\in \mathring{W}(\alpha)$. 
\begin{lem} We have
    $$\alpha \in \mathring{R}^+ \text{is repetition-free} \iff \begin{cases}
        \mathbf{a}_i(\alpha) \leq 1 \ \ \forall i\in \mathring{I}, & \text{if $\alpha$ is short} \\
        \mathbf{a}_i^{\vee}(\alpha) \leq 1 \ \ \forall i \in \mathring{I}, & \text{if $\alpha$ is long.}
    \end{cases}$$
    \begin{proof}
        Let $\alpha$ be a repetition-free root. By the comment preceding the lemma, there exists an element $w^{-1}\in \mathring{W}(\alpha)$ with $w(\alpha_i)=\alpha$, and the index $i$
together with all indices occurring in $w$ are pairwise distinct. Then
        $$|\text{supp}(\alpha)|= \ell(w) + 1 = \begin{cases}
        \text{ht}(\alpha), & \text{if $\alpha$ is short} \\
        \text{ht}^{\vee}(\alpha) , & \text{if $\alpha$ is long}
    \end{cases}$$
  and one direction of the proof follows. For the converse direction let $w^{-1}\in \mathring{W}(\alpha)$ and assume that $\alpha$ is not repetition-free. Thus, there must be some repetition among the involved indices, and we obtain
  $|\text{supp}(\alpha)|\leq \ell(w)$, which completes the proof.
    \end{proof}
\end{lem}

The first two parts of the following proposition can be derived from \cite[Proposition 3.5]{CPsl2} or \cite[Section 1.3]{JM14a}. 
\begin{prop}\label{idenga6} The following identities hold:
    
\begin{enumerate}
    \item  For $i\in \mathring{I}$, we have 
    $$\phi^+_{i}(u)=K_i \frac{\Lambda_i(uq_i^{-1})}{\Lambda_i(uq_i)},\ \ \Lambda_{i,r}=\frac{-q_i^{r}}{q_i^r-q_i^{-r}}\sum_{s=1}^{r}q_i^{-s}\phi^+_{i,s}\Lambda_{i,r-s}K_i^{-1},\ r\in\mathbb{N}.$$ \vspace{0,2cm}
     \item For $i\in \mathring{I}$, we have modulo the left ideal generated by $\mathcal{I}_{i,\infty}^{0}$ $$\Lambda_{i,r}\equiv(-1)^rq_i^{r(r-1)}\frac{(\chi_{i,0}^+)^{r}(\chi_{i,1}^-)^{r}}{([r]_{i}!)^2}K_i^{-r},\ \ r\in\mathbb{N}.$$ \vspace{0,2cm}
     \item Let $\mathring{\lie g}$ not of type $G_2$ and $\alpha\in\mathring{R}^+$ be repetition-free, and write $\alpha=w_{}(\alpha_i)$ as in the definition. Modulo the left ideal generated by $\chi_{\ell,p}^+$ with $p\geq 0$ and $\ell\in\mathrm{supp}(\alpha)$ we have 
     $$\mathbf{T}_w(\Lambda_i(u))\equiv\prod_{j\in\mathring{I}} \Lambda_j\left(q^{n_{j,1}}u\right)\cdots \Lambda_j\left(q^{n_{j,\mathbf{a}_j^{\vee}(\alpha)}}u\right)$$
     for some non-negative integers $n_{j,s}$, $j\in\mathring{I},\ 1\leq s\leq \mathbf{a}_j^{\vee}(\alpha)$.
\end{enumerate}
\begin{proof}
We show the last item by induction on $\ell(w)$, where the base case is clear. Write $w=\mathbf{s}_kw'$ with $\ell(w)=\ell(w')+1$, and note that $\beta=w'(\alpha_i)\in \mathring{R}^+$ is again a repetition-free, satisfying $k\notin\mathrm{supp}(\beta)\subseteq \mathrm{supp}(\alpha)$. By the induction hypothesis, we can write $\mathbf{T}_{w'}(\Lambda_i(u))$ in the desired form modulo the prescribed ideal. Now each summand of $\mathbf{T}_{k}(\chi_{\ell,p}^+)$, for $p\in\mathbb{Z}_+$ and $\ell\in \mathrm{supp}(\beta)$, is contained in $\mathcal{I}_{\ell', \infty}^{0}$ for some $\ell'\in\mathrm{supp}(\alpha)$, since $k\notin I(w')$. Finally, by acting further with $\mathbf{T}_k$ on $\mathbf{T}_{w'}(\Lambda_i(u))$, the claim follows from Corollary~\ref{identfuerphil0}.
\end{proof}
\end{prop}
\begin{cor}\label{coxident1}
    Let $\mathring{\lie g}$ not of type $G_2$ and $\gamma_0\in\mathring{R}^+$ be repetition-free and write $w(\alpha_i)=\gamma_0$ for some $w^{-1}\in \mathring{W}(\gamma_0)$. Then, modulo the left ideal generated by $\chi_{\ell,p}^+$ with $p\geq 0$ and $\ell\in\mathrm{supp}(\gamma_0)$, we have up to a nonzero scalar multiple 
    $$(\mathbf{T}_w(\chi_{i,0}^+))^r(\chi_{\gamma_0,1}^-)^{r}K_{\gamma_0}^{-r}
    \equiv \left(\prod_{j\in\mathring{I}} \Lambda_j(q^{n_{j,1}}u)\cdots \Lambda_j\left(q^{n_{j,\mathbf{a}_j^{\vee}(\gamma_0)}}u\right)\right)_r$$
    for some non-negative integers $n_{j,s}$, $j\in\mathring{I},\ 1\leq s\leq \mathbf{a}_j^{\vee}(\gamma_0)$ and the subscript on the right-hand side indicates that we are taking the $r$-th coefficient of the series.
    \begin{proof}  We first note that each summand of $\mathbf{T}_{w}(\chi_{i,r}^+)$, $r\in\mathbb{Z}_+$, is contained in $\mathcal{I}_{\ell,\infty}^{0}$ for some $\ell\in\mathrm{supp}(\gamma_0)$, since $i\notin I(w)$. Hence with Proposition~\ref{idenga6}(2) we get (up to a non zero scalar multiple)
\begin{align*}(\mathbf{T}_{w_{}}(\chi_{i,0}^+))^r(\chi_{\gamma_0,1}^-)^{r}K_{\gamma_0}^{-r}&= \mathbf{T}_{w}\left((\chi_{i,0}^+)^{r}(\chi_{i,1}^-)^{r}K_i^{-r}\right)\equiv \mathbf{T}_{w}(\Lambda_{i,r}).\end{align*} 
Now the claim follows with Proposition~\ref{idenga6}(3).
    \end{proof}
\end{cor}
\begin{rem}
For the classification of simple objects in $\mathcal{C}_q$, the identity stated in Corollary~\ref{coxident1} will be essential. We conjecture that the assumptions that $\gamma_0$ is repetition-free and that $\mathring{\lie g}$ is not of type $G_2$ are actually not needed. If Corollary~\ref{coxident1} holds without these assumptions, then all subsequent results continue to hold in full generality.
\end{rem}
\textit{From now on, we assume that $\gamma_0$ is repetition-free whenever $0\in J$ and $\mathring{\lie g}$ is not of type $G_2$.}
\subsection{}\label{subsectionmap}
We summarize a few immediate consequences from the discussion so far. Every simple object 
in $\mathcal{C}_q$ is of the form $V=V(\Psi)$; let $v$ be its highest $\ell$-weight vector. Then the discussion so far implies the following: 
\begin{enumerate}[(i)]
    \item  There exists $\mu_V\in \mathring{P}_0^+$ with 
$$\prod_{i\in\mathring{I}}\psi_{i,0}^{\mathbf{a}_i(\gamma_j)}= q_j^{\mu_V(\gamma_j^{\vee})},\ \ \forall j\in J.$$ 
%This follows by considering $V(\Psi)$ as a representation for $\mathcal{B}_q^j,\ j\in J$.
\item Given $i\in\mathring{I}$ there exists $r_i\in\mathbb{Z}_+$ with $(\chi_{i,1}^-)^{r_i+1}v=0$ as $V(\Psi)$ is finite-dimensional. 
Then Proposition~\ref{idenga6}(2) gives the existence of a polynomial 
%das Polynom $\boldsymbol{\pi}_i(u)$ kodiert die Operation von $\Lambda_i(u)$
$\boldsymbol{\pi}_i(u)\in\mathbb{K}[u]$ with constant term 1 such that
\begin{equation}\label{ratio}\Psi_i(u)=\psi_{i,0}\frac{\boldsymbol{\pi}_i(uq_i^{-1})}{\boldsymbol{\pi}_i(uq_i)}\end{equation}
Moreover, if $j\in \mathring{I}\cap J$ we have $$K_{j}v=K_{\gamma_{\bar{j}}}v=q_j^{\mu_V(\gamma_{\bar{j}}^{\vee})}v,\ \ \chi_{j,0}^+v=0\implies (\chi_{j,0}^-)^{\mu_V(\gamma_{\bar{j}}^{\vee})+1}v=0.$$ 
%where the implication follows by considering $V(\Psi)$ as a representation for $\mathcal{B}_q^{j}$. 
Thus, by applying $\mathbf{T}_{\varpi_i}^{-1}$ to the identity in Proposition~\ref{idenga6}(2), gives
$$\mathrm{deg}(\boldsymbol{\pi}_j(u))\leq \mu_V(\gamma_{\bar{j}}^{\vee}),\ \forall j\in \mathring{I}\cap J.$$
Similarly, 
$$K_{\gamma_{0}}v=q_0^{\mu_V(\gamma_{0}^{\vee})}v,\ \ \chi_{\gamma_0,-1}^+v=0\implies (\chi_{\gamma_0,1}^-)^{\mu_V(\gamma_{0}^{\vee})+1}v=0$$
and Corollary~\ref{coxident1} implies 
$$\mathrm{deg}(\boldsymbol{\pi}_0(u))\leq \mu_V(\gamma_{0}^{\vee}),\ \text{ where }\ \boldsymbol{\pi}_0(u):=\prod_{i\in\mathring{I}}\boldsymbol{\pi}_i(u)^{\mathbf{a}_i^{\vee}(\gamma_0)}.$$
\item Since $V(\Psi)$ is an $\ell$-weight module, the polynomials $\boldsymbol{\pi}_i(u)$ have to split over the field $\mathbb{K}$ and vice versa. The reasoning follows the same structure as in \cite[Theorem 4.19 and Remark 4.3]{JA11a}, so we omit the details.
\end{enumerate}
\begin{defn} As in the classical case, we define $\mathcal{P}^+_{q}$  to be the set of tuples
$$\boldsymbol{\pi}_{\mu,\mathbf{b}}=(\boldsymbol{\pi}_{1}(u),\dots,\boldsymbol{\pi}_{n}(u),\mu,\mathbf{b}),\ \ \boldsymbol{\pi}_{i}(u)\in\mathbb{K}[u],\ \ \mu\in \mathring{P}_0^+,\ \  \mathbf{b}\in \mathbb{K}^{|\mathring{I}\backslash \mathring{I}\cap J|}$$
satisfying 
$$\boldsymbol{\pi}_i(u) \text{ splits over $\mathbb{K}$},\ \ \boldsymbol{\pi}_i(0)=1,\ \forall i\in\mathring{I},\ \ \mathrm{deg}(\boldsymbol{\pi}_{\bar{j}}(u))\leq \mu(\gamma^{\vee}_j) \ \forall j\in J$$ 

$$\prod_{i\in \mathring{I}\backslash \mathring{I}\cap J}b^{\mathbf{a}_i(\gamma_0)}_i=q^{\epsilon_0\mu(\gamma_0^{\vee})-\sum_{i\in \mathring{I}\cap J}\epsilon_i\mathbf{a}_i(\gamma_0)\mu(\gamma_{\bar{i}}^{\vee})}, \ \ \text{if $0\in J$}$$
\end{defn}
The main result of this section is the following; recall that $\gamma_0$ is assumed to be repetition-free and $\mathring{\lie g}$ is not of type $G_2$.
\begin{thm}\label{mainthmrep}
    We have a one-to-one correspondence 
    $$\{\text{simple objects in $\mathcal{C}_q$}\}/\sim \ \ \longrightarrow \mathcal{P}^+_{q}$$
$$V=V(\Psi)\rightarrow \boldsymbol{\pi}_{\mu_V,\mathbf{b}}$$
where $\mu_V$ is described in Subsection~\ref{subsectionmap}(i), the polynomials $\boldsymbol{\pi}_i(u)$ in Subsection~\ref{subsectionmap}(ii) and $\mathbf{b}$ is determined by $b_i=\psi_{i,0}$ for $i\in \mathring{I}\backslash \mathring{I}\cap J$.
\begin{proof}
The discussion so far defines a well-defined injective map. Let $\boldsymbol{\pi}_{\mu,\mathbf{b}} \in \mathcal{P}^+_{q}$, and denote by $\Psi$ the $\ell$-weight determined by the above data, i.e.,
$$\Psi_{i,0}=b_i,\ i\in \mathring{I}\backslash \mathring{I}\cap J,\ \Psi_{i,0}=q_i^{\mu(\gamma_{\bar{i}}^{\vee})}, \ i\in \mathring{I}\cap J$$
and $\Psi_{i,r}$ for $r>0$ is determined by the ratio's of the polynomials. Surjectivity will be established if we can show that the module $V(\boldsymbol{\pi}_{\mu,\mathbf{b}}) := V(\Psi)$ is finite-dimensional. This, in turn, will follow from the next two statements.
\begin{enumerate}[(a)]
\item If all polynomials are constant, i.e. $\boldsymbol{\pi}_i(u)=1$ for all $i\in\mathring{I}$, then  $V(\boldsymbol{\pi}_{\mu,\mathbf{b}})$ is finite-dimensional.
\item Let $V(\boldsymbol{\pi})$ be the irreducible representation for the quantum group $\mathbf{U}_q$ associated to $\boldsymbol{\pi})$ and consider it as a representation for $\mathbf{U}^J_q$ by restriction. Then there is a map
    $$M(\Psi)\rightarrow V(\boldsymbol{\pi})\otimes V(\boldsymbol{\pi'}_{\mu',\mathbf{b'}}),\ 1\mapsto v_1\otimes v_2$$
    where $v_1\otimes v_2$ is the tensor product of highest $\ell$-weights and  $$b_i'=q_i^{-\mathrm{deg}(\boldsymbol{\pi}_i)}\Psi_{i,0},\ i\in \mathring{I}\backslash \mathring{I}\cap J,$$$$
    \mu'(\gamma_j^{\vee})=\mu(\gamma_j^{\vee})-\mathrm{deg}(\boldsymbol{\pi}_{\bar{j}}(u)), \ j\in J,\ \   \boldsymbol{\pi}'_i(u)=1\  \forall i\in\mathring{I}.$$
    \end{enumerate}
    To see this, we proceed as follows. It is straightforward to verify that $\boldsymbol{\pi'}_{\mu',\mathbf{b'}} \in \mathcal{P}_q^+$. Therefore, by part (a) and the general theory of ordinary quantum affine algebras, the tensor product $V(\boldsymbol{\pi})\otimes V(\boldsymbol{\pi'}_{\mu',\mathbf{b'}})$ is finite-dimensional. In particular, by part (b), the module $M(\Psi)$ admits a non-zero finite-dimensional quotient. This implies that the corresponding proper submodule must be contained in the unique maximal submodule. 
    
    Hence, it suffices to establish parts (a) and (b) where part (b) is immediate from Theorem~\ref{Dami}(1),(3) by recalling that $\chi_{i,r}^+,\chi_{\gamma_0,r-1}^+\in\mathbf{U}_q^j(+)=\mathbf{U}_q^J\cap \langle \chi_{i,r}^+: i\in\mathring{I},\ s\in\mathbb{Z}\rangle$.
    
  In the remainder of the proof, we will establish part (a). So assume that all polynomials are constant. This implies that $h_{i,r}v=0$ for all $i\in\mathring{I}$ and $r>0$. We now consider the submodule generated by all elements of the form
\begin{equation}\label{subgenir0}\chi_{\beta_1,r_1}^{-}\cdots \chi_{\beta_s,r_s}^-v,\ s\geq 1, \ (r_i,\beta_i)\in \{(r,\beta)\in \mathbb{Z}_+\times \mathring{R}^+: r>0 \text{ or }\mathrm{supp}(\beta)\subseteq J\}\end{equation}
such that there exists $i\in\{1,\dots,s\}$ with $\mathrm{supp}(-\beta_i+r_i\delta)\nsubseteq J$. We claim that this submodule does not contain the highest $\ell$-weight vector $v$. By the triangular decomposition, it suffices to show that for any homogeneous element  $X\in \mathbf{U}_q^J(+)$ with respect to the $Q$-grading, we have
$$Yv\neq v,\ \ Y=X\chi_{\beta_1,r_1}^{-}\cdots \chi_{\beta_s,r_s}^-$$
where $\chi_{\beta_1,r_1}^{-}\cdots \chi_{\beta_s,r_s}^-v$ is as in \eqref{subgenir0}. Recall that $\mathbf{U}_q^J(+)$ is generated by $\chi_{i,r}^+$, $i\in\mathring{I}$, $r\geq 0$ and $\chi_{\gamma_0,-1}^+$ (if $0\in J$). Since $\mathrm{supp}(\gamma_0-\delta)\subseteq J$ and $\mathrm{supp}(-\beta_s+r_s\delta)\nsubseteq J$ the weight of $Y$ lies in $p\alpha_k+Q_{I\backslash\{k\}}$ for some $k\notin J$ and $p>0$. Rewriting $Y$ again in the second triangular decomposition and using $h_{i,r}v=0$ and $\mathbf{U}_q^J(+)v=0$, we conclude that the vector $Yv$ cannot be a non-zero multiple of $v$, for weight reasons.

\textit{Remark: The relation $h_{i,r}v=0$ is crucial in the above argument, otherwise $Yv$ could be proportional to an element of $\mathbf{U}_q^J(0)v$. Also if $\mathrm{supp}(-\beta+r\delta)\subseteq J$, then for example, $\chi_{\beta,-1}^{+}\chi_{\beta,1}^{-}v$ could be a non-zero multiple of $v$.}

Therefore, by the irreducibility of the representation, the aforementioned submodule must be zero, and therefore $V(\boldsymbol{\pi}_{\mu,\mathbf{b}})$ is a cyclic module for $\mathring{\mathbf{U}}_q^J$. Note that if $X^{\pm}\in \mathbf{U}_q^J(\pm)$ are homogeneous such that the weight of $X^+X^-$ is non-zero and its support is not contained in $J$, then we can express $X^+X^-$ in the second triangular decomposition whose projection onto $\mathbf{U}_q^J(-)\mathcal{K}$ (if it is non-zero) must contain generators $\chi_{\beta,r}^-$ with $\mathrm{supp}(-\beta+r\delta)\nsubseteq J$. In particular, by the discussion above $X^+X^-v=0$. 
Now consider the submodule generated by $(\chi_{j,0}^{-})^{\mu(\gamma_{\bar{j}}^{\vee})+1}v$ for some $j\in\mathring{I}\cap J$. We have 
\begin{equation}\label{hhttbbnn}\chi_{\gamma_0,-1}^+(\chi_{j,0}^{-})^{\mu(\gamma_{\bar{j}}^{\vee})+1}v=0,\ \ \chi_{i,r}^{+}(\chi_{j,0}^{-})^{\mu(\gamma_{\bar{j}}^{\vee})+1}v=0,\ \ i\in\mathring{I},\ r\geq 0.\end{equation}
The first part of \eqref{hhttbbnn} follows from the commutativity relation 
$$[\chi_{\gamma_0,-1}^+,\chi_{j,0}^{-}]=-\kappa(\gamma_0)\mathbf{T}_{w_{\circ}}([F_0K_0,\mathbf{T}_{w_{\circ}}(F_j)])=\kappa(\gamma_0)\mathbf{T}_{w_{\circ}}([F_0K_0,K_{\bar{j}}^{-1}E_{\bar{j}}])=0.$$ The second part of \eqref{hhttbbnn} is clear for $i\neq j$ and otherwise using standard $\mathfrak{sl}_2$-theory will give the claim for $r=0$, while the case $r>0$ is a consequence of the above discussion, since $\mathrm{supp}(-\mu(\gamma_{\bar{j}}^{\vee})\alpha_j+r\delta)\nsubseteq J$.

Thus, again by the irreducibility we must have $(\chi_{j,0}^{-})^{\mu(\gamma_{\bar{j}}^{\vee})+1}v=0$. In the final step consider the submodule generated by $(\chi_{\gamma_0,1}^-)^{\mu(\gamma_{0}^{\vee})+1}v$. Again by $\mathfrak{sl}_2$-theory we have 
$\chi_{\gamma_0,-1}^+(\chi_{\gamma_0,1}^-)^{\mu(\gamma_{0}^{\vee})+1}v=0$ and it remains to argue why 
$$\chi_{i,r}^+(\chi_{\gamma_0,1}^-)^{\mu(\gamma_{0}^{\vee})+1}v=0,\ \ i\in\mathring{I},\ r\geq 0$$
holds. Now $$\mathrm{supp}((\mu(\gamma_{0}^{\vee})+1)(-\gamma_0+\delta)+\alpha_i+r\delta)\nsubseteq J$$
unless $r=0$ and $i\in \mathring{I}\cap J$ in which case the claim is again clear, since the elements commute.
Thus $V$ is in fact irreducible for $\mathring{\mathbf{U}}_q^J$ and finite-dimensional.
\end{proof}
\end{thm}

\bibliographystyle{plain}
\bibliography{bibfile}

\begin{thebibliography}{10}

\bibitem{BK24a}
Leon Barth and Deniz Kus.
\newblock Prime representations in the {H}ernandez-{L}eclerc category:
  classical decompositions.
\newblock {\em Canad. J. Math.}, 76(6):1987--2018, 2024.

\bibitem{BK05a}
Pascal Baseilhac and Kozo Koizumi.
\newblock A new (in)finite-dimensional algebra for quantum integrable models.
\newblock {\em Nuclear Phys. B}, 720(3):325--347, 2005.

\bibitem{B94}
Jonathan Beck.
\newblock Braid group action and quantum affine algebras.
\newblock {\em Comm. Math. Phys.}, 165(3):555--568, 1994.

\bibitem{B94A}
Jonathan Beck.
\newblock Convex bases of {PBW} type for quantum affine algebras.
\newblock {\em Comm. Math. Phys.}, 165(1):193--199, 1994.

\bibitem{BK96}
Jonathan Beck and Victor~G. Kac.
\newblock Finite-dimensional representations of quantum affine algebras at
  roots of unity.
\newblock {\em J. Amer. Math. Soc.}, 9(2):391--423, 1996.

\bibitem{KBthesis}
Kudret Bostanci.
\newblock Representations and quantizations of parabolic subalgebras.
\newblock Ruhr-University Bochum, PhD thesis.

\bibitem{Bou1}
Nicolas Bourbaki.
\newblock {\em Lie groups and {L}ie algebras. {C}hapters 4--6}.
\newblock Elements of Mathematics (Berlin). Springer-Verlag, Berlin, 2002.
\newblock Translated from the 1968 French original by Andrew Pressley.

\bibitem{BC19a}
Matheus Brito and Vyjayanthi Chari.
\newblock Tensor products and {$q$}-characters of {HL}-modules and monoidal
  categorifications.
\newblock {\em J. \'Ec. polytech. Math.}, 6:581--619, 2019.

\bibitem{BCKV22}
Matheus Brito, Vyjayanthi Chari, Deniz Kus, and R.~Venkatesh.
\newblock Quantum affine algebras, graded limits, and flags.
\newblock {\em J. Indian Inst. Sci.}, 102(3):1001--1031, 2022.

\bibitem{Cint86}
Vyjayanthi Chari.
\newblock Integrable representations of affine {L}ie-algebras.
\newblock {\em Invent. Math.}, 85(2):317--335, 1986.

\bibitem{C95min}
Vyjayanthi Chari.
\newblock Minimal affinizations of representations of quantum groups: the rank
  {$2$} case.
\newblock {\em Publ. Res. Inst. Math. Sci.}, 31(5):873--911, 1995.

\bibitem{C02B}
Vyjayanthi Chari.
\newblock Braid group actions and tensor products.
\newblock {\em Int. Math. Res. Not.}, (7):357--382, 2002.

\bibitem{CKO18}
Vyjayanthi Chari, Deniz Kus, and Matt Odell.
\newblock Borel--de {S}iebenthal pairs, global {W}eyl modules and
  {S}tanley-{R}eisner rings.
\newblock {\em Math. Z.}, 290(1-2):649--681, 2018.

\bibitem{CPsl2}
Vyjayanthi Chari and Andrew Pressley.
\newblock Quantum affine algebras.
\newblock {\em Comm. Math. Phys.}, 142(2):261--283, 1991.

\bibitem{CP94b}
Vyjayanthi Chari and Andrew Pressley.
\newblock {\em A guide to quantum groups}.
\newblock Cambridge University Press, Cambridge, 1994.

\bibitem{CP95}
Vyjayanthi Chari and Andrew Pressley.
\newblock Quantum affine algebras and their representations.
\newblock In {\em Representations of groups ({B}anff, {AB}, 1994)}, volume~16
  of {\em CMS Conf. Proc.}, pages 59--78. Amer. Math. Soc., Providence, RI,
  1995.

\bibitem{CSVW16}
Vyjayanthi Chari, Peri Shereen, R.~Venkatesh, and Jeffrey Wand.
\newblock A {S}teinberg type decomposition theorem for higher level {D}emazure
  modules.
\newblock {\em J. Algebra}, 455:314--346, 2016.

\bibitem{CLW21a}
Xinhong Chen, Ming Lu, and Weiqiang Wang.
\newblock Serre-{L}usztig relations for {$\imath$}quantum groups.
\newblock {\em Comm. Math. Phys.}, 382(2):1015--1059, 2021.

\bibitem{CLW21b}
Xinhong Chen, Ming Lu, and Weiqiang Wang.
\newblock A {S}erre presentation for the {\it \i}quantum groups.
\newblock {\em Transform. Groups}, 26(3):827--857, 2021.

\bibitem{Da98a}
Ilaria Damiani.
\newblock La {$R$}-matrice pour les alg\`ebres quantiques de type affine non
  tordu.
\newblock {\em Ann. Sci. \'Ecole Norm. Sup. (4)}, 31(4):493--523, 1998.

\bibitem{Da15a}
Ilaria Damiani.
\newblock From the {D}rinfeld realization to the {D}rinfeld-{J}imbo
  presentation of affine quantum algebras: injectivity.
\newblock {\em Publ. Res. Inst. Math. Sci.}, 51(1):131--171, 2015.

\bibitem{KdC90}
Corrado De~Concini and Victor~G. Kac.
\newblock Representations of quantum groups at roots of {$1$}.
\newblock In {\em Operator algebras, unitary representations, enveloping
  algebras, and invariant theory ({P}aris, 1989)}, volume~92 of {\em Progr.
  Math.}, pages 471--506. Birkh\"{a}user Boston, Boston, MA, 1990.

\bibitem{D872}
V.~G. Drinfeld.
\newblock A new realization of {Y}angians and of quantum affine algebras.
\newblock {\em Dokl. Akad. Nauk SSSR}, 296(1):13--17, 1987.

\bibitem{FKKS11}
Ghislain Fourier, Tanusree Khandai, Deniz Kus, and Alistair Savage.
\newblock Local {W}eyl modules for equivariant map algebras with free abelian
  group actions.
\newblock {\em J. Algebra}, 350:386--404, 2012.

\bibitem{FWW24}
Noah Friesen, Alex Weeks, and Curtis Wendlandt.
\newblock Braid group actions, {B}axter polynomials, and affine quantum groups.
\newblock arXiv:2401.06402.

\bibitem{GK81}
Ofer Gabber and Victor~G. Kac.
\newblock On defining relations of certain infinite-dimensional {L}ie algebras.
\newblock {\em Bull. Amer. Math. Soc. (N.S.)}, 5(2):185--189, 1981.

\bibitem{H05ab}
David Hernandez.
\newblock Monomials of {$q$} and {$q, t$}-characters for non simply-laced
  quantum affinizations.
\newblock {\em Math. Z.}, 250(2):443--473, 2005.

\bibitem{HJ12a}
David Hernandez and Michio Jimbo.
\newblock Asymptotic representations and {D}rinfeld rational fractions.
\newblock {\em Compos. Math.}, 148(5):1593--1623, 2012.

\bibitem{HL10a}
David Hernandez and Bernard Leclerc.
\newblock Cluster algebras and quantum affine algebras.
\newblock {\em Duke Math. J.}, 154(2):265--341, 2010.

\bibitem{HL13b}
David Hernandez and Bernard Leclerc.
\newblock Monoidal categorifications of cluster algebras of type {$A$} and
  {$D$}.
\newblock In {\em Symmetries, integrable systems and representations},
  volume~40 of {\em Springer Proc. Math. Stat.}, pages 175--193. Springer,
  Heidelberg, 2013.

\bibitem{HK02}
Jin Hong and Seok-Jin Kang.
\newblock {\em Introduction to quantum groups and crystal bases}, volume~42 of
  {\em Graduate Studies in Mathematics}.
\newblock American Mathematical Society, Providence, RI, 2002.

\bibitem{Hu90}
James~E. Humphreys.
\newblock {\em Reflection groups and {C}oxeter groups}, volume~29 of {\em
  Cambridge Studies in Advanced Mathematics}.
\newblock Cambridge University Press, Cambridge, 1990.

\bibitem{JA11a}
Dijana Jakeli\'c and Adriano~Adrega de~Moura.
\newblock Tensor products, characters, and blocks of finite-dimensional
  representations of quantum affine algebras at roots of unity.
\newblock {\em Int. Math. Res. Not. IMRN}, (18):4147--4199, 2011.

\bibitem{JM14a}
Dijana Jakeli\'{c} and Adriano Moura.
\newblock On {W}eyl modules for quantum and hyper loop algebras.
\newblock In {\em Recent advances in representation theory, quantum groups,
  algebraic geometry, and related topics}, volume 623 of {\em Contemp. Math.},
  pages 99--133. Amer. Math. Soc., Providence, RI, 2014.

\bibitem{J96ab}
Jens~Carsten Jantzen.
\newblock {\em Lectures on quantum groups}, volume~6 of {\em Graduate Studies
  in Mathematics}.
\newblock American Mathematical Society, Providence, RI, 1996.

\bibitem{KW92}
V.~G. Kac and S.~P. Wang.
\newblock On automorphisms of {K}ac-{M}oody algebras and groups.
\newblock {\em Adv. Math.}, 92(2):129--195, 1992.

\bibitem{K90}
Victor~G. Kac.
\newblock {\em Infinite-dimensional {L}ie algebras}.
\newblock Cambridge University Press, Cambridge, third edition, 1990.

\bibitem{Ko14a}
Stefan Kolb.
\newblock Quantum symmetric {K}ac-{M}oody pairs.
\newblock {\em Adv. Math.}, 267:395--469, 2014.

\bibitem{KV20}
Deniz Kus and R.~Venkatesh.
\newblock Borel--de {S}iebenthal theory for affine reflection systems.
\newblock {\em Mosc. Math. J.}, 21(1):99--127, 2021.

\bibitem{L99a}
Gail Letzter.
\newblock Symmetric pairs for quantized enveloping algebras.
\newblock {\em J. Algebra}, 220(2):729--767, 1999.

\bibitem{L02b}
Gail Letzter.
\newblock Coideal subalgebras and quantum symmetric pairs.
\newblock In {\em New directions in {H}opf algebras}, volume~43 of {\em Math.
  Sci. Res. Inst. Publ.}, pages 117--165. Cambridge Univ. Press, Cambridge,
  2002.

\bibitem{LWZ24a}
Ming Lu, Weiqiang Wang, and Weinan Zhang.
\newblock Braid group action and quasi-split affine {$\imath$}quantum groups
  {II}: {H}igher rank.
\newblock {\em Comm. Math. Phys.}, 405(6):Paper No. 142, 33, 2024.

\bibitem{Lu90a}
George Lusztig.
\newblock Finite-dimensional {H}opf algebras arising from quantized universal
  enveloping algebra.
\newblock {\em J. Amer. Math. Soc.}, 3(1):257--296, 1990.

\bibitem{L10a}
George Lusztig.
\newblock {\em Introduction to quantum groups}.
\newblock Modern Birkh\"{a}user Classics. Birkh\"{a}user/Springer, New York,
  2010.
\newblock Reprint of the 1994 edition.

\bibitem{MRS03a}
A.~I. Molev, E.~Ragoucy, and P.~Sorba.
\newblock Coideal subalgebras in quantum affine algebras.
\newblock {\em Rev. Math. Phys.}, 15(8):789--822, 2003.

\bibitem{MY12a}
E.~Mukhin and C.~A.~S. Young.
\newblock Path description of type {B} {$q$}-characters.
\newblock {\em Adv. Math.}, 231(2):1119--1150, 2012.

\bibitem{N13b}
Katsuyuki Naoi.
\newblock Demazure crystals and tensor products of perfect
  {K}irillov-{R}eshetikhin crystals with various levels.
\newblock {\em J. Algebra}, 374:1--26, 2013.

\bibitem{N13a}
Katsuyuki Naoi.
\newblock Demazure modules and graded limits of minimal affinizations.
\newblock {\em Represent. Theory}, 17:524--556, 2013.

\bibitem{Ne25a}
Andrei Negu$\text{\c{t}}$.
\newblock Category $\mathcal{O}$ for quantum loop algebras.
\newblock arXiv:2501.00724.

\bibitem{NSS}
Erhard Neher, Alistair Savage, and Prasad Senesi.
\newblock Irreducible finite-dimensional representations of equivariant map
  algebras.
\newblock {\em Trans. Amer. Math. Soc.}, 364(5):2619--2646, 2012.

\bibitem{Nou95a}
Masatoshi Noumi and Tetsuya Sugitani.
\newblock Quantum symmetric spaces and related {$q$}-orthogonal polynomials.
\newblock In {\em Group theoretical methods in physics ({T}oyonaka, 1994)},
  pages 28--40. World Sci. Publ., River Edge, NJ, 1995.

\bibitem{Pa95a}
Paolo Papi.
\newblock Convex orderings in affine root systems.
\newblock {\em J. Algebra}, 172(3):613--623, 1995.

\end{thebibliography}

\end{document}